\documentclass[oneside,english]{amsart}
\usepackage[T1]{fontenc}
\usepackage[latin9]{inputenc}
\usepackage{geometry}
\usepackage{babel}
\usepackage{amstext}
\usepackage{amsthm}
\usepackage{amssymb}
\usepackage{graphicx}
\usepackage[unicode=true,pdfusetitle,
 bookmarks=true,bookmarksnumbered=false,bookmarksopen=false,
 breaklinks=false,pdfborder={0 0 1},backref=false,colorlinks=false]
 {hyperref}

\makeatletter
\numberwithin{equation}{section}
\numberwithin{figure}{section}
\theoremstyle{plain}
\newtheorem{thm}{\protect\theoremname}
\theoremstyle{plain}
\newtheorem{cor}[thm]{\protect\corollaryname}
\theoremstyle{definition}
\newtheorem{defn}[thm]{\protect\definitionname}
\theoremstyle{plain}
\newtheorem{prop}[thm]{\protect\propositionname}
\theoremstyle{remark}
\newtheorem{rem}[thm]{\protect\remarkname}
\theoremstyle{plain}
\newtheorem{lem}[thm]{\protect\lemmaname}

\makeatother

\providecommand{\corollaryname}{Corollary}
\providecommand{\definitionname}{Definition}
\providecommand{\lemmaname}{Lemma}
\providecommand{\propositionname}{Proposition}
\providecommand{\remarkname}{Remark}
\providecommand{\theoremname}{Theorem}

\begin{document}
\title{Massive Scaling Limit of the Ising Model: Subcritical Analysis and
Isomonodromy}
\author{S. C. Park}
\address{Ecole Polytechnique Fédérale de Lausanne, EPFL SB MATHAA CSFT, CH-1015
Lausanne, Switzerland}
\email{\textsf{sungchul.park@epfl.ch}}
\begin{abstract}
We study the spin $n$-point functions of the planar Ising model on
a simply connected domain $\Omega$ discretised by the square lattice
$\delta\mathbb{Z}^{2}$ under near-critical scaling limit. While the
scaling limit on the full-plane $\mathbb{C}$ has been analysed in
terms of a fermionic field theory, the limit in general $\Omega$
has not been studied. We will show that, in a massive scaling limit
wherein the inverse temperature is scaled $\beta\sim\beta_{c}-m_{0}\delta$
for a constant $m_{0}<0$, the renormalised spin correlations converge
to a continuous quantity determined by a boundary value problem set
in $\Omega$. In the case of $\Omega=\mathbb{C}$ and $n=2$, this
result reproduces the celebrated formula of \cite{wmtb} involving
the Painlevé III transcendent. To this end, we generalise the comprehensive
discrete complex analytic framework used in the critical setting to
the massive setting, which results in a perturbation of the usual
notions of analyticity and harmonicity.
\end{abstract}

\maketitle
\tableofcontents{}

\section{Introduction}

The Ising model is a classical model of ferromagnetism first introduced
by Lenz \cite{lenz}, whose simplicity and rich emergent structure
have allowed for applications in various areas of science. In two
dimensions, tt famously exhibits a continuous phase transition \cite{onsager,yang},
where the characteristic length of the model diverges and the model
becomes scale invariant. Consequently the model at the critical temperature
$\beta_{c}$ is expected to exhibit conformal symmetry under scaling
limit, a prediction which has been formalised in terms of the Conformal
Field Theory \cite{BPZ}.

Given the infinite dimensional array of 2D local conformal symmetry
\cite{di-francesco-mathieu-senechal}, it is natural to study the
scaling limit of the model not only on the full plane $\mathbb{C}$
but on an arbitrary simply connected domain $\Omega$. Accordingly,
recent research has focused on giving a rigorous description of the
interaction between various physical quantities of the model under
critical scaling limit and the conformal geometry of $\Omega$, which
results in explicit formulae for the limit of spin correlations at
the microscopic scale \cite{hosm2013,ghp}, at the macroscopic scale
\cite{chelkak-hongler}, or in some mixture of the two \cite{hon2010,chi18}
in terms of quantities such as the conformal radius and its derivatives.
Convergence of correlations also proves to be useful in proving convergence
in more general senses: \cite{camia15} proves that the discrete spin
field converges to a continuous random distribution using the convergence
of their correlations.

Central to such analyses of the critical regime are discrete fermion
correlations, which manifest themselves as discrete functions capable
of encoding relevant physical quantities. They are discrete counterparts
of the massless free fermion correlations in the continuous CFT, which
turn out to be explicit holomorphic functions thanks to conformal
symmetry. Discrete fermions instead enjoy a strong notion of discrete
analyticity \cite{smirnov-ii}, which, unlike some of its weaker variants,
readily lends itself to precompactness estimates that ultimately yield
convergence to the continuous fermion.

However, the free fermion is not an object unique to the massless
(conformal) field theory; indeed, the general field theory of the
free fermion specifies a mass parameter $m$, or equivalently a length
scale $\xi\propto1/\left|m\right|$. In general, the corresponding
regime in the Ising model is the \emph{near-critical} scaling limit,
where the deviation from criticality $\beta_{c}-\beta$ scales proportionally
to the lattice spacing $\delta$. Such scaling keeps the physical
correlation length $\xi$ asymptotically constant, allowing the limit
to be physically described by the massive fermion.

The discrete fermion survives in this near-critical setup, and discrete
analyticity persists albeit in a perturbed sense \cite{dgp,hkz,dt}.
While the continuous massive fermion is usually described by the two-dimensional
massive Dirac equation, our strong discrete analyticity in fact features
twice as many relations, resulting in (the discrete counterpart of)
a perturbation of the ordinary Cauchy-Riemann equations in 1D. Since
there are as many lattice equations in the near-critical limit as
in the critical limit, it is natural to attempt to carry out in the
former the analogues of analyses from the latter. We note here that
in addition to such a \emph{thermal} perturbation, one may also consider
a \emph{magnetic} perturbation to introduce mass (e.g. \cite{camia16}).
Another direction of research has recently focused on universality
with respect to general lattice, see \cite{Che18}.

In this paper, we undertake the analysis of macroscopic Ising spin
correlations on a simply connected domain $\Omega$ in the near-critical
scaling limit where $\beta_{c}-\beta$ is held equal to $m_{0}\delta$
for a fixed $m_{0}<0$ with $+$ boundary conditions. We establish
the existence of scaling functions to which renormalised spin correlations
converge, and show that their logarithmic derivatives are determined
by an explicit boundary value problem set in $\Omega$. This extends
the results of \cite{chelkak-hongler} to the massive regime (save
for the conformal covariance, which should not hold), and our proof
combines the strategies of that paper with a massive perturbation
of analytic function theory, both in the discrete and the continuous
settings. In the former, massive harmonic and holomorphic functions
can be studied via their relation to massive (extinguished) random
walk; in the latter, the perturbed Cauchy-Riemann equation is dubbed
Vekua equation and extensively treated in a theory established by
Carleman, Bers, and Vekua, among others \cite{bers,vek}.

In the full plane, the massive scaling limit of the spin correlations
was revealed to exhibit a surprising integrability property. Wu, McCoy,
Tracy, and Barouch \cite{wmtb} first demonstrated that the $2$-point
function on the plane can be described in terms of the Painlevé III
transcendent. Subsequently, Sato, Miwa, and Jimbo \cite{sato-miwa-jimbo}
recast the continuous analysis in terms of isomonodromic deformation
theory, where Painlevé equations are known to arise, and obtained
a closed set of differential equations for the $n$-point function.
Letting $\Omega=\mathbb{C}$, we reproduce the $2$-point scaling
limit in the case of full-plane (whose classical treatment is given
in, e.g., \cite{patr,palmer}), setting up the continuous analysis.
We explicitly carry out the isomonodromic analysis following the formulation
of \cite{kako80}.

\subsection{Main Results}

Let $\Omega$ be a bounded simply connected domain with smooth boundary.
We will treat the unbounded cases $\Omega=\mathbb{C},\mathbb{H}$
as well. Define the rotated square lattice $\Omega_{\delta}:=\Omega\cap\delta(1+i)\mathbb{Z}^{2}=\Omega\cap\mathbb{C}_{\delta}$.
We define the \emph{Ising probability measure} $\mathbb{P}=\mathbb{P}_{\Omega_{\delta},\beta}^{+}$
with $+$ boundary conditions at \emph{inverse temperature} $\beta>0$
on the space of \emph{spin configurations} $\left\{ \pm1\right\} ^{\Omega_{\delta}}$
by
\[
\mathbb{P}_{\Omega_{\delta},\beta}^{+}\left[\sigma:\Omega_{\delta}\to\left\{ \pm1\right\} \right]\propto\exp\sum_{i\sim j}\beta\sigma_{i}\sigma_{j},
\]
where the sum is over pairs $\left\{ i,j\right\} \subset\mathbb{C}_{\delta},i\in\Omega_{\delta}$
such that $\left|i-j\right|=\sqrt{2}\delta$ and we define $\sigma_{j}=1$
for $j\notin\Omega_{\delta}$ ($j\in\mathcal{F}\left[\Omega_{\delta}\right]$
in terms of detailed notation in Section \ref{subsec:Notation}).
If $a\in\Omega$, we understand by $\sigma_{a}$ the spin at a closest
point in $\Omega_{\delta}$ to $a$.

The planar Ising model on the square lattice undergoes a phase transition
at the critical temperature $\beta_{c}=\frac{1}{2}\ln\left(1+\sqrt{2}\right)$.
Henceforth we will fix a negative parameter $m$ and set $\beta=\beta(\delta)=\beta_{c}-\frac{m\delta}{2}$.
This is a \emph{subcritical} massive limit, where the spins stay in
the ordered phase while approaching criticality.
\begin{thm}
\label{thm:1}Let $\Omega$ be a bounded simply connected domain and
suppose $a_{1},\ldots,a_{n}\in\Omega$. Under $\delta\downarrow0,\beta=\beta_{c}-\frac{m\delta}{2}$,
the spin $n$-point function converges to a continuous function of
$a_{1},\ldots,a_{n}$,
\[
\delta^{-\frac{n}{8}}\mathbb{E}_{\Omega_{\delta},\beta(\delta)}^{+}\left[\sigma_{a_{1}}\cdots\sigma_{a_{n}}\right]\to\left\langle a_{1},\ldots,a_{n}\right\rangle _{\Omega,m}^{+},
\]
 and its logarithmic derivative $\partial_{a_{1}}\ln\left\langle a_{1},\ldots,a_{n}\right\rangle _{\Omega,m}^{+}=\mathcal{A}_{\Omega}^{1}+i\mathcal{A}_{\Omega}^{i}$,
where $\partial_{a_{1}}=\frac{1}{2}\left(\partial_{x_{1}}-i\partial_{y_{1}}\right)$,
is determined by the solution to the boundary value problem of Proposition
\ref{prop:contbvp} set on the domain $\Omega$. The functions $\left\langle a_{1},\ldots,a_{n}\right\rangle _{\Omega,m}^{+}$
are uniquely determined by their diagonal and boundary behaviours,
see Section \ref{subsec:intro-proof}. In particular, as $a_{1}\to\partial\Omega$,
\[
\left\langle a_{1}\right\rangle _{\Omega,m}^{+}\sim\left\langle a_{1}\right\rangle _{\Omega,0}^{+}:=\mathcal{C}\cdot2^{\frac{1}{4}}\text{rad}^{-\frac{1}{8}}(a_{1},\Omega),
\]
where $\mathcal{C}:=2^{\frac{1}{6}}e^{-\frac{3}{2}\zeta'(-1)}$, $\text{rad}(a_{1},\Omega)$
is the conformal radius of $\Omega$ as seen from $a_{1}$, and $\zeta$
is the Riemann zeta function.
\end{thm}

If $\phi:\mathbb{D}\to\Omega$ is a conformal map with $\phi(0)=a_{1}$,
then $\text{rad}(a_{1},\Omega)=\left|\phi'(a_{1})\right|$. The normalisation
of our continuous functions $\left\langle \cdot\right\rangle _{\Omega,0}^{+}$
differ to that of \cite{chelkak-hongler} by a factor of $\mathcal{C}^{n}$.
The derivation of the following result in our analytical setup may
be of independent interest.
\begin{cor}[\cite{wmtb,sato-miwa-jimbo,kako80}]
\label{cor:2}The $2$-point function in the full-plane is given
by
\begin{alignat*}{1}
\left\langle -a,a\right\rangle _{\mathbb{C},m}^{+} & =cst\cdot\cosh h_{0}(am)\cdot\exp\left[\int_{-\infty}^{am}r\left[\left(h_{0}'(r)\right)^{2}-4\sinh^{2}2h_{0}(r)\right]dr\right],\\
\left\langle -a,a\right\rangle _{\mathbb{C},-m}^{\text{free}} & =cst\cdot\sinh h_{0}(am)\cdot\exp\left[\int_{-\infty}^{am}r\left[\left(h_{0}'(r)\right)^{2}-4\sinh^{2}2h_{0}(r)\right]dr\right],
\end{alignat*}
where $a>0$ and $\eta_{0}=-\frac{1}{2}\ln h_{0}$ is a solution to
the Painlevé III equation
\[
r\eta_{0}\eta''_{0}=r\left(\eta_{0}'\right)^{2}-\eta_{0}\eta_{0}'-4r+4r\eta_{0}^{4}.
\]

The constants are fixed by the condition that as $a\to0$
\[
\left\langle -a,a\right\rangle _{\mathbb{C},m}^{+}\sim\left\langle -a,a\right\rangle _{\mathbb{C},0}^{+}:=\frac{\mathcal{C}^{2}}{\left|2a\right|^{1/4}},
\]
and $\mathcal{C}:=2^{\frac{1}{6}}e^{-\frac{3}{2}\zeta'(-1)}$ as above.
\end{cor}

\subsection{Notation\label{subsec:Notation}}

Following signs will be used throughout the paper:
\[
\lambda:=e^{\frac{i\pi}{4}},\beta_{c}=\frac{1}{2}\ln\left(1+\sqrt{2}\right),\mathbb{H}:=\left\{ z\in\mathbb{C}:\text{Im}z>0\right\} ,A\oplus B:=\left(A\cup B\right)\setminus\left(A\cap B\right).
\]

We will write partial derivatives in contracted form, i.e. $\partial_{x}=\frac{\partial}{\partial x}$,
etc. And denote by $\partial_{z},\partial_{\bar{z}}$ the Wirtinger
derivatives: where $z=x+iy$,
\[
\partial_{z}:=\frac{\partial_{x}-i\partial_{y}}{2},\partial_{\bar{z}}:=\frac{\partial_{x}+i\partial_{y}}{2}=\frac{e^{i\theta}\left(\partial_{r}+ir^{-1}\partial_{\theta}\right)}{2}.
\]

We will also use $\partial$ to denote directional derivatives; i.e.
$\partial_{1}=\partial_{x}$, $\partial_{i}=\partial_{y}$, and so
on. If $z\in\partial\Omega$, denote by $\nu_{\text{out}}(z)\in\mathbb{C}$
as the unit normal at $x$, i.e. the unit complex number which points
to the direction of outer normal vector at $z$. Then $\partial_{\nu_{\text{out}}}$
is the outer normal derivative in the direction of $\nu_{\text{out}}$.

We denote by $\partial_{\lambda}^{\delta},\Delta^{\delta}$ the following
discrete operators:
\begin{alignat*}{1}
\partial_{\lambda}^{\delta}f(z) & :=f(z+\sqrt{2}\lambda\delta)-f(z)\text{, etc.,}\\
\Delta^{\delta}f(z) & :=f(z+\sqrt{2}\lambda\delta)+f(z-\sqrt{2}\bar{\lambda}\delta)+f(z-\sqrt{2}\lambda\delta)+f(z+\sqrt{2}\bar{\lambda}\delta)-4f(z),
\end{alignat*}
wherever they make sense, if $z\notin\mathcal{V}\left[\Omega_{\delta}\right]$.
On $\mathcal{V}\left[\Omega_{\delta}\right]$, we make a small modification
in the coefficients in $\Delta^{\delta}$; see (\ref{eq:lap_vert}).
Note that $(\sqrt{2}\delta)^{-2}\Delta^{\delta}\to\Delta$.

\subsubsection*{Mass Parametrisation.}

There are various equivalent ways of parametrising the deviation $\beta_{c}-\beta$,
and we summarise the relation amongst them here at once, which hold
at all times. In this paper, $M,m,\Theta$ will be supposed to be
\textbf{negative} unless otherwise specified. We also assume $\delta$
is small enough to, e.g., have $\beta\in(0,\infty)$.
\begin{itemize}
\item \emph{Discrete mass} $M:=\beta_{c}-\beta$ is scaled $M=\frac{m\delta}{2}$
with the continuous mass $m$ being a constant.
\item Pure phase factor $e^{2i\Theta}:=\lambda^{-3}\frac{e^{-2\beta}+i}{e^{-2\beta}-i}$
with $\Theta\in\left]-\frac{\pi}{8},\frac{\pi}{8}\right[$. Equivalently
$e^{2\beta}=\cot\left(\frac{\pi}{8}+\Theta\right)$. $\Theta$ is
scaled $\Theta\sim\frac{m\delta}{2}$. Also define $M_{H}:=2\sin2\Theta\sqrt{\frac{2}{\cos4\Theta}}$
which is the mass coefficient in massive harmonicity.
\end{itemize}

\subsubsection*{Graph Notation.}

Recall that we work with the rotated square lattice $\mathbb{C}_{\delta}:=\delta(1+i)\mathbb{Z}^{2}$.
Our graph $\Omega_{\delta}$ comprises the following components (see
Figure \ref{fig:The-square-lattice.}):
\begin{itemize}
\item \emph{faces} $\mathcal{F}\left[\Omega_{\delta}\right]:=\Omega\cap\delta(1+i)\mathbb{Z}^{2}$,
\item \emph{vertices} $\mathcal{V}\left[\Omega_{\delta}\right]:=\left\{ f\pm\delta,f\pm i\delta:f\in\mathcal{F}\left[\Omega_{\delta}\right]\right\} \subset\mathbb{C}_{\delta}^{*}$,
\item \emph{edge} $\mathcal{E}\left[\Omega_{\delta}\right]:=\left\{ (ij)=(ji):i,j\in\mathcal{V}\left[\Omega_{\delta}\right],\left|i-j\right|=\sqrt{2}\delta\right\} $,
and
\item \emph{corners} $\mathcal{C}\left[\Omega_{\delta}\right]:=\left\{ (vf):v\in\mathcal{V}\left[\Omega_{\delta}\right],f\in\mathcal{F}\left[\Omega_{\delta}\right],\left|v-f\right|=\delta\right\} $.
\end{itemize}
For consistency with the low-temperature expansion of the model, we
prefer to visualise the lattice in its dual form. Note that just as
the faces are represented by their midpoints above, an edge $(ij)$
and a corner $(vf)$ will be identified with their midpoints $\frac{i+j}{2}$
and $\frac{v+f}{2}$, respectively. Additionally, we draw a \emph{half-edge}
between either an edge midpoint or a corner to a nearest vertex.

For $\tau=1,i,\lambda,\bar{\lambda}$, a corner $c\in\mathcal{C}\left[\Omega_{\delta}\right]$
is in $\mathcal{C}^{\tau}\left[\Omega_{\delta}\right]$ if the nearest
vertex is in the direction $-\tau^{-2}$. The edges in $\mathcal{C}^{1}\left[\Omega_{\delta}\right],\mathcal{C}^{i}\left[\Omega_{\delta}\right]$
are respectively called \emph{real} and \emph{imaginary corners}.

We will frequently denote union of various sets by concatenation,
e.g. $\mathcal{EC}\left[\Omega_{\delta}\right]:=\mathcal{E}\left[\Omega_{\delta}\right]\cup\mathcal{C}\left[\Omega_{\delta}\right]$.

\subsubsection*{Graph Boundary.}
\begin{itemize}
\item $\overline{\mathcal{E}}\left[\Omega_{\delta}\right]:=\left\{ (ij)=(ji):i\in\mathcal{V}\left[\Omega_{\delta}\right],j\in\mathcal{V}\left[\mathbb{C}_{\delta}\right],|i-j|=\sqrt{2}\delta\right\} $,
$\partial\mathcal{E}\left[\Omega_{\delta}\right]:=\overline{\mathcal{E}}\left[\Omega_{\delta}\right]\setminus\mathcal{E}\left[\Omega_{\delta}\right]$,
\item boundary vertices $\partial\mathcal{V}\left[\Omega_{\delta}\right]$
are the endpoints of edges in $\partial\mathcal{E}\left[\Omega_{\delta}\right]$
not in $\mathcal{V}\left[\Omega_{\delta}\right]$, 
\item boundary faces $\partial\mathcal{F}\left[\Omega_{\delta}\right]$
are faces in $\mathcal{F}\left[\mathbb{C}_{\delta}\right]\setminus\mathcal{F}\left[\Omega_{\delta}\right]$
which are $\delta$ away from a vertex in $\mathcal{V}\left[\Omega_{\delta}\right]$,
and
\item boundary corners $\partial\mathcal{C}\left[\Omega_{\delta}\right]=\left\{ (vf):v\in\mathcal{V}\left[\Omega_{\delta}\right],f\in\partial\mathcal{F}\left[\Omega_{\delta}\right],\left|v-f\right|=\delta\right\} $ 
\item $\nu_{\text{out}}(z)$ for $z\in\partial\mathcal{E}\left[\Omega_{\delta}\right]$
is the unit complex number corresponding to the orientation of $z$
pointing outwards from $\Omega$.
\end{itemize}

\subsubsection*{Double Cover.}

The fermions we introduce in forthcoming sections are functions defined
on the double cover of the continuous and discrete domains $\Omega,\Omega_{\delta}$.
Define $\left[\Omega,a_{1},\ldots,a_{n}\right]$ as the double cover
of $\Omega$ ramified at distinct interior points $a_{1},\ldots,a_{n}\in\Omega$;
in particular, it is a Riemann surface where $\sqrt{\left(z-a_{1}\right)\cdots\left(z-a_{n}\right)}$
is well-defined, smooth and single-valued. In the case where $\left\{ \bar{a_{1}},\ldots,\bar{a_{n}}\right\} =\left\{ a_{1},\ldots,a_{n}\right\} $,
conjugation on the double cover is defined by requiring that $\sqrt{\left(\bar{z}-a_{1}\right)\cdots\left(\bar{z}-a_{n}\right)}=\overline{\sqrt{\left(z-a_{1}\right)\cdots\left(z-a_{n}\right)}}$.
On $\left[\mathbb{C},a_{1}\right]$, we will refer to the slit domains
$\mathbb{X}^{+}:=\left\{ \text{Re}\sqrt{z}>0\right\} $ and $\mathbb{Y}^{+}:=\left\{ \text{Im}\sqrt{z}>0\right\} $.

When the choice of the lift of $z\in\Omega$ is clear, we will write
$z$ to denote the lift in $\left[\Omega,a_{1},\ldots,a_{n}\right]$.
Conversely, if $z\in\left[\Omega,a_{1},\ldots,a_{n}\right]$, we will
write $z$, or, for clarification, $\pi(z)\in\Omega\setminus\left\{ a_{1},\ldots,a_{n}\right\} $
for the projection onto the planar domain; $z^{\cdot}\in\left[\Omega,a_{1},\ldots,a_{n}\right]$
is the lift of $\pi(z)$ which is not $z$. A function which switches
sign under switching $z$ and $z^{\cdot}$ is called a \emph{spinor}.

We say that two points $z,w\in\left[\Omega,a_{1},\ldots,a_{n}\right]$
are \emph{on the same sheet} if we can draw a straight line segment
between them; i.e. the straight line segment on $\Omega$ which connects
$\pi(z),\pi(w)$ can be lifted to connect $z,w$ on $\left[\Omega,a_{1},\ldots,a_{n}\right]$.

For the discrete double cover $\left[\Omega_{\delta},a_{1},\ldots,a_{n}\right]$,
we will take closest faces in $\Omega_{\delta}$ to $a_{1},\ldots,a_{n}\in\Omega$,
and then lift components of $\Omega_{\delta}$ minus those $n$ faces.
Clearly, $\left[\Omega_{\delta},a_{1},\ldots,a_{n}\right]$ is a lattice
which is locally isomorphic to the planar lattice $\Omega_{\delta}$.
Given the first monodromy face $a_{1}$, we will fix a lift of $a_{1}+\frac{\delta}{2}$
and refer to it throughout this paper.

\begin{figure}
\includegraphics[width=0.9\columnwidth]{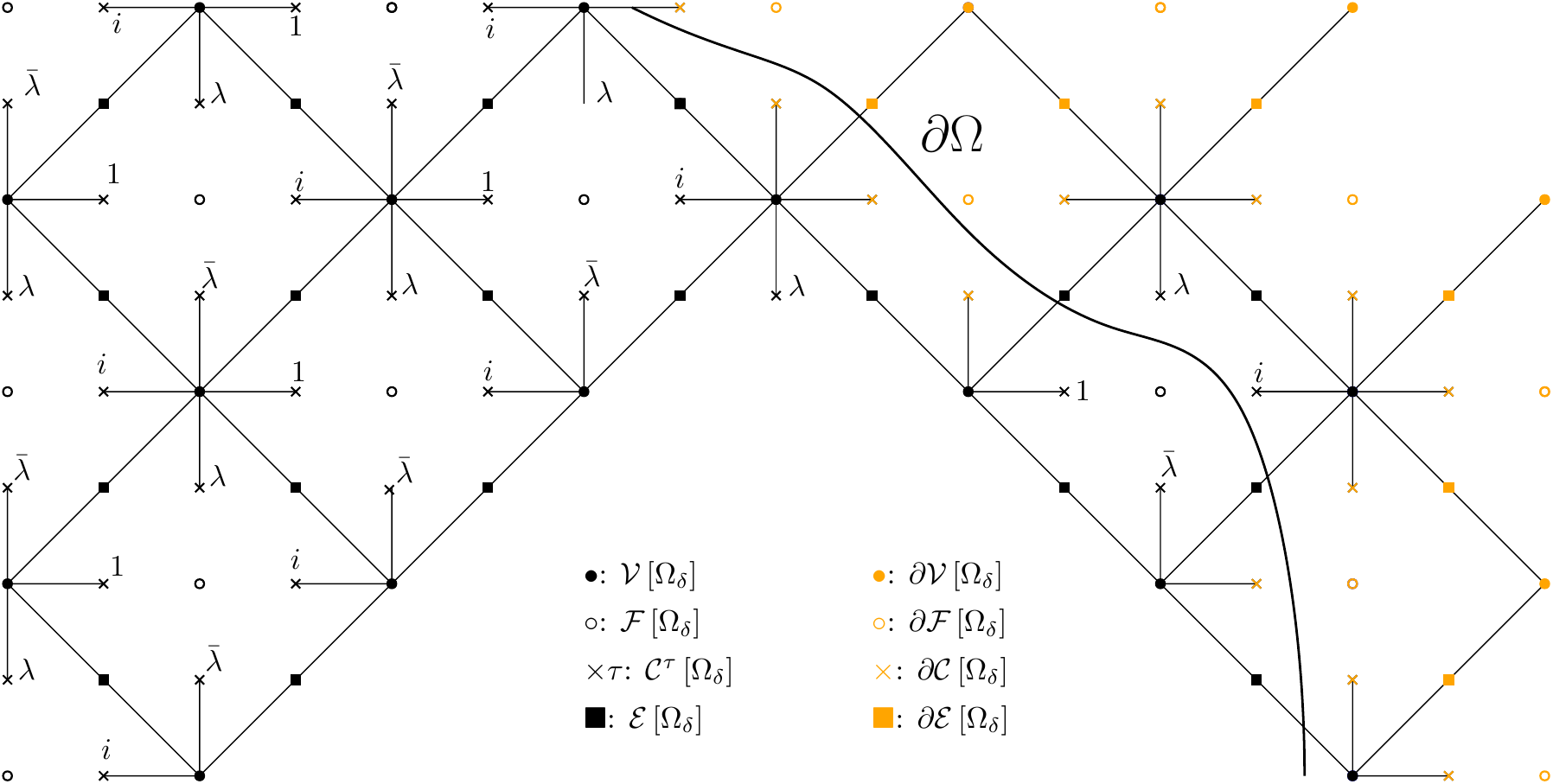}

\caption{\label{fig:The-square-lattice.}The square lattice. Proposition \ref{prop:ident_limit}:
$F_{\left[\Omega_{\delta},a_{1},\ldots,a_{n}\right]}=O(1)$ on boundary
edges and corners in orange, since $H_{\left[\Omega_{\delta},a_{1},\ldots,a_{n}\right]}^{\bullet}=O(\delta)$
on adjacent vertices in $\mathcal{V}\left[\Omega_{\delta},a_{1},\ldots,a_{n}\right]$.}
\end{figure}

\subsection{Proof of the Main Theorems\label{subsec:intro-proof}}
\begin{proof}[Proof of Theorem \ref{thm:1}.]
 Given the fermion convergence of Theorem \ref{thm:convergence}
and identification of Ising quantities in terms of the fermions of
Proposition \ref{prop:isingcor}, we can integrate the discrete logarithmic
derivative which converges in the scaling limit (see also \cite[Proposition 2.22, Remark 2.23]{chelkak-hongler})
\begin{alignat*}{1}
\frac{1}{2\delta}\left[\frac{\mathbb{E}_{\Omega_{\delta}}^{\beta,+}\left[\sigma_{a_{1}+2\delta}\sigma_{a_{2}}\cdots\sigma_{a_{n}}\right]}{\mathbb{E}_{\Omega_{\delta}}^{\beta,+}\left[\sigma_{a_{1}}\cdots\sigma_{a_{n}}\right]}-1\right] & \to\mathcal{A}_{\Omega}^{1}(a_{1},\ldots,a_{n}),\\
\frac{1}{2\delta}\left[\frac{\mathbb{E}_{\Omega_{\delta}}^{\beta,+}\left[\sigma_{a_{1}+2i\delta}\sigma_{a_{2}}\cdots\sigma_{a_{n}}\right]}{\mathbb{E}_{\Omega_{\delta}}^{\beta,+}\left[\sigma_{a_{1}}\cdots\sigma_{a_{n}}\right]}-1\right] & \to-\mathcal{A}_{\Omega}^{i}(a_{1},\ldots,a_{n}),
\end{alignat*}
 to get that $\frac{\mathbb{E}_{\Omega_{\delta}}^{\beta,+}\left[\sigma_{b_{1}}\cdots\sigma_{b_{n}}\right]}{\mathbb{E}_{\Omega_{\delta}}^{\beta,+}\left[\sigma_{a_{1}}\cdots\sigma_{a_{n}}\right]}$
scales to a continuous limit $\frac{\left\langle b_{1},\ldots,b_{n}\right\rangle _{\Omega,m}^{+}}{\left\langle a_{1},\ldots,a_{n}\right\rangle _{\Omega,m}^{+}}$
for $a_{1},\ldots,a_{n},b_{1},\ldots,b_{n}\in\Omega$. The convergence
for $\mathcal{A}_{\Omega}^{i}$ follows by considering the result
in a $-90^{\circ}$ rotated domain (see also \cite[Proof of Theorems 1.5 and 1.7]{chelkak-hongler}).

Now it remains to uniquely relate the massive convergence rate to
the massless convergence rate: $\frac{\mathbb{E}_{\Omega_{\delta}}^{\beta,+}\left[\sigma_{a_{1}}\cdots\sigma_{a_{n}}\right]}{\mathbb{E}_{\Omega_{\delta}}^{\beta_{c},+}\left[\sigma_{a_{1}}\cdots\sigma_{a_{n}}\right]}\to\frac{\left\langle a_{1},\ldots,a_{n}\right\rangle _{\Omega,m}^{+}}{\left\langle a_{1},\ldots,a_{n}\right\rangle _{\Omega,0}^{+}}$
for some $a_{1},\ldots,a_{n}\in\Omega$. Given the convergence of
$\delta^{-\frac{n}{8}}\mathbb{E}_{\Omega_{\delta}}^{\beta_{c},+}\left[\sigma_{a_{1}}\cdots\sigma_{a_{n}}\right]$
to a continuous limit $\left\langle a_{1},\ldots,a_{n}\right\rangle _{\Omega,0}^{+}$,
we have convergence of the massive correlation to unique $\left\langle a_{1},\ldots,a_{n}\right\rangle _{\Omega,m}^{+}$.

The procedure partially relies on the process in the massless case
of relating the bounded domain correlations to full-plane correlations
from \cite{chelkak-hongler}.

Note that $\beta>\beta_{c}$, and denote the dual temperature by $\beta^{*}<\beta_{c}$.
We always assume a scaling $\beta_{c}-\beta=\frac{m\delta}{2}$ for
$m<0$.
\begin{enumerate}
\item Relating two point functions: $\frac{\mathbb{E}_{\Omega_{\delta}}^{\beta,+}\left[\sigma_{a_{1}}\sigma_{a_{2}}\right]}{\mathbb{E}_{\Omega_{\delta}}^{\beta_{c},+}\left[\sigma_{a_{1}}\sigma_{a_{2}}\right]}$
tends to a continuous limit, since
\begin{equation}
1\leq\frac{\mathbb{E}_{\Omega_{\delta}}^{\beta,+}\left[\sigma_{a_{1}}\sigma_{a_{2}}\right]}{\mathbb{E}_{\Omega_{\delta}}^{\beta_{c},+}\left[\sigma_{a_{1}}\sigma_{a_{2}}\right]}\leq\frac{\mathbb{E}_{\Omega_{\delta}}^{\beta,+}\left[\sigma_{a_{1}}\sigma_{a_{2}}\right]}{\mathbb{E}_{\Omega_{\delta}}^{\beta^{*},+}\left[\sigma_{a_{1}}\sigma_{a_{2}}\right]}\leq\frac{\mathbb{E}_{\Omega_{\delta}}^{\beta,+}\left[\sigma_{a_{1}}\sigma_{a_{2}}\right]}{\mathbb{E}_{\Omega_{\delta}}^{\beta^{*},\text{free}}\left[\sigma_{a_{1}}\sigma_{a_{2}}\right]}\rightarrow\left|\mathcal{B}_{\Omega}(a_{1},a_{2}|m)\right|^{-1},\label{eq:Bto1}
\end{equation}
where we successively used the monotonicity of spin correlation in
inverse temperature (e.g. by coupling with FK-Ising) and in boundary
condition (FKG inequality, \cite[Theorem 3.21]{friedli-velenik}).
We also used the convergence of Theorem \ref{thm:convergence} of
the Ising ratio to $\left|\mathcal{B}_{\Omega}(a_{1},a_{2}|m)\right|$.
We conclude by noting that $\left|\mathcal{B}_{\Omega}(a_{1},a_{2}|m)\right|$
can be made arbitrarily close to $1$ (Lemma \ref{lem:B}) by merging
$a_{1},a_{2}$.
\item One point functions: note that by above, we may write
\begin{alignat}{1}
\frac{\left\langle a_{1},a_{2}\right\rangle _{\Omega,m}^{+}}{\left\langle a_{1},a_{2}\right\rangle _{\Omega,0}^{+}} & =\exp\left[\int_{a_{2}}^{a_{1}}\left[\mathcal{A}_{\Omega}^{1}(z,a_{2}|m)-\mathcal{A}_{\Omega}^{1}(z,a_{2}|0)\right]dx\right.\nonumber \\
 & -\left.\left[\mathcal{A}_{\Omega}^{i}(z,a_{2}|m)-\mathcal{A}_{\Omega}^{i}(z,a_{2}|0)\right]dy\right],\label{eq:comparability-massive}
\end{alignat}
along any line from $a_{2}$ to $a_{1}$ if the integral converges.
Choose $a_{1}$ and $a_{2}$ as in Lemma \ref{lem:two-point-delicate}
(see Figure \ref{fig:integration}), then we can bound the integral:
$\frac{\left\langle a_{1},a_{2}\right\rangle _{\Omega,m}^{+}}{\left\langle a_{1},a_{2}\right\rangle _{\Omega,0}^{+}}\leq e^{cst\cdot\left(\epsilon^{-\gamma}\left|a_{1}-a_{2}\right|+\left|a_{1}-a_{2}\right|^{1-\gamma}\right)}$
for some fixed $\gamma\in(0,1)$. Now choose $\left|a_{1}-a_{2}\right|=\epsilon^{\kappa}$
for some $\gamma<\kappa<1$. Then $\frac{\left\langle a_{1},a_{2}\right\rangle _{\Omega,m}^{+}}{\left\langle a_{1},a_{2}\right\rangle _{\Omega,0}^{+}}\to1$
as $\epsilon\to0$, while the hyperbolic distance between $a_{1}$
and $a_{2}$ grows; indeed, the hyperbolic distance is comparable
to $\ln\frac{\left|a_{1}-a_{2}\right|}{\text{dist}(\left\{ a_{1},a_{2}\right\} ,\partial\Omega)}\propto(1-\kappa)\left|\ln\epsilon\right|$.
Therefore, by \cite[(1.3)]{chelkak-hongler}, $\frac{\left\langle a_{1},a_{2}\right\rangle _{\Omega,0}^{+}}{\left\langle a_{1}\right\rangle _{\Omega,0}^{+}\left\langle a_{2}\right\rangle _{\Omega,0}^{+}}\to1$.
Now we may relate the massive one-point functions to the massless
ones, so that the former are uniquely identified in terms of the latter.
By the monotonicity of spin correlation, $\mathbb{E}_{\Omega_{\delta}}^{\beta_{c},+}\left[\sigma_{a_{1}}\right]\leq\mathbb{E}_{\Omega_{\delta}}^{\beta,+}\left[\sigma_{a_{1}}\right]$,
and by GKS inequality \cite[Theorem 3.20]{friedli-velenik}, $\mathbb{E}_{\Omega_{\delta}}^{\beta,+}\left[\sigma_{a_{1}}\right]\mathbb{E}_{\Omega_{\delta}}^{\beta,+}\left[\sigma_{a_{2}}\right]\leq\mathbb{E}_{\Omega_{\delta}}^{\beta,+}\left[\sigma_{a_{1}}\sigma_{a_{2}}\right]$,
so
\begin{equation}
1\leq\frac{\mathbb{E}_{\Omega_{\delta}}^{\beta,+}\left[\sigma_{a_{1}}\right]}{\mathbb{E}_{\Omega_{\delta}}^{\beta_{c},+}\left[\sigma_{a_{1}}\right]}\leq\frac{\mathbb{E}_{\Omega_{\delta}}^{\beta,+}\left[\sigma_{a_{1}}\sigma_{a_{2}}\right]}{\mathbb{E}_{\Omega_{\delta}}^{\beta_{c},+}\left[\sigma_{a_{1}}\right]\mathbb{E}_{\Omega_{\delta}}^{\beta_{c},+}\left[\sigma_{a_{2}}\right]}=\frac{\mathbb{E}_{\Omega_{\delta}}^{\beta_{c},+}\left[\sigma_{a_{1}}\sigma_{a_{2}}\right]}{\mathbb{E}_{\Omega_{\delta}}^{\beta_{c},+}\left[\sigma_{a_{1}}\right]\mathbb{E}_{\Omega_{\delta}}^{\beta_{c},+}\left[\sigma_{a_{2}}\right]}\cdot\frac{\mathbb{E}_{\Omega_{\delta}}^{\beta,+}\left[\sigma_{a_{1}}\sigma_{a_{2}}\right]}{\mathbb{E}_{\Omega_{\delta}}^{\beta_{c},+}\left[\sigma_{a_{1}}\sigma_{a_{2}}\right]}.\label{eq:one-point-decorrelation}
\end{equation}
By the above discussion, both factors on the right is made arbitrarily
close to $1$ by taking small enough $\epsilon$. Moreover, such choice
of $\epsilon$ is uniform in the smoothness of the domain (concretely,
the bound on the derivative of a conformal map to the disc).
\item More points: we again use the GKS inequality, in that
\begin{equation}
1\leq\frac{\mathbb{E}_{\Omega_{\delta}}^{\beta,+}\left[\sigma_{a_{1}}\cdots\sigma_{a_{n}}\right]}{\mathbb{E}_{\Omega_{\delta}}^{\beta,+}\left[\sigma_{a_{1}}\right]\cdots\mathbb{E}_{\Omega_{\delta}}^{\beta,+}\left[\sigma_{a_{n}}\right]}\leq\frac{\mathbb{E}_{\Omega_{\delta}^{1}}^{\beta,+}\left[\sigma_{a_{1}}\right]\cdots\mathbb{E}_{\Omega_{\delta}^{n}}^{\beta,+}\left[\sigma_{a_{n}}\right]}{\mathbb{E}_{\Omega_{\delta}}^{\beta,+}\left[\sigma_{a_{1}}\right]\cdots\mathbb{E}_{\Omega_{\delta}}^{\beta,+}\left[\sigma_{a_{n}}\right]},\label{eq:gks}
\end{equation}
where $\Omega^{1}\ni a_{1},\ldots,\Omega^{n}\ni a_{n}$ are choices
of disjoint smooth simply connected subdomains of $\Omega$, each
sharing a macroscopic boundary arc with $\Omega$. By FKG inequality,
requiring the spins in $\Omega\setminus\bigcup_{j=1}^{n}\Omega^{j}$
to be plus raises the correlation, which gives the second inequality
above. By the uniform identification of the one point functions, there
is $\epsilon>0$ such that if each of $a_{1},\ldots,a_{n}$ are $\epsilon$-close
to the boundary arc that $\Omega^{1},\ldots,\Omega^{n}$ shares with
$\Omega$, both $\mathbb{E}_{\Omega_{\delta}^{j}}^{\beta,+}\left[\sigma_{a_{j}}\right],\mathbb{E}_{\Omega_{\delta}}^{\beta,+}\left[\sigma_{a_{j}}\right]$
become close to their massless counterparts. Since the massless correlations
are Caratheodory stable (\cite[(1.3)]{chelkak-hongler}), as $\epsilon$
decreases, $\mathbb{E}_{\Omega_{\delta}^{j}}^{\beta_{c},+}\left[\sigma_{a_{j}}\right]/\mathbb{E}_{\Omega_{\delta}}^{\beta_{c},+}\left[\sigma_{a_{j}}\right]$
become close to $1$. Clearly, any ratio between $\mathbb{E}_{\Omega_{\delta}^{j}}^{\beta,+}\left[\sigma_{a_{j}}\right],\mathbb{E}_{\Omega_{\delta}}^{\beta,+}\left[\sigma_{a_{j}}\right],\mathbb{E}_{\Omega_{\delta}^{j}}^{\beta_{c},+}\left[\sigma_{a_{j}}\right],\mathbb{E}_{\Omega_{\delta}}^{\beta_{c},+}\left[\sigma_{a_{j}}\right]$
may be made arbitrarily close to $1$ by setting $\epsilon$ small
enough, and thus $\frac{\mathbb{E}_{\Omega_{\delta}}^{\beta,+}\left[\sigma_{a_{1}}\cdots\sigma_{a_{n}}\right]}{\mathbb{E}_{\Omega_{\delta}}^{\beta,+}\left[\sigma_{a_{1}}\right]\cdots\mathbb{E}_{\Omega_{\delta}}^{\beta,+}\left[\sigma_{a_{n}}\right]}$
as well; this fixes the normalisation of an arbitrary $n$-point function.
\end{enumerate}
\end{proof}
\begin{figure}
\includegraphics[width=0.5\textwidth]{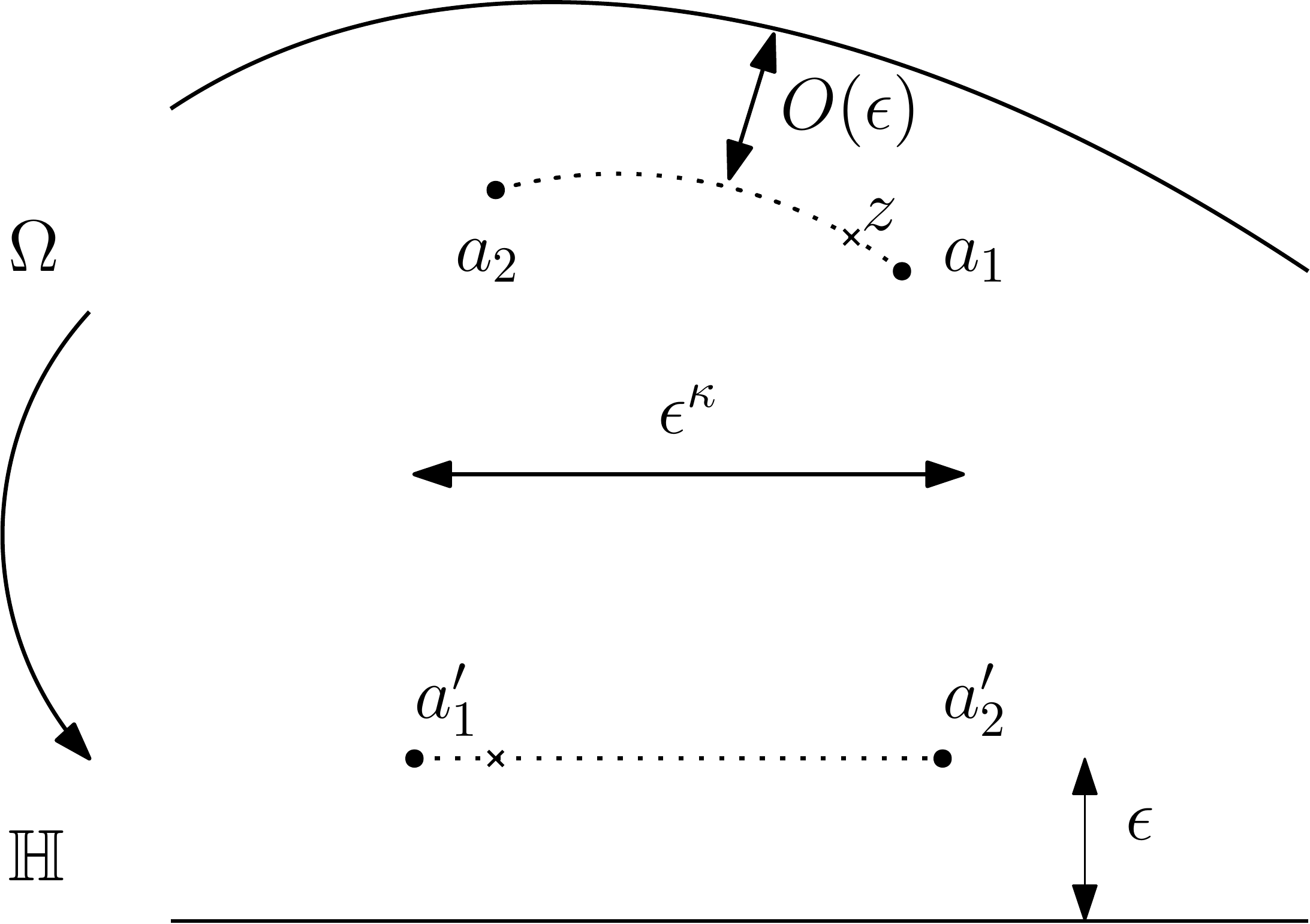}

\caption{Setup for applying Lemma \ref{lem:two-point-delicate} to have decorrelation
between $a_{1},a_{2}$.}
\label{fig:integration}

\end{figure}

\begin{proof}[Proof of Corollary \ref{cor:2}.]
 Note that the argument for the scaling limit of the two-point functions
in the proof of Theorem \ref{thm:1} apply in the full plane as well.
It also fixes the normalisation as $a\to0$ by \cite[Remark 2.26]{chelkak-hongler}.
From Theorem \ref{thm:convergence}, we need to integrate $\mathcal{A}_{\mathbb{C}}^{1}(-a,a)$.
Recall $r=am$, and define $\left\langle r\right\rangle _{\mathbb{C}}^{+}:=\left\langle -a,a\right\rangle _{\mathbb{C},m}^{+}$.

By (\ref{eq:painleve_A}), $-m^{-1}\left.\partial_{a}\ln\left\langle -a,b\right\rangle _{\mathbb{C},m}^{+}\right|_{b=a}=\left.-\frac{1}{2}\left(\ln\cosh h_{0}(r)\right)'-r\left[\frac{1}{2}\left(h_{0}'(r)\right)^{2}-2\sinh^{2}2h_{0}(r)\right]\right|_{r=am}$,
which can be rephrased as 
\[
\partial_{r}\ln\left\langle r\right\rangle _{\mathbb{C}}^{+}=m^{-1}\partial_{a}\left\langle -a,a\right\rangle _{\mathbb{C},m}^{+}=\left(\ln\cosh h_{0}(r)\right)'+r\left[\left(h_{0}'(r)\right)^{2}-4\sinh^{2}2h_{0}(r)\right].
\]
Then 
\[
\left\langle -a,a\right\rangle _{\mathbb{C},m}^{+}=\left\langle r\right\rangle _{\mathbb{C}}^{+}=cst\cdot\cosh h_{0}(am)\cdot\exp\left[\int_{-\infty}^{am}r\left[\left(h_{0}'(r)\right)^{2}-4\sinh^{2}2h_{0}(r)\right]dr\right].
\]
Then the definition $\tanh h_{0}:=\mathcal{B}_{0}$ gives the other
case.
\end{proof}

\subsection{Structure of the Paper}

This paper contains four sections and one appendix to which technical
calculations and estimates are deferred. Section \ref{sec:Massive-Discrete-Fermions}
defines our main analytical tool, the discrete fermions. The combinatorial
definition in \ref{subsec:Bounded-Domain-Fermions} involves contours
on the discrete bounded domain and is seen to naturally encode the
logarithmic derivative of the spin correlation. Its discrete complex
analytic properties are then established, which are exploited in Section
\ref{subsec:Infinite-Volume-Limit} to give a definition on the full
plane by an infinite volume limit.

Since analysis of the continuous fermions is needed for the scaling
limit process (for a unique characterisation of the continuous limit),
we carry out the continuum analysis first in Section \ref{subsec:Massive-Complex-Analysis:}.
We formulate the boundary value problem on $\Omega$ for our continuous
fermions, which will be a massive perturbation of holomorphic functions
treated in \cite{bers,vek}. We verify various properties we will
use: such as the expansion in terms of formal 'powers'. Analysis in
the continuum is continued in Section \ref{subsec:Analysis-of-the}.

Convergence of the discrete fermion under scaling limit is done in
Section \ref{sec:Discrete-Analysis:-Scaling}. The analysis is divided
into two parts: bulk convergence (Section \ref{subsec:convergence}),
where the discrete fermion evaluated on compact subsets of $\left[\Omega,a_{1},\ldots,a_{n}\right]$
is shown to uniformly converge to the continuous fermion, and analysis
near the singularity (Section \ref{subsec:Analysis-near-the}), where
the discrete fermion evaluated at a point in $\left[\Omega_{\delta},a_{1},\ldots,a_{n}\right]$
microscopically away from a monodromy face is identified from the
coefficients of a massive analytic version of power series expansion
of the continuous fermion. Bulk convergence is done in a standard
manner, by first showing that the set of discrete fermion correlations
are precompact and then uniquely identifying the limit. Analysis near
the singularity mainly uses ideas from \cite{chelkak-hongler}, where
the continuous power series expansion is modelled in the discrete
setting then the coefficients carefully matched.

In Section \ref{sec:Continuum-Analysis:-Painlev}, we collect the
analysis of the massive fermions necessary for the integration of
the logarithmic derivatives and isomonodromic analysis in Section
\ref{subsec:Analysis-of-the}, and finally in Section \ref{subsec:Derivation-of-Painlev}
we carry out the isomonodromic analysis and obtain the Painlevé III
transcendent, which can be identified in the logarithmic derivative
of spin correlations in $\mathbb{C}$ given the convergence results.

\subsection*{Acknowledgements}

The author would like to thank Clément Hongler for guidance, Konstantin
Izyurov, Kalle Kytölä, Julien Dubedat, Béatrice de Tilière, and Francesco
Spadaro for helpful discussions, and the ERC under grant SG CONSTAMIS
for financial support.

\section{Massive Fermions\label{sec:Massive-Discrete-Fermions}}

In this section, we introduce the main discrete tool of our analysis,
the massive discrete fermion correlations. While the object and the
terminology hearkens back to the physical analysis of the Ising model,
we shall give an explicit definition in Section \ref{subsec:Bounded-Domain-Fermions}
as a complex function which extrapolates the desired local physical
quantity to the entire domain. In the same subsection we also show
that, as a discrete function, it exhibits a notion in discrete complex
analysis called (massive) s-holomorphicity. At first we only define
the fermion in bounded discretised domains, i.e. finite sets; however
we will define them in the complex plane in Section \ref{subsec:Infinite-Volume-Limit}.
Then, we carry out analysis of the continuous spinors, to which the
discrete spinors presented in the previous section is shown to converge
in Section \ref{sec:Discrete-Analysis:-Scaling}. Since the proof
of scaling limit requires unique identification of the continuous
limit, we first give necessary analytic background in Section \ref{subsec:Massive-Complex-Analysis:}.

\subsection{Bounded Domain Fermions and Discrete Analysis\label{subsec:Bounded-Domain-Fermions}}

We introduce here the main object of our analysis, the discrete fermion
$F$. Note that this function is essentially the same object as in
\cite[Definition 2.1]{chelkak-hongler} albeit at general $\beta$,
and we try to keep the same normalisation and notation where appropriate.
The contents of this subsection are valid for any $\beta>0$.

In order to use the low-temperature expansion of the Ising model,
we first define $\Gamma_{\Omega_{\delta}}\subset2^{\mathcal{E}\left[\Omega_{\delta}\right]}$
as the collection of closed contours, i.e. set $\omega$ of edges
in $\Omega_{\delta}$ such that an even number of edges in $\omega$
meet at any given vertex. Given $+$ boundary condition, any $\omega$
is clearly in one-to-one correspondence with a spin configuration
$\sigma$ (where $\omega$ delineates clusters of identical spins),
and we can compute the partition function of the model and the correlation
\begin{alignat*}{1}
\mathcal{Z}_{\Omega_{\delta}}^{+,\beta} & :=\sum_{\omega\in\mathcal{C}_{\Omega_{\delta}}}e^{-2\beta|\omega|}\\
\mathbb{E}_{\Omega_{\delta}}^{+,\beta}\left[\sigma_{a_{1}}\cdots\sigma_{a_{n}}\right] & =\left|\mathcal{Z}_{\Omega_{\delta}}^{+,\beta}\right|^{-1}\sum_{\omega\in\Gamma_{\Omega_{\delta}}}e^{-2\beta|\omega|}(-1)^{\#\text{loops}_{a_{1},\ldots,a_{n}}\left(\omega\right)},
\end{alignat*}
where $\#\text{loops}_{a_{1},\ldots,a_{n}}$ denotes the parity of
loops in $\omega$ which separate the boundary spins from an odd number
of $a_{1},\ldots,a_{n}$. The unnormalised correlation $\mathcal{Z}_{\Omega_{\delta}}^{+,\beta}\left[\sigma_{a_{1}}\cdots\sigma_{a_{n}}\right]:=\mathcal{Z}_{\Omega_{\delta}}^{+,\beta}\cdot\mathbb{E}_{\Omega_{\delta}}^{+,\beta}\left[\sigma_{a_{1}}\cdots\sigma_{a_{n}}\right]$
will be used for normalisation below.
\begin{defn}
\label{def:bounded-fermion} For a bounded simply connected domain
$\Omega\subset\mathbb{C}$ with $n$ distinct interior points $a_{1},\ldots a_{n}$
and inverse temperature $\beta>0$, define for $z\in\mathcal{E}\mathcal{C}\left[\Omega_{\delta},a_{1},\ldots,a_{n}\right]$
which is not a lift of $a_{1}+\frac{\delta}{2}$ the \emph{discrete
massive fermion correlation}, or simply the \emph{discrete fermion}
\[
F_{\left[\Omega_{\delta},a_{1},\ldots,a_{n}\right]}\left(z|\beta\right)=\frac{1}{\mathcal{Z}_{\Omega_{\delta}}^{+,\beta}\left[\sigma_{a_{1}}\cdots\sigma_{a_{n}}\right]}\sum_{\gamma\in\Gamma_{\Omega_{\delta}}(a_{1}+\frac{\delta}{2},z)}c_{z}e^{-2\beta\left|\gamma\right|}\cdot\phi_{a_{1},\ldots,a_{n}}\left(\gamma,z\right)
\]
where
\begin{itemize}
\item $\Gamma_{\Omega_{\delta}}(a_{1}+\frac{\delta}{2},z)$ is the collection
of $\gamma=\omega\oplus\gamma_{0}$, where $\omega$ runs over elements
over $\Gamma_{\Omega_{\delta}}$, $\gamma_{0}$ is a fixed simple
lattice path from $a_{1}+\frac{\delta}{2}$ to $\pi(z)\in\mathcal{E}\mathcal{C}\left[\Omega_{\delta}\right]$,
and $\oplus$ refers to the XOR (symmetric difference) operation.
$|\gamma|$ is the number of full edges in $\gamma$. $c_{z}:=\cos\left(\frac{\pi}{8}+\Theta(\beta)\right)^{-1}$
if $z$ is an edge, and $1$ if it is a corner. Note that none of
these definitions refer to the double cover.
\item $\phi_{a_{1},\ldots,a_{n}}\left(\gamma,z\right)$ is a pure phase
factor, independent of $\beta$, defined by
\[
\phi_{a_{1},\ldots,a_{n}}\left(\gamma,z\right)=e^{-\frac{i}{2}\text{wind}(\text{p}(\gamma))}(-1)^{\#\text{loops}_{a_{1},\ldots,a_{n}}\left(\gamma\setminus\text{p}(\gamma)\right)}\text{sheet}_{a_{1},\ldots,a_{n}}(\text{p}(\gamma)),
\]
where $\text{p}(\gamma)$ is a simple path (we allow for self-touching,
as long as there is no self-crossing) from $a_{1}+\frac{\delta}{2}$
to $\pi(z)$ chosen in $\gamma$, $\text{wind}(\text{p}(\gamma))$
is the total turning angle of the tangent of $\text{p}(\gamma)$,
and $\text{sheet}_{a_{1},\ldots,a_{n}}(\text{p}(\gamma))\in\left\{ \pm1\right\} $
is defined to be $+1$ if the lift of $\text{p}(\gamma)$ to the double
cover starting from the fixed lift of $a_{1}+\frac{\delta}{2}$ (fixed
once forever in Section \ref{subsec:Notation}) ends at $z$ and $-1$
if it ends at $z^{\cdot}$. $\phi_{a_{1},\ldots,a_{n}}$ is well-defined;
see e.g. \cite{chelkak-izyurov,chelkak-hongler}. 
\end{itemize}
Note that $F_{\left[\Omega_{\delta},a_{1},\ldots,a_{n}\right]}$ is
naturally a spinor, i.e. $F_{\left[\Omega_{\delta},a_{1},\ldots,a_{n}\right]}\left(z^{\cdot}|\beta\right)=-F_{\left[\Omega_{\delta},a_{1},\ldots,a_{n}\right]}\left(z|\beta\right)$.
\end{defn}

The massive fermion satisfies a perturbed notion of discrete holomorphicity,
called \emph{massive s-holomorphicity}. First, define the projection
operator $\text{Proj}_{e^{i\theta}\mathbb{R}}x:=\frac{x+e^{2i\theta}\bar{x}}{2}$
to be the projection of the complex number $x$ to the line $e^{i\theta}\mathbb{R}$.
\begin{prop}
\label{prop:shol}The discrete massive fermion $F_{\left[\Omega_{\delta},a_{1},\ldots,a_{n}\right]}\left(\cdot|\beta\right)$
is \emph{massive s-holomorphic}, i.e. it satisfies
\begin{equation}
e^{\mp i\Theta}\text{Proj}_{e^{\pm i\Theta}\tau(c)\mathbb{R}}F_{\left[\Omega_{\delta},a_{1},\ldots,a_{n}\right]}\left(c\mp\frac{\tau(c)^{-2}i\delta}{2}|\beta\right)=F_{\left[\Omega_{\delta},a_{1},\ldots,a_{n}\right]}\left(c|\beta\right),\label{eq:s-holom}
\end{equation}
between $c\mp\frac{\tau(c)^{-2}i\delta}{2}\in\mathcal{E}\left[\Omega_{\delta},a_{1},\ldots,a_{n}\right]$
and $c\in\mathcal{C}\left[\Omega_{\delta},a_{1},\ldots,a_{n}\right]$
which is not a lift of $a_{1}+\frac{\delta}{2}$. At (the fixed lift
of) $a_{1}+\frac{\delta}{2}$, we have instead
\begin{equation}
e^{\mp i\Theta}\text{Proj}_{e^{\pm i\Theta}i\mathbb{R}}F_{\left[\Omega_{\delta},a_{1},\ldots,a_{n}\right]}\left(a_{1}+\frac{\delta\pm\delta i}{2}|\beta\right)=\mp i.\label{eq:singularity}
\end{equation}
\end{prop}

\begin{proof}
The proof of massive s-holomorphicity, essentially identical to the
massless case \cite[Subsection 3.1]{chelkak-hongler}, uses the bijection
between $\Gamma_{\Omega_{\delta}}(a_{1}+\frac{\delta}{2},c\mp\frac{\tau(c)^{-2}i\delta}{2})$
and $\Gamma_{\Omega_{\delta}}(a_{1}+\frac{\delta}{2},c)$ by the $\oplus$
(symmetric difference) operation. Without loss of generality, notice
that $\gamma\mapsto\gamma\oplus t$ with the fixed path $t:=\left\{ \left(c-\frac{\tau(c)^{-2}i\delta}{2},c-\frac{\tau(c)^{-2}\delta}{2}\right)\left(c-\frac{\tau(c)^{-2}\delta}{2},c\right)\right\} $
maps $\Gamma_{\Omega_{\delta}}(a_{1}+\frac{\delta}{2},c-\frac{\tau(c)^{-2}i\delta}{2})$
to $\Gamma_{\Omega_{\delta}}(a_{1}+\frac{\delta}{2},c)$ bijectively
(see Figure \ref{fig:The-square-lattice.}).

We now only need to show that the summand in the definition of $F_{\left[\Omega_{\delta},a_{1},\ldots,a_{n}\right]}\left(c\mp\frac{\tau(c)^{2}i\delta}{2}|\beta\right)$
for a given $\gamma$ transforms to the summand in $F_{\left[\Omega_{\delta},a_{1},\ldots,a_{n}\right]}\left(c|\beta\right)$
for $\gamma\oplus t$. Given $\text{p}(\gamma)\subset\gamma$, we
may take $\text{p}(\gamma\oplus t):=\text{p}(\gamma)\oplus t$. Clearly
$\left(-1\right){}^{\#\text{loops}},\text{sheet}$ remain the same
for $\gamma$ and $\gamma\oplus t$, while 
\[
\text{wind}\left(\text{p}(\gamma\oplus t)\right)=\begin{cases}
\text{wind}\left(\text{p}(\gamma)\right)+\frac{\pi}{4} & \text{if }\gamma\cap t\neq\emptyset\\
\text{wind}\left(\text{p}(\gamma)\right)-\frac{3\pi}{4} & \text{if }\gamma\cap t=\emptyset
\end{cases}\in\tau(c)\mathbb{R}.
\]

In the case where $\gamma\cap t\neq\emptyset$, $\left|\gamma\oplus t\right|=\left|\gamma\right|$
and 
\begin{alignat*}{1}
e^{-i\Theta}\text{Proj}_{e^{i\Theta}\tau(c)\mathbb{R}}e^{-\frac{i}{2}\text{wind}\left(\text{p}\left(\gamma\right)\right)} & =e^{-i\Theta}\text{Proj}_{e^{i\Theta}\tau(c)\mathbb{R}}\left[e^{i\Theta}e^{-\frac{i}{2}\text{wind}\left(\text{p}\left(\gamma\oplus t\right)\right)}\right]e^{-\frac{\pi i}{8}-i\Theta}\\
 & =e^{-i\Theta}\left[e^{i\Theta}e^{-\frac{i}{2}\text{wind}\left(\text{p}\left(\gamma\oplus t\right)\right)}\right]\text{Re}e^{-\frac{\pi i}{8}-i\Theta}\\
 & =e^{-\frac{i}{2}\text{wind}\left(\text{p}\left(\gamma\oplus t\right)\right)}\cos\left(\frac{\pi}{8}+\Theta\right),
\end{alignat*}
while if $\gamma\cap t=\emptyset$, $\left|\gamma\oplus t\right|=\left|\gamma\right|+1$
and similarly
\begin{alignat*}{1}
e^{-i\Theta}\text{Proj}_{e^{i\Theta}\tau(c)\mathbb{R}}e^{-\frac{i}{2}\text{wind}\left(\text{p}\left(\gamma\right)\right)} & =e^{-i\Theta}\left[e^{i\Theta}e^{-\frac{i}{2}\text{wind}\left(\text{p}\left(\gamma\oplus t\right)\right)}\right]\text{Re}e^{\frac{3\pi i}{8}-i\Theta}\\
 & =e^{-\frac{i}{2}\text{wind}\left(\text{p}\left(\gamma\oplus t\right)\right)}\sin\left(\frac{\pi}{8}+\Theta\right)\\
 & =e^{-\frac{i}{2}\text{wind}\left(\text{p}\left(\gamma\oplus t\right)\right)}e^{2\beta}\cos\left(\frac{\pi}{8}+\Theta\right),
\end{alignat*}
and the result follows.

At $a_{1}+\frac{\delta}{2}$, (say) $\gamma\in\Gamma_{\Omega_{\delta}}(a_{1}+\frac{\delta}{2},a_{1}+\frac{\delta}{2}+\frac{i\delta}{2})$
is mapped bijectively to $\gamma\oplus t\in\Gamma_{\Omega_{\delta}}$.
It is easy to see that
\[
\text{wind}\left(\text{p}(\gamma)\right)=\begin{cases}
-\frac{5\pi}{4}\pm2\pi & \text{if }t\subset\gamma,\\
-\frac{\pi}{4}\pm2\pi & \text{if }\gamma\cap t=\emptyset
\end{cases},
\]
and $(-1)^{\#\text{loops}_{a_{1},\ldots,a_{n}}\left(\gamma\setminus\text{p}(\gamma)\right)}\text{sheet}_{a_{1},\ldots,a_{n}}(\text{p}(\gamma))=\left(-1\right)^{\#\text{loops}_{a_{1},\ldots,a_{n}}\left(\gamma\oplus t\right)}$.
Now a simple computation similar to above yields the result.
\end{proof}
\begin{figure}

\includegraphics[width=0.7\columnwidth]{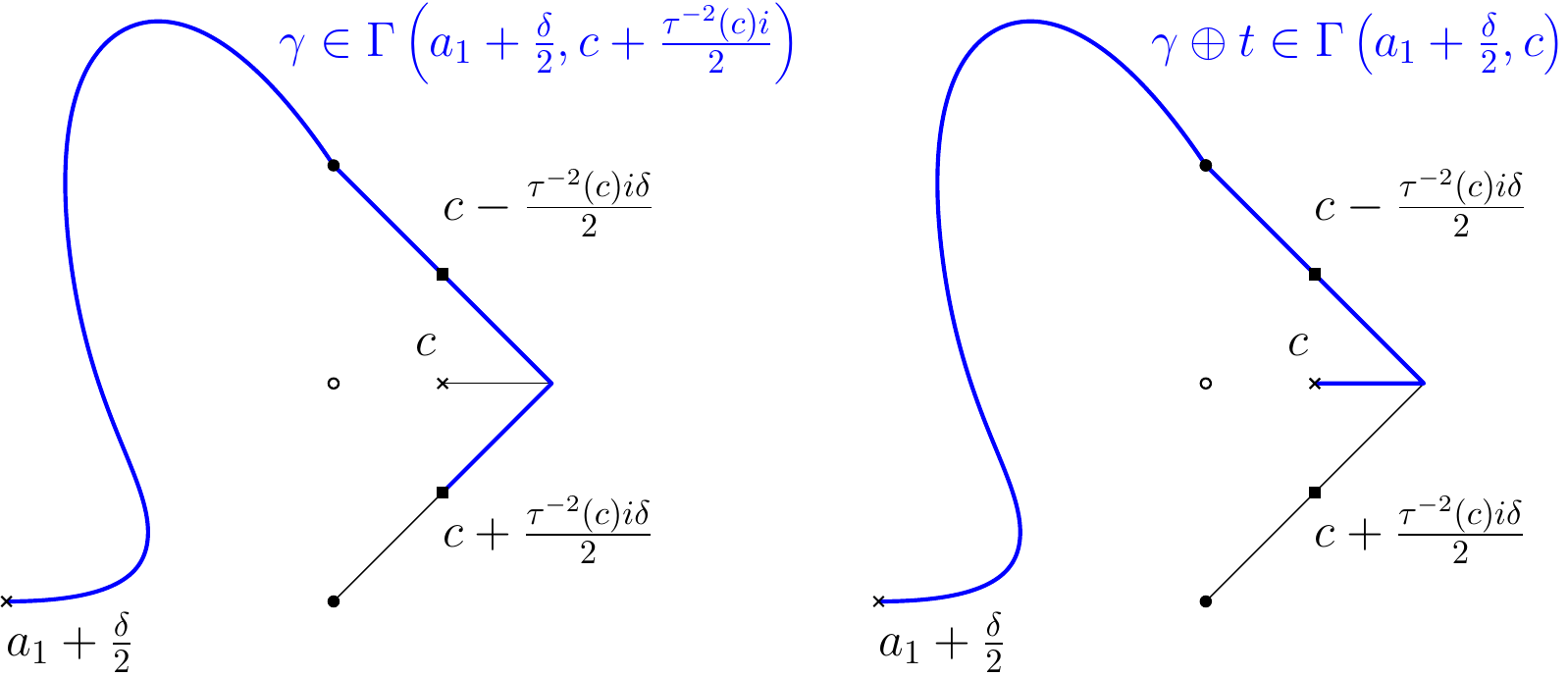}

\caption{\label{fig:The-proof-of}The proof of massive s-holomorphicity by
bijection.}

\end{figure}

We will take (\ref{eq:s-holom}) as the definition of \emph{($M,\beta,\Theta$)-massive
s-holomorphicity} (or just \emph{s-holomorphicity} when the mass is
clear) \emph{at} $c$ (or \emph{between} $c\mp\frac{\tau(c)^{-2}i\delta}{2}$).
We will see later that massive s-holomorphicity corresponds to the
continuous notion of perturbed analyticity $\partial_{z}f=m\bar{f}$.

The fermion defined above is a deterministic function without explicit
connection to the Ising model\textendash we now record how it encodes
probabilistic information.
\begin{prop}
\label{prop:isingcor}For $v=-1,0,1$, we have
\begin{alignat}{1}
F_{\left[\Omega_{\delta},a_{1},\ldots,a_{n}\right]}\left(a_{1}+\delta+\frac{e^{iv\frac{\pi}{2}}}{2}\delta|\beta\right) & =e^{-iv\frac{\pi}{4}}\frac{\mathbb{E}_{\Omega_{\delta}}^{\beta,+}\left[\sigma_{a_{1}+\delta+e^{iv\frac{\pi}{2}}\delta}\sigma_{a_{2}}\cdots\sigma_{a_{n}}\right]}{\mathbb{E}_{\Omega_{\delta}}^{\beta,+}\left[\sigma_{a_{1}}\cdots\sigma_{a_{n}}\right]},\label{eq:isingcor}
\end{alignat}
where $a_{1}+\delta+\frac{e^{iv\frac{\pi}{2}}}{2}\delta$ is lifted
to the same sheet as the fixed lift of $a_{1}+\frac{\delta}{2}$,
and thus for any $\beta_{1}<\beta_{2}$ there exist constants $C_{1}(\beta_{1},\beta_{2}),C_{2}(\beta_{1},\beta_{2})>0$,
uniform in $\beta\in\left[\beta_{1},\beta_{2}\right],\text{\ensuremath{\Omega}},\delta$,
such that
\[
C_{1}<\left|F_{\left[\Omega_{\delta},a_{1},\ldots,a_{n}\right]}\left(a_{1}+\delta+\frac{e^{iv\frac{\pi}{2}}}{2}\delta|\beta\right)\right|<C_{2}.
\]

In the case $n=2$ we have also 
\begin{equation}
\left|F_{\left[\Omega_{\delta},a_{1},a_{2}\right]}\left(a_{2}+\frac{\delta}{2}|\beta\right)\right|=\frac{\mathbb{E}_{\Omega_{\delta}^{*}}^{\beta^{*},\text{free}}\left[\sigma_{a_{1}+\delta}\sigma_{a_{2}+\delta}\right]}{\mathbb{E}_{\Omega_{\delta}}^{\beta,+}\left[\sigma_{a_{1}}\sigma_{a_{2}}\right]}.\label{eq:twopointcor}
\end{equation}
\end{prop}

\begin{proof}
(\ref{eq:isingcor}) and (\ref{eq:twopointcor}) were proved for the
massless case in \cite[Lemma 2.6]{chelkak-hongler} and the proof
is easily seen to not depend on a specific value of $\beta$. 

We carry out the $v=1$ case here: the path $t:=\left\{ (a_{1}+\frac{\delta}{2},a_{1}+\delta),(a_{1}+\delta,a_{1}+\delta+\frac{i\delta}{2})\right\} $
is admissible (i.e. without a crossing) as $\text{p}(\gamma)$ with
$\text{wind}(\text{p(\ensuremath{\gamma})})=\frac{\pi}{2},\text{sheet}_{a_{1},\ldots,a_{n}}(\text{p}(\gamma))=1$
in $\gamma\in\Gamma_{\Omega_{\delta}}(a_{1}+\frac{\delta}{2},a_{1}+\delta+\frac{i\delta}{2})$
if and only if $\gamma$ does not contain a loop separating $a_{1}$
from $a_{1}+\delta+i\delta$ and $(-1)^{\#\text{loops}_{a_{1},\ldots,a_{n}}\left(\gamma\setminus t\right)}=(-1)^{\#\text{loops}_{a_{1}+\delta+i\delta,\ldots,a_{n}}\left(\gamma\setminus t\right)}$.
If $\gamma$ does contain such a loop $L$, we need to choose $\text{p}(\gamma)=t\cup L$
with $\text{wind}\left(\text{p}(\gamma)\right)=\frac{\pi}{2}\pm2\pi$.
In this case, $\text{sheet}_{a_{1},\ldots,a_{n}}\left(\text{p}(\gamma)\right)=1$
if and only if $L$ encloses an even number of $a_{i}$. This means
that 
\begin{alignat*}{1}
(-1)^{\#\text{loops}_{a_{1},\ldots,a_{n}}\left(\gamma\setminus\text{p}(\gamma)\right)}\text{sheet}_{a_{1},\ldots,a_{n}}(\text{p}(\gamma)) & =(-1)^{\#\text{loops}_{a_{1},\ldots,a_{n}}\left(\gamma\setminus t\right)}\\
 & =-(-1)^{\#\text{loops}_{a_{1}+\delta+i\delta,\ldots,a_{n}}\left(\gamma\setminus t\right)},
\end{alignat*}
so in all cases the summand in the definition of $F_{\left[\Omega_{\delta},a_{1},\ldots,a_{n}\right]}$
is $\lambda^{-1}(-1)^{\#\text{loops}_{a_{1}+\delta+i\delta,\ldots,a_{n}}\left(\gamma\setminus t\right)}$
and the result follows. The uniform bound comes from the finite energy
property of the Ising model (such expectations are uniformly bounded
from $0$ and $\infty$ in any finite distribution).
\end{proof}
A crucial feature of (discrete) massive s-holomorphicity, shared with
its continuous counterpart, is that the line integral $\text{Re}\int\left[F_{\left[\Omega_{\delta},a_{1},\ldots,a_{n}\right]}\left(z|\beta\right)\right]^{2}dz$
can be defined path-independently.
\begin{prop}
\label{prop:isquare}There is a single valued function $H_{\left[\Omega_{\delta},a_{1},\ldots,a_{n}\right]}\left(x|\beta\right)$
up to a global constant on $\mathcal{VF}\left[\Omega_{\delta}\right]$
constructed by
\begin{alignat}{1}
H_{\left[\Omega_{\delta},a_{1},\ldots,a_{n}\right]}^{\circ}\left(f|\beta\right)-H_{\left[\Omega_{\delta},a_{1},\ldots,a_{n}\right]}^{\bullet}\left(v|\beta\right) & :=2\delta\left|F_{\left[\Omega_{\delta},a_{1},\ldots,a_{n}\right]}\left(\frac{v+f}{2}|\beta\right)\right|^{2},\label{eq:squareint_def}
\end{alignat}
where $f\in\mathcal{F}\left[\Omega_{\delta}\right]\cup\partial\mathcal{F}\left[\Omega_{\delta}\right]$
and $v\in\mathcal{V}\left[\Omega_{\delta}\right]$ are $\delta$ away
from each other so that $\frac{v+f}{2}$ is the corner between them.
Put $\left|F_{\left[\Omega_{\delta},a_{1},\ldots,a_{n}\right]}\left(a_{1}+\frac{\delta}{2}|\beta\right)\right|^{2}=1$.

At the boundary faces $H^{\circ}$ is constant, so we may put
\begin{alignat}{1}
H_{\left[\Omega_{\delta},a_{1},\ldots,a_{n}\right]}^{\circ}\left(f|\beta\right):=0 & \text{ if }f\in\partial\mathcal{F}\left[\Omega_{\delta}\right],\label{eq:squareint_boundary}
\end{alignat}
and further define $H_{\left[\Omega_{\delta},a_{1},\ldots,a_{n}\right]}^{\bullet}\left(v|\beta\right):=0\text{ if }v\in\partial\mathcal{V}\left[\Omega_{\delta}\right]$,
then across a boundary edge $e=(a_{\text{int}}a)$ with $a_{\text{int}}\in\mathcal{V}\left[\Omega_{\delta}\right],a\in\partial\mathcal{V}\left[\Omega_{\delta}\right]$,
we have 
\begin{alignat}{1}
\partial_{\nu_{\text{out}}}^{\delta}H_{\left[\Omega_{\delta},a_{1},\ldots,a_{n}\right]}^{\bullet}\left(e|\beta\right) & :=-H_{\left[\Omega_{\delta},a_{1},\ldots,a_{n}\right]}^{\bullet}\left(a_{\text{int}}|\beta\right)\nonumber \\
 & =2\cos^{2}\left(\frac{\pi}{8}+\Theta\right)\delta\left|F_{\left[\Omega_{\delta},a_{1},\ldots,a_{n}\right]}\left(e|\beta\right)\right|^{2}\label{eq:squareint_outernormal}\\
 & \geq0.\nonumber 
\end{alignat}
\end{prop}

\begin{proof}
(\ref{eq:squareint_def}) gives rise to a single valued function because
doing a loop around any edge will give zero: if $e=(v_{1}v_{2})$
is an edge which is incident to faces $f_{1},f_{2}$, $H_{\left[\Omega_{\delta},a_{1},\ldots,a_{n}\right]}^{\circ}\left(f_{1}\right)-H_{\left[\Omega_{\delta},a_{1},\ldots,a_{n}\right]}^{\bullet}\left(v_{1}\right)+H_{\left[\Omega_{\delta},a_{1},\ldots,a_{n}\right]}^{\circ}\left(f_{2}\right)-H_{\left[\Omega_{\delta},a_{1},\ldots,a_{n}\right]}^{\bullet}\left(v_{2}\right)$
according to (\ref{eq:squareint_def}) is simply equal to $2\delta\left|F_{\left[\Omega_{\delta},a_{1},\ldots,a_{n}\right]}\left(\frac{f_{1}+v_{1}}{2}\right)\right|^{2}+2\delta\left|F_{\left[\Omega_{\delta},a_{1},\ldots,a_{n}\right]}\left(\frac{f_{2}+v_{2}}{2}\right)\right|^{2}=2\delta\left|F_{\left[\Omega_{\delta},a_{1},\ldots,a_{n}\right]}\left(e\right)\right|^{2}$,
since the two corner values at $\frac{f_{1,2}+v_{1,2}}{2}$ are simply
projections of $F_{\left[\Omega_{\delta},a_{1},\ldots,a_{n}\right]}\left(e\right)$
onto orthogonal lines by s-holomorphicity. $H_{\left[\Omega_{\delta},a_{1},\ldots,a_{n}\right]}^{\circ}\left(f_{2}\right)-H_{\left[\Omega_{\delta},a_{1},\ldots,a_{n}\right]}^{\bullet}\left(v_{1}\right)+H_{\left[\Omega_{\delta},a_{1},\ldots,a_{n}\right]}^{\circ}\left(f_{1}\right)-H_{\left[\Omega_{\delta},a_{1},\ldots,a_{n}\right]}^{\bullet}\left(v_{2}\right)$
is also equal to $2\delta\left|F_{\left[\Omega_{\delta},a_{1},\ldots,a_{n}\right]}\left(e\right)\right|^{2}$,
and increments along the loop $f_{1}\sim v_{1}\sim f_{2}\sim v_{2}\sim f_{1}$
sum to zero.

Then the boundary behaviour can be easily verified noting that $F_{\left[\Omega_{\delta},a_{1},\ldots,a_{n}\right]}\left(e|\beta\right)\in\nu_{\text{out }}(e)^{-1/2}\mathbb{R}$
if $e$ is a boundary edge. See \cite[Proposition 3.6]{chelkak-hongler}
for a massless counterpart.
\end{proof}
\begin{rem}
Writing corner values in terms of edge projections, we obtain that
$H_{\left[\Omega_{\delta},a_{1},\ldots,a_{n}\right]}$ is a discrete
version of the integral $\text{Re}\int\left[F_{\left[\Omega_{\delta},a_{1},\ldots,a_{n}\right]}\left(z|\beta\right)\right]^{2}dz$
in that 
\begin{equation}
H_{\left[\Omega_{\delta},a_{1},\ldots,a_{n}\right]}^{\circ}\left(f_{1}|\beta\right)-H_{\left[\Omega_{\delta},a_{1},\ldots,a_{n}\right]}^{\circ}\left(f_{2}|\beta\right)=\sqrt{2}\sin\left(\frac{\pi}{4}+2\Theta\right)\text{Re}\left[F_{\left[\Omega_{\delta},a_{1},\ldots,a_{n}\right]}\left(z|\beta\right)^{2}\left(f_{1}-f_{2}\right)\right],\label{eq:squareint}
\end{equation}
 for all $f_{1},f_{2}$ of distance $\sqrt{2}\delta$ from each other,
and 
\begin{equation}
H_{\left[\Omega_{\delta},a_{1},\ldots,a_{n}\right]}^{\bullet}\left(v_{1}|\beta\right)-H_{\left[\Omega_{\delta},a_{1},\ldots,a_{n}\right]}^{\bullet}\left(v_{2}|\beta\right)=\sqrt{2}\sin\left(\frac{\pi}{4}-2\Theta\right)\text{Re}\left[F_{\left[\Omega_{\delta},a_{1},\ldots,a_{n}\right]}\left(z|\beta\right)^{2}\left(v_{1}-v_{2}\right)\right],\label{eq:squareint2}
\end{equation}
 for $v_{1},v_{2}$ of distance $\sqrt{2}\delta$ from each other.
\end{rem}

The functions $H^{\circ,\bullet}$ constructed in the previous proposition
satisfy a discrete version of a second-order partial differential
equation. We first recall the standard discrete operators $\partial_{\bar{z}}^{\delta},\Delta^{\delta}$;
as alluded to in the notation section, we make a small modification
to the conventional definition for the laplacian in $\mathcal{V}\left[\Omega_{\delta}\right]$.
We make this boundary modification specifically for $\mathcal{V}\left[\Omega_{\delta}\right]$,
which lets us define $H_{\left[\Omega_{\delta},a_{1},\ldots,a_{n}\right]}^{\bullet}=0$
on boundary vertices; see \cite{chsm2012} for a motivation. 
\begin{defn}
Suppose $A$ is a function defined on $\mathcal{E}\left[\Omega_{\delta}\right]$
and $B$ on $\mathcal{V}\left[\Omega_{\delta}\right]$ (or $\mathcal{F}\left[\Omega_{\delta}\right]$,
or any locally isomorphic graph). We define the discrete Wirtinger
derivative and laplacian by
\begin{alignat}{1}
\partial_{\bar{z}}^{\delta}A(x) & :=\sum_{m=0}^{3}i^{m}e^{i\pi/4}A\left(x+i^{m}e^{i\pi/4}\frac{\delta}{\sqrt{2}}\right)\text{ if, e.g., }x\in\mathcal{VF}\left[\Omega_{\delta}\right],\nonumber \\
\Delta^{\delta}B(x) & :=\sum_{m=0}^{3}c_{m}\left[B\left(x+i^{m}e^{i\pi/4}\sqrt{2}\delta\right)-B(x)\right]\text{ if, e.g., }x\in\mathcal{V}\left[\Omega_{\delta}\right],\label{eq:lap_vert}
\end{alignat}
where $c_{m}=1$ if $x+i^{m}e^{i\pi/4}\sqrt{2}\delta\in\mathcal{V}\left[\Omega_{\delta}\right]$,
and $c_{m}=\frac{\sin\left(\frac{\pi}{4}-2\Theta\right)}{\cos^{2}\left(\frac{\pi}{8}+\Theta\right)}$
if $x+i^{m}e^{i\pi/4}\sqrt{2}\delta\in\partial\mathcal{V}\left[\Omega_{\delta}\right]$.
For any other lattice, $c_{m}\equiv1$. If $\Delta^{\delta}B(x)=M_{H}^{2}B(x)$
for $M_{H}^{2}=\frac{8\sin^{2}2\Theta}{\cos4\Theta}$, we call $B$
\emph{($\Theta$,$M_{H}^{2}$)-massive harmonic} at $x$.
\end{defn}

The spinor $F_{\left[\Omega_{\delta},a_{1},\ldots,a_{n}\right]}$
is $\Theta$-massive harmonic \cite{bedc,dgp,hkz}, at least away
from monodromy; see Proposition \ref{prop:Appendix}, and also the
proof of Proposition \ref{prop:onepointspinor} for the behaviour
near monodromy.

Note that in the scaling limit $\delta\downarrow0$, $\frac{1}{2\sqrt{2}\delta}\partial_{\bar{z}}^{\delta}\to\partial_{\bar{z}}:=\frac{1}{2}\left(\partial_{x}+i\partial_{y}\right)$
and $\frac{1}{\left(\sqrt{2}\delta\right)^{2}}\Delta^{\delta}\to\Delta$.
Therefore, in the following result, the second terms become negligible
as $\delta\to0$ compared to the first, at least if $F$ is in some
sense regular; for an analogue in the scaling limit, see (\ref{eq:continuoushlap}).
\begin{prop}
\label{prop:squareint}For $x\neq a_{1}+\delta,a_{2},\ldots a_{n},$
we have 
\begin{alignat}{1}
\Delta^{\delta}H_{\left[\Omega_{\delta},a_{1},\ldots,a_{n}\right]}^{\circ}(x) & =2\sin\left(\frac{\pi}{4}+2\Theta\right)\delta\cdot\nonumber \\
 & \left[A_{\Theta}\sum_{n=0}^{3}\left|F\left(x+i^{n}e^{i\pi/4}\frac{\delta}{\sqrt{2}}\right)\right|^{2}+B_{\Theta}\left|\partial_{\bar{z}}\bar{F}\right|^{2}(x)\right]\label{eq:squareint_lap}\\
\Delta^{\delta}H_{\left[\Omega_{\delta},a_{1},\ldots,a_{n}\right]}^{\bullet}(x) & =2\sin\left(\frac{\pi}{4}-2\Theta\right)\delta\cdot\nonumber \\
 & \left[-A_{-\Theta}\sum_{n=0}^{3}\left|F\left(x+i^{n}e^{i\pi/4}\frac{\delta}{\sqrt{2}}\right)\right|^{2}-B_{-\Theta}\left|\partial_{\bar{z}}\bar{F}\right|^{2}(x)\right]\nonumber 
\end{alignat}
where $A_{\Theta},B_{\Theta}$ are explicit numbers which depend on
$\Theta$. $A$ is odd and $B$ is even in $\Theta$, $A_{\Theta}\sim4\sqrt{2}\Theta$
and $B_{\Theta}\sim\frac{1}{2\sqrt{2}}$ as $\Theta\to0$ (see (\ref{eq:squareint_dbar})).
\end{prop}

\begin{proof}
We calculate $\partial_{\bar{z}}F_{\left[\Omega_{\delta},a_{1},\ldots,a_{n}\right]}^{2}$
in Proposition \ref{prop:Appendix}. The laplacians follow straightforwardly
by noting that $\Delta^{\delta}H_{\left[\Omega_{\delta},a_{1},\ldots,a_{n}\right]}^{\circ,\bullet}(x)=\text{Re}\left[2\sin\left(\frac{\pi}{4}\pm\Theta\right)\delta\partial_{\bar{z}}F_{\left[\Omega_{\delta},a_{1},\ldots,a_{n}\right]}^{2}\right]$,
which is true if $x\in\mathcal{V}\left[\Omega_{\delta}\right]$ is
adjacent to $\partial\mathcal{V}\left[\Omega_{\delta}\right]$ as
well (the boundary conductances are defined precisely to preserve
this relation). See also Remark \ref{rem:movable}.
\end{proof}
\begin{rem}
\label{rem:movable}Note that Proposition \ref{prop:squareint} applies
for $x=a_{1}$ as well, since, thanks to the singularity (\ref{eq:singularity}),
projections of $F_{\left[\Omega_{\delta},a_{1},\ldots,a_{n}\right]}\left(a_{1}+e^{i\frac{\pi}{4}}\cdot\frac{\delta}{\sqrt{2}}\right),F_{\left[\Omega_{\delta},a_{1},\ldots,a_{n}\right]}\left(a_{1}+e^{i\frac{3\pi}{2}}\cdot e^{i\frac{\pi}{4}}\cdot\frac{\delta}{\sqrt{2}}\right)$
(on the sheet which is cut along $\mathbb{R}_{\geq0}$ and has the
fixed lift of $a_{1}+\frac{\delta}{2}$ on the upper side of $\mathbb{R}_{\geq0}$)
onto the line $i\mathbb{R}$ are equal to $-i$; on this sheet, $F_{\left[\Omega_{\delta},a_{1},\ldots,a_{n}\right]}$
does show s-holomorphicity between $4$ edges surrounding $a_{1}$.
The singularity effectively transfers the monodromy at face $a_{1}$
to the vertex $a_{1}+\delta$.
\end{rem}

\subsection{Full-Plane Fermions\label{subsec:Infinite-Volume-Limit}}

We now define the fermion $F_{\left[\Omega_{\delta},a_{1},\ldots,a_{n}\right]}$
for $\Omega=\mathbb{C}$ by taking increasingly bigger balls $B_{R}=B_{R}(0)$,
allowing us to extract informations about the Ising measure on the
corresponding discretised domains.
\begin{lem}
Fix $\delta>0$ and $\Theta<0$. There exists a constant $C=C(\delta,\Theta)$
such that $\left|F_{\left[\left(B_{R}\right)_{\delta},a_{1},\ldots,a_{n}\right]}(c)\right|\leq C$
for any $c\in\mathcal{C}\left(\left[\left(B_{R}\right)_{\delta},a_{1},\ldots,a_{n}\right]\right)$.
\end{lem}

\begin{proof}
Discrete Green's formula implies that
\begin{equation}
\sum_{v\in\mathcal{V}\left[\left(B_{R}\right)_{\delta}\right]}\Delta^{\delta}H_{\left[\left(B_{R}\right)_{\delta},a_{1},\ldots,a_{n}\right]}^{\bullet}(v)=\frac{\sin\left(\frac{\pi}{4}-2\Theta\right)}{\cos^{2}\left(\frac{\pi}{8}+\Theta\right)}\sum_{e\in\partial\mathcal{E}\left[\left(B_{R}\right)_{\delta}\right]}\partial_{\nu_{\text{out}}}^{\delta}H_{\left[\left(B_{R}\right)_{\delta},a_{1},\ldots,a_{n}\right]}^{\bullet}\left(e\right),\label{eq:discrete-greens-formula}
\end{equation}
and we can use (\ref{eq:squareint_lap}) on vertices other than $a_{1}+\delta$
for the laplacian and (\ref{eq:squareint_outernormal}) for the boundary
outer derivatives, so leaving just the laplacian at $(a_{1}+\delta)$
on the left hand side
\begin{alignat}{1}
 & \Delta^{\delta}H_{\left[\left(B_{R}\right)_{\delta},a_{1},\ldots,a_{n}\right]}^{\bullet}(a_{1}+\delta)=\nonumber \\
\sum_{\begin{array}{c}
v\in\mathcal{V}\left[\left(B_{R}\right)_{\delta}\right]\\
v\neq a_{1}+\delta
\end{array}} & 2\sin\left(\frac{\pi}{4}-\Theta\right)\delta\cdot\left[A_{-\Theta}\sum_{n=0}^{3}\left|F\left(v+i^{n}e^{i\pi/4}\frac{\delta}{\sqrt{2}}\right)\right|^{2}+B_{-\Theta}\left|\partial_{\bar{z}}\bar{F}\right|^{2}(v)\right]\nonumber \\
 & +\sum_{e\in\partial\mathcal{E}\left[\left(B_{R}\right)_{\delta}\right]}2\delta\sin\left(\frac{\pi}{4}-2\Theta\right)\left|F_{\left[\left(B_{R}\right)_{\delta},a_{1},\ldots,a_{n}\right]}\left(e\right)\right|^{2}\nonumber \\
 & \geq2\sin\left(\frac{\pi}{4}-\Theta\right)A_{-\Theta}\delta\sum_{\begin{array}{c}
c\in\mathcal{C}\left[\left(B_{R}\right)_{\delta}\right]\\
c\neq a_{1}+\delta+\frac{i^{m}\delta}{2}
\end{array}}\left|F\left(c\right)\right|^{2}\label{eq:precompactness}
\end{alignat}

Note that $A_{-\Theta},B_{-\Theta}>0$, so we have the desired bound
if $\Delta^{\delta}H_{\left[\left(B_{R}\right)_{\delta},a_{1},\ldots,a_{n}\right]}^{\bullet}(a_{1}+\delta)$
is bounded. But by s-holomorphicity, clearly 
\begin{equation}
\Delta^{\delta}H_{\left[\left(B_{R}\right)_{\delta},a_{1},\ldots,a_{n}\right]}^{\bullet}(a_{1}+\delta)\leq16\sqrt{2}\max_{0\leq m\leq3}\delta\left|F_{\left[\left(B_{R}\right)_{\delta},a_{1},\ldots,a_{n}\right]}\left(a_{1}+\delta+\frac{e^{iv\frac{\pi}{2}}}{2}\delta\right)\right|^{2},\label{eq:a1lap}
\end{equation}
but the fermion value on the right hand side is bounded by a constant
only depending on $\beta$ by Proposition \ref{prop:isingcor}.
\end{proof}
Given uniform boundedness, we can use diagonalisation to get a subsequential
limit $F_{\left[\mathbb{C}_{\delta},a_{1},\ldots,a_{n}\right]}$ on
the whole of $\left[\mathbb{C}_{\delta},a_{1},\ldots,a_{n}\right]$.
It suffices to show that such a limit must be unique.
\begin{prop}
\label{prop:unbounded-prop}Any subsequential limit $F_{\left[\mathbb{C}_{\delta},a_{1},\ldots,a_{n}\right]}$
of $F_{\left[\left(B_{R}\right)_{\delta},a_{1},\ldots,a_{n}\right]}$
as $R\to\infty$
\begin{enumerate}
\item shows massive s-holomorphicity and the singularity at $a_{1}+\frac{\delta}{2}$
as in Proposition \ref{prop:shol}, 
\item at infinity: $F_{\left[\mathbb{C}_{\delta},a_{1},\ldots,a_{n}\right]}\to0$
uniformly and its square integral $H_{\left[\mathbb{C}_{\delta},a_{1},\ldots,a_{n}\right]}$
tends to a finite constant,
\item has finite discrete '$L^{2}$ norm':
\begin{alignat}{1}
\sum_{\begin{array}{c}
v\in\mathcal{V}\left[\mathbb{C}_{\delta}\right]\\
v\neq a_{1}+\delta
\end{array}}2 & \sin\left(\frac{\pi}{4}-\Theta\right)\delta\cdot\left[A_{-\Theta}\sum_{n=0}^{3}\left|F\left(v+i^{n}e^{i\pi/4}\frac{\delta}{\sqrt{2}}\right)\right|^{2}\right]\label{eq:unbounded_ineq}\\
\leq & \sum_{\begin{array}{c}
v\in\mathcal{V}\left[\mathbb{C}_{\delta}\right]\\
v\neq a_{1}+\delta
\end{array}}\left|\Delta^{\delta}H_{\left[\mathbb{C}_{\delta},a_{1},\ldots,a_{n}\right]}^{\bullet}(v)\right|<cst\cdot\delta<\infty,\nonumber 
\end{alignat}
with constant independent of $\delta$.
\end{enumerate}
Moreover, there is only one such function $F_{\left[\mathbb{C}_{\delta},a_{1},\ldots,a_{n}\right]}$.
\end{prop}

\begin{proof}
The first entry is immediate from the corresponding properties of
$F_{\left[\left(B_{R}\right)_{\delta},a_{1},\ldots,a_{n}\right]}$.

The inequality (\ref{eq:unbounded_ineq}) can be deduced from a uniform
bound for the analogue for $\left|\Delta^{\delta}H_{\left[\left(B_{R}\right)_{\delta},a_{1},\ldots,a_{n}\right]}^{\bullet}\right|$.
It is bounded independently of $R$ by (\ref{eq:discrete-greens-formula})
and (\ref{eq:a1lap}), so monotone convergence gives the desired inequality
for the limit as $R\to\infty$.

The infinity behaviour is then immediate from the fact that the sum
of $\left|F_{\left[\mathbb{C}_{\delta},a_{1},\ldots,a_{n}\right]}\right|^{2}$
along any line in $\mathbb{C}\setminus B_{R}$ vanishes uniformly
as $R\to\infty$ by (\ref{eq:unbounded_ineq}).

If there are two such $F_{\left[\mathbb{C}_{\delta},a_{1},\ldots,a_{n}\right]}$,
their difference $\hat{F}_{\left[\mathbb{C}_{\delta},a_{1},\ldots,a_{n}\right]}$
is everywhere s-holomorphic, and since the sum of $\left|\hat{F}_{\left[\mathbb{C}_{\delta},a_{1},\ldots,a_{n}\right]}\right|^{2}$
along any line in $\mathbb{C}\setminus B_{R}$ is finite and decays
to zero, the square integral $\hat{H}_{\left[\mathbb{C}_{\delta},a_{1},\ldots,a_{n}\right]}$
is finite and constant at infinity. $\hat{H}_{\left[\Omega_{\delta},a_{1},\ldots,a_{n}\right]}^{\bullet}$
is everywhere superharmonic and is finite at infinity, so $H_{\left[\Omega_{\delta},a_{1},\ldots,a_{n}\right]}^{\dagger\bullet}$
is constant and $F_{\left[\Omega_{\delta},a_{1},\ldots,a_{n}\right]}^{\dagger}\equiv0$
.
\end{proof}
Fix $\Theta<0$. On the full plane, we have an explicit characterisation
of the one point spinor in terms of the \emph{massive harmonic measure}
of the slit plane $\text{hm}_{(1+i)\mathbb{Z}^{2}\setminus\mathbb{R}_{<0}}^{0}$:
$\text{hm}_{(1+i)\mathbb{Z}^{2}\setminus\mathbb{R}_{<0}}^{0}(z|\Theta)$
for $z\in(1+i)\mathbb{Z}^{2}$ is the probability of a simple random
walk started at $z$ extinguished at each step with probability $\left(1+\frac{2\sin^{2}2\Theta}{\cos(4\Theta)}\right)^{-1}\frac{2\sin^{2}2\Theta}{\cos(4\Theta)}$
to successfully reach $0$ before hitting $(1+i)\mathbb{Z}^{2}\cap\mathbb{R}_{<0}$.
$\text{hm}_{(1+i)\mathbb{Z}^{2}\setminus\mathbb{R}_{<0}}^{0}(\cdot|\Theta)$
is the unique $\Theta$-massive harmonic function (in the sense of
(\ref{eq:dmharm})) on $(1+i)\mathbb{Z}^{2}\setminus\mathbb{R}_{\leq0}$
which has the boundary values $1$ at $0$ and $0$ on $(1+i)\mathbb{Z}^{2}\cap\mathbb{R}_{<0}$
and infinity.
\begin{prop}
\label{prop:onepointspinor}Denote the slit planes $\mathbb{X}^{+}:=\left\{ z\in\left[\mathbb{C},0\right]:\text{Re}\sqrt{z}>0\right\} \cong\mathbb{C}\setminus\mathbb{R}_{>0}$
and $\mathbb{Y}^{+}:=\left\{ z\in\left[\mathbb{C},0\right]:\text{Im}\sqrt{z}>0\right\} \cong\mathbb{C}\setminus\mathbb{R}_{>0}$.
Then
\[
F_{\left[\mathbb{C}_{\delta},0\right]}(c\delta|\Theta)=\begin{cases}
\text{hm}_{(1+i)\mathbb{Z}^{2}\setminus\mathbb{R}_{<0}}^{0}\left(c-\frac{3}{2}|\Theta\right) & c\delta\in\mathbb{X}^{+}\cap\mathcal{C}^{1}\left[\mathbb{C}_{\delta},0\right]\\
-i\text{hm}_{(1+i)\mathbb{Z}^{2}\setminus\mathbb{R}_{<0}}^{0}\left(\frac{1}{2}-c|\Theta\right) & c\delta\in\mathbb{Y}^{+}\cap\mathcal{C}^{i}\left[\mathbb{C}_{\delta},0\right]
\end{cases}.
\]
\end{prop}

\begin{proof}
By Proposition \ref{prop:Appendix}, $F_{\left[\mathbb{C}_{\delta},0\right]}(c\delta)$
restricted to $\mathcal{C}^{1}\left[\mathbb{C}_{\delta},0\right]$
or $\mathcal{C}^{i}\left[\mathbb{C}_{\delta},0\right]$ is $\Theta$-massive
harmonic, except possibly at the lifts of $\pm\frac{\delta}{2}$ (because
there is no planar neighbourhood $G_{\delta}$ around these points
in $\left[\mathbb{C}_{\delta},0\right]$, see Figure \ref{fig:Using-holomorphicity-to})
and $\frac{3\delta}{2}$ (because $F_{\left[\mathbb{C}_{\delta},0\right]}(c\delta)$
is not s-holomorphic at the lifts of $\frac{\delta}{2}$).

From Definition \ref{def:bounded-fermion}, it is clear that $F_{\left[\mathbb{C}_{\delta}\cap B_{R},0\right]}(c\delta)=0$
for any $c\in\mathcal{C}^{1}\left[\mathbb{C}_{\delta},0\right]$ on
the (lift of) negative real line and $c\in\mathcal{C}^{i}\left[\mathbb{C}_{\delta},0\right]$
on the (lift of) positive real line, since the complex phase $\phi_{a_{1},\ldots,a_{n}}\left(\gamma,c\right)$
for any $\gamma\in\Gamma_{\mathbb{C}_{\delta}\cap B_{R}}(\frac{\delta}{2},\pi(c))$
switches sign for the reflection across the real line $\gamma_{r}:=\left\{ \bar{e}:e\in\gamma\right\} $,
and $\Gamma_{\mathbb{C}_{\delta}\cap B_{R}}(\frac{\delta}{2},\pi(c))$
is invariant under the reflection. We conclude that $F_{\left[\mathbb{C}_{\delta},0\right]}(c\delta)=0$
for $c\in\mathcal{C}^{1}\left[\mathbb{C}_{\delta},0\right]$ .

Also note that $F_{\left[\mathbb{C}_{\delta},0\right]}(\frac{3\delta}{2})=1$,
where $\frac{3\delta}{2}$ is evaluated at the lift which is on the
same sheet as the fixed lift of $\frac{\delta}{2}$, since the corresponding
Ising quantity in (\ref{eq:isingcor}) is precisely the ratio of two
adjacent magnetisations (spin $1$-point functions); in the infinite
volume limit with $\beta>\beta_{c}$, the ratio tends to $1$.

All in all, $F_{\left[\mathbb{C}_{\delta},0\right]}(\cdot|\Theta)$
restricted to $\mathbb{X}^{+}\cap\mathcal{C}^{1}\left[\mathbb{C}_{\delta},0\right]$
is the massive harmonic function which takes the boundary values $1$
on the lift of $\frac{3\delta}{2}$ in $\mathbb{X}^{+}\cap\mathcal{C}^{1}\left[\mathbb{C}_{\delta},0\right]$
and $0$ on the lift of the negative real line and infinity, whence
the identification with $\text{hm}_{(1+i)\mathbb{Z}^{2}\setminus\mathbb{R}_{<0}}^{0}$.
The same argument applies for its restriction to $\mathbb{Y}^{+}\cap\mathcal{C}^{i}\left[\mathbb{C}_{\delta},0\right]$,
noting the boundary value at $\frac{\delta}{2}$ given by (\ref{eq:singularity}).
\end{proof}

\subsection{Massive Complex Analysis: Continuous Fermions and their Uniqueness\label{subsec:Massive-Complex-Analysis:}}

We now present the boundary value problem which the limit of our discrete
fermions solves. Since massive models by definition do not possess
conformal invariance, we cannot expect the limit to be holomorphic
as in the critical case; instead, it satisfies a particular perturbed
notion of holomorphicity $\partial_{\bar{z}}f=m\bar{f}$, which belongs
to various similar families of functions dubbed \emph{generalised
analytic} or \emph{pseudoanalytic} in writings of, e.g., L. Bers and
I. N. Vekua. Here we use a small excerpt of the established theory,
saying that $f$ is holomorphic up to a continuous factor. Although
any nonzero smooth function is trivially pseudoanalytic (with the
non-constant complex coefficient $\partial_{\bar{z}}/\bar{f}$), fixing
the governing equation $\partial_{\bar{z}}f=m\bar{f}$ allows us to
use known functions which satisfy the same equation to successively
cancel out singularities; in other words, we have a generalised version
of Laurent series. Imposing a real constant $m$ in addition yields
the possibility to define the line integral $\text{Re}\int f^{2}dz$.

In analogy with the discrete terminology, we refer to our particular
continuous condition \emph{($m$)-massive holomorphicity.} We will
assume $m<0$ is fixed throughout this section.
\begin{lem}[Similarity Principle]
\label{lem:similarity}Let $D$ be a bounded domain in $\mathbb{C}$.
If a continuously differentiable function $f:D\to\mathbb{C}$ is $m$-massive
meromorphic, i.e. satisfies $\partial_{\bar{z}}f:=\frac{1}{2}\frac{\partial}{\partial x}f+\frac{i}{2}\frac{\partial}{\partial y}f=m\bar{f}$
except on a finite set, there exists a Hölder continuous function
$s$ on $\bar{U}$ such that $e^{-s}f$ is holomorphic in $U$.

In particular, $f$ only vanishes at isolated points.
\end{lem}

\begin{proof}
See \cite[Section 10]{bers} or \cite[Section III.4]{vek}; we also
explain the underlying idea and explicitly carry out specific cases
in Section \ref{subsec:Analysis-of-the}, e.g. the proof of Lemma
\ref{lem:obs-decomposition}.
\end{proof}
\begin{cor}
\label{cor:powerseries}There exists a family of $m$-massive holomorphic
functions $Z_{n}^{1},Z_{n}^{i}$ for each $n\in\mathbb{Z}$ such that
as $z\to0$
\[
Z_{n}^{1}(z)\sim z^{n},Z_{n}^{i}(z)\sim iz^{n},
\]
and any function $f$ that is $m$-massive holomorphic in a punctured
neighbourhood of $a$ can be expressed near $a$ as formal power series
in $Z_{n}$, i.e.
\begin{equation}
f(z)=\sum_{n}\left[A_{n}^{1}Z_{n}^{1}(z-a)+A_{n}^{i}Z_{n}^{i}(z-a)\right]\label{eq:powerseries}
\end{equation}
for real numbers $A_{n}^{1,i}$. If $f(z)$ is a spinor defined on
a double cover ramified at $a$, $f$ admits formal power series in
analogous functions indexed by half-integers $Z_{n+\frac{1}{2}}$.
These expansions are uniformly convergent in a small disc or annulus
around $a$, respectively if $f$ has a singularity at $a$ or not.
\end{cor}

\begin{proof}
This corollary depends on the fact that $s$ in Lemma \ref{lem:similarity}
is continuous, so that if a massive holomorphic function is $o((z-a)^{n})$
at $a$ then it is automatically $O((z-a)^{n+1})$; see \cite[Section 5]{bers}.
\end{proof}
\begin{rem}
\label{rem:formalpowers}There is no canonical choice of the `local
formal powers' $Z_{n}^{1,i}$. We will need explicit functions to
expand continuous fermions around their singularities and analyse
them further to derive Painlevé III in Section \ref{sec:Continuum-Analysis:-Painlev};
we hereby fix the following radially symmetric functions for half-integers
$\nu$:
\begin{alignat}{1}
Z_{\nu}^{1}(re^{i\theta}) & :=\frac{\Gamma(\nu+1)}{|m|^{\nu}}\left[e^{i\nu\theta}I_{\nu}(2|m|r)+(\text{sgn}m)\cdot e^{-i(\nu+1)\phi}I_{\nu+1}(2|m|r)\right],\label{eq:formalpowers}\\
Z_{\nu}^{i}(re^{i\theta}) & :=\frac{\Gamma(\nu+1)}{|m|^{\nu}}\left[ie^{i\nu\theta}I_{\nu}(2|m|r)-i(\text{sgn}m)\cdot e^{-i(\nu+1)\phi}I_{\nu+1}(2|m|r)\right],\nonumber 
\end{alignat}
where $I_{n}$ is the modified Bessel function of the first kind.
One can easily verify the desired asymptotics and massive holomorphicity
from the corresponding facts about $I_{\nu}$, namely that $I_{\nu}(r)\stackrel{r\to0}{\sim}\Gamma(\nu+1)\left(\frac{r}{2}\right)^{\nu}$
and $I_{\nu}'(r)=I_{\nu\pm1}(r)\pm\frac{\nu}{r}I_{\nu}(r)$ \cite[Chapter 10]{dlmf}
(See Section \ref{sec:Continuum-Analysis:-Painlev} for useful formulae,
and note also that $\partial_{\bar{z}}=\frac{1}{2}e^{i\theta}\left(\partial_{r}+ir^{-1}\partial_{\theta}\right)$).

As a special case, we have $Z_{-\frac{1}{2}}^{1}(z)=\frac{e^{2m\left|z\right|}}{\sqrt{z}}$.
\end{rem}

\begin{prop}
\label{prop:contbvp}The following boundary value problem on a bounded
smooth simply connected domain $\Omega$ has at most one solution:

$f:\left[\Omega,a_{1},\ldots,a_{n}\right]\to\mathbb{C}$ satisfies
\begin{enumerate}
\item $f$ is continuously differentiable, square integrable, and $\partial f=m\bar{f}$
in $\left[\Omega,a_{1},\ldots,a_{n}\right]$,
\item $f(z)\in\sqrt{\nu_{out}^{-1}(z)}\mathbb{R}$ on the lift of $\partial\Omega$,
where $\nu_{out}(z)$ is the outer normal at $\pi(z)\in\partial\Omega$,
\item $(z-a_{1})^{1/2}f(z)\to1$ as $z\to a_{1}$, and
\item $\exists\mathcal{B}_{j}\in\mathbb{R}$ such that $(z-a_{j})^{1/2}f(z)\to i\mathcal{B}_{j}$
as $z\to a_{j}$ for $j=2,\ldots,n$.
\end{enumerate}
\end{prop}

\begin{proof}
If there are any two such functions $f_{1},f_{2}$, applying Green-Riemann's
formula to their difference $\hat{f}$ on $\Omega_{r}=\Omega\setminus\cup_{j}B_{r}(a_{j})$
for small $r>0$ yields
\begin{equation}
\oint_{\partial\Omega_{r}}\hat{f}{}^{2}dz=2i\iint_{\Omega_{r}}\partial_{\bar{z}}\hat{f}{}^{^{2}}=2i\iint_{\Omega_{r}}2m|\hat{f}|^{2}dz\in i\mathbb{R}.\label{eq:green-riemann}
\end{equation}

$\hat{f}\sim cst\cdot(z-a_{1})^{1/2}$ near $a_{1}$ in view of Corollary
\ref{cor:powerseries} so $\oint_{\partial B_{r}(a_{1})}\hat{f}^{2}dz$
tends to zero as $r\to0$,  and $\hat{f}\to i\hat{\mathcal{B}_{j}}(z-a_{j})^{-1/2}$
for some $\hat{\mathcal{B}_{i}}\in\mathbb{R}$ as $z$ tends to any
other $a_{j}$. On the inner circles, we have
\begin{alignat}{1}
\oint_{\partial B_{r}(a_{j})}\hat{f}{}^{2}dz\xrightarrow{r\to0} & -2\pi i\hat{\mathcal{B}_{j}^{2}}\in i\mathbb{R}\text{ for }j=2,\ldots,n,\label{eq:innermonodromy}
\end{alignat}
whereas the boundary condition on the lift of $\partial\Omega$ readily
gives $\frac{1}{i}\oint_{\partial\Omega}\hat{f}{}^{2}dz=\oint_{\partial\Omega}|\hat{f}|^{2}ds\geq0$.
Therefore 
\[
0\leq\frac{1}{i}\oint_{\partial\Omega}\hat{f}{}^{2}dz=2\iint_{\Omega_{r}}2m|\hat{f}|^{2}dz-\sum_{j}2\pi\hat{\mathcal{B}_{j}}^{2}\leq0,
\]
so $\hat{f}\equiv0$.
\end{proof}
\begin{rem}
\label{rem:hbvp}As in the proof of the Proposition \ref{prop:contbvp},
it is easy to show that $h:=\text{Re}\int f^{2}dz$ is globally well-defined
in $\Omega\setminus\left\{ a_{1},\ldots,a_{n}\right\} $. In terms
of $h$, the boundary condition $f(z)\in\sqrt{\nu_{out}^{-1}}\mathbb{R}$
and the asymptotics around $a_{i}$ is equivalent to 
\begin{enumerate}
\item $h$ is constant on $\partial\Omega$ and there is no $x_{0}\in\partial\Omega$
such that $h>h(\partial\Omega)$ in a neighbourhood of $x_{0}$, and
\item $h$ is bounded below near $a_{2},\ldots,a_{n}$ and $h^{\dagger}(w):=\text{Re}\int^{w}\left(f(z)-Z_{-\frac{1}{2}}^{1}(z-a_{1})\right)^{2}dz$
is single valued and bounded near $a_{1}$.
\end{enumerate}
We note that this boundary problem can easily be extended to the case
where $\Omega=\mathbb{C}$, by requiring that $h$ be finite and constant
at infinity (we will in fact require that $f$ decays exponentially
fast, see also Lemma \ref{lem:one-point}).
\end{rem}

We now define quantities which reflect the geometry of $\Omega$ exploiting
the boundary value problem above, which will turn out to be directly
related to the Ising correlations through the connection which is
precisely our main convergence result in Section \ref{subsec:Analysis-near-the}.
Determination of these quantities through isomonodromic deformation
is the main subject of Section \ref{sec:Continuum-Analysis:-Painlev}.
\begin{defn}
\label{def:A-powerseries}Given a solution $f_{\left[\Omega,a_{1},\ldots,a_{n}\right]}$
of the boundary value problem presented in Proposition \ref{prop:contbvp}
(the \emph{continuous massive fermion}), define $\mathcal{A}_{\Omega}^{1,i}=\mathcal{A}_{\Omega}^{1,i}\left(a_{1},\ldots,a_{n}|m\right)$
as the real coefficients in the expansion
\begin{equation}
f_{\left[\Omega,a_{1},\ldots,a_{n}\right]}(z)=Z_{-\frac{1}{2}}^{1}(z-a_{1})+2\mathcal{A}_{\Omega}^{1}Z_{\frac{1}{2}}^{1}(z-a_{1})+2\mathcal{A}_{\Omega}^{i}Z_{\frac{1}{2}}^{i}(z-a_{1})+O\left((z-a)^{3/2}\right).\label{eq:A}
\end{equation}

In addition, in the case where $n=2$, define $\mathcal{B}_{\Omega}=\mathcal{B}_{\Omega}\left(a_{1},a_{2}|m\right)$
as the coefficient
\begin{equation}
f_{\left[\Omega,a_{1},a_{2}\right]}(z)=\mathcal{B}_{\Omega}Z_{-\frac{1}{2}}^{i}(z-a_{1})+O\left((z-a)^{1/2}\right).\label{eq:B}
\end{equation}

For notational convenience, we do not assume $\mathcal{B}_{\Omega}>0$
as in \cite{chelkak-hongler}; instead, its sign depends on the sheet
choice of $Z_{-\frac{1}{2}}^{i}(z-a_{1})$, which we will explicitly
fix whenever needed.
\end{defn}

\section{Discrete Analysis: Scaling Limit\label{sec:Discrete-Analysis:-Scaling}}

In this section, we show the convergence of the discrete fermions
introduced in Section \ref{sec:Massive-Discrete-Fermions} to their
continuous counterparts. Then we show that the family of discrete
fermions as $\delta\downarrow0$ is precompact in Section \ref{subsec:convergence},
whose limit satisfies a unique characterisation as laid out in Proposition
\ref{prop:contbvp}. These suffice to show convergence to the desired
limit. Throughout this section, we assume a continuous mass $m<0$
is fixed.

\subsection{Bulk Convergence\label{subsec:convergence}}

We finally undertake convergence of the discrete fermion $F_{\left[\Omega_{\delta},a_{1},\ldots,a_{n}\right]}$
to its continuous counterpart $f_{\left[\Omega,a_{1},\ldots,a_{n}\right]}$
in scaling limit. First, we need to interpolate the discrete function
defined on $\left[\Omega_{\delta},a_{1},\ldots,a_{n}\right]$ on the
continuous domain $\left[\Omega,a_{1},\ldots,a_{n}\right]$. While
any reasonable interpolation (e.g. linear interpolation used in many
papers dealing with massless case) should converge to the unique continuous
limit, we will assume an interpolation scheme with a continuously
differentiable $F_{\left[\Omega_{\delta},a_{1},\ldots,a_{n}\right]}$
as an easy way to show that the limit itself is continuous differentiable.
While we do not explicitly carry it out, we could show the limit is
smooth by using arbitrarily more regular interpolation scheme.
\begin{prop}
\label{prop:precompactness}Suppose $\Omega$ is bounded, simply connected
with smooth boundary or $\mathbb{C}$. For any compact subset $K\subset\Omega$,
any infinite collection $\left(\frac{2}{\pi}\delta_{k}\right)^{-1/2}F_{\left[\Omega_{\delta_{k}},a_{1},\ldots,a_{n}\right]}=:f_{\left[\Omega_{\delta_{k}},a_{1},\ldots,a_{n}\right]}$
with $\delta_{k}\downarrow0$ has a subsequence that (suitably interpolated
as above) converges in $C^{1}(K)$-topology to a continuously differentiable
limit.
\end{prop}

\begin{proof}
By Arzelà-Ascoli, it suffices to show that the discrete derivatives
\begin{alignat*}{1}
\delta^{-1}\partial_{x}^{\delta}f_{\left[\Omega_{\delta},a_{1},\ldots,a_{n}\right]}(z) & :=\delta^{-1}\left[f_{\left[\Omega_{\delta},a_{1},\ldots,a_{n}\right]}(z+2\delta)-f_{\left[\Omega_{\delta},a_{1},\ldots,a_{n}\right]}(z)\right];\\
\delta^{-1}\partial_{y}^{\delta}f_{\left[\Omega_{\delta},a_{1},\ldots,a_{n}\right]}(z) & :=\delta^{-1}\left[f_{\left[\Omega_{\delta},a_{1},\ldots,a_{n}\right]}(z+2\delta i)-f_{\left[\Omega_{\delta},a_{1},\ldots,a_{n}\right]}(z)\right],
\end{alignat*}
 are equicontinuous on $K$. In view of Proposition \ref{prop:hmestimate},
it suffices to show that the 'discrete $L^{2}$ norm' $\sum_{c\in\mathcal{C}\left(K_{\delta}\right)}\left|f_{\left[\Omega_{\delta},a_{1},\ldots,a_{n}\right]}(c)\right|^{2}\delta^{2}$
is bounded by a universal number.

Apply the discrete Green's formula as in (\ref{eq:discrete-greens-formula})
to $H_{\left[\Omega_{\delta},a_{1},\ldots,a_{n}\right]}$, then we
get the following analogue to (\ref{eq:unbounded_ineq}) (dividing
both sides by $\delta$),
\begin{alignat}{1}
cst & \geq\delta^{-1}\Delta^{\delta}H_{\left[\Omega_{\delta},a_{1},\ldots,a_{n}\right]}^{\bullet}(a_{1}+\delta)\label{eq:precomp_lap}\\
 & \geq2\sin\left(\frac{\pi}{4}-\Theta\right)\sum_{\begin{array}{c}
c\in\mathcal{C}\left[\Omega_{\delta}\right]\\
c\neq a_{1}+\delta+\frac{i^{m}\delta}{2}
\end{array}}\delta A_{-\Theta}\left|f_{\left[\Omega_{\delta},a_{1},\ldots,a_{n}\right]}\left(c\right)\right|^{2}\nonumber \\
 & +\sum_{e\in\partial\mathcal{E}\left[\Omega_{\delta}\right]}2\delta\sin\left(\frac{\pi}{4}-2\Theta\right)\left|f_{\left[\Omega_{\delta},a_{1},\ldots,a_{n}\right]}\left(e\right)\right|^{2},\nonumber \\
 & \geq2\sin\left(\frac{\pi}{4}-\Theta\right)\sum_{\begin{array}{c}
c\in\mathcal{C}\left[\Omega_{\delta}\right]\\
c\neq a_{1}+\delta+\frac{i^{m}\delta}{2}
\end{array}}\delta A_{-\Theta}\left|f_{\left[\Omega_{\delta},a_{1},\ldots,a_{n}\right]}\left(c\right)\right|^{2}
\end{alignat}
and since $A_{-\Theta}\sim-2\sqrt{2}m\delta$, we have the desired
$L^{2}$ bound from the sum of $\delta A_{-\Theta}\left|f_{\left[\Omega_{\delta},a_{1},\ldots,a_{n}\right]}\left(c\right)\right|^{2}$.
\end{proof}
With a sequence $K_{m}$ of increasing compact subsets such that $\bigcup_{m}K_{m}=\Omega\setminus\left\{ a_{1},\ldots,a_{n}\right\} $
and using diagonalisation, we can find a global subsequential limit
$f_{\left[\Omega,a_{1},\ldots,a_{n}\right]}$ with uniform convergence
in compact subsets of $\Omega$. We finish the proof of convergence
by showing that a limit must satisfy the boundary value problem of
Proposition \ref{prop:contbvp}, and thus is unique. We note that
continuous differentiability, square integrability and the condition
$\partial_{\bar{z}}f_{\left[\Omega,a_{1},\ldots,a_{n}\right]}=m\overline{f_{\left[\Omega,a_{1},\ldots,a_{n}\right]}}$
follows straightforwardly from Proposition \ref{prop:Appendix} and
Remark \ref{rem:mdhol}, so we are left to verify the boundary conditions,
as laid out in Remark \ref{rem:hbvp}.

We first treat the following explicit case:
\begin{lem}
\label{lem:one-point}The one point spinor $f_{\left[\mathbb{C}_{\delta},0\right]}$
converges to $Z_{-\frac{1}{2}}^{1}(z=re^{i\theta})=\frac{e^{2mr}}{\sqrt{z}}$
uniformly in compact subsets of $\left[\mathbb{C}_{\delta},0\right]$.
\end{lem}

\begin{proof}
We show that a subsequential limit $f_{\left[\mathbb{C},0\right]}$
satisfies the unique characterisation of Remark \ref{rem:hbvp}, and
thus has to be equal to $Z_{-\frac{1}{2}}^{1}$, an explicit solution.
We first show that $f_{\left[\mathbb{C},0\right]}$ vanishes at infinity
sufficiently fast to yield that $h_{\left[\mathbb{C},0\right]}:=\text{Re}\int f_{\left[\mathbb{C},0\right]}^{2}dz$
is constant at infinity. But this follows from the identification
of the discrete spinor with the massive harmonic measure (the hitting
probability of an extinguished random walk) of the tip of a slit in
Proposition \ref{prop:onepointspinor} and the exponential estimates
of Proposition \ref{prop:hmestimate}.

It now remains to identify the singularity at $0$ as $f_{\left[\mathbb{C},0\right]}\sim z{}^{-1/2}$.
Note that the massless harmonic measure \cite{kesten,lali2004} and
thus the spinor \cite[Lemma 2.14]{chelkak-hongler} (for the identification
$\vartheta(\delta)=\left(\frac{2}{\pi}\delta\right)^{1/2}$, see \cite[Lemma 5.14]{ghp})
show that exact behaviour. Since the hitting probability of a massive
random walk is dominated by the hitting probability of a simple random
walk, $f_{\left[\mathbb{C},0\right]}\cdot z^{1/2}\xrightarrow{z\to0}\alpha$
for some $\alpha\in\left[0,1\right]$. But the probability that the
massive random walk is extinguished can be made arbitrarily close
to $0$ as length scale becomes negligible compared to $\frac{1}{\left|m\right|}$,
so the two harmonic measures become asymptotically equal as $z\to0$
and we conclude $\alpha=1$.
\end{proof}
By translation invariance of its definition, $f_{\left[\mathbb{C}_{\delta},a_{1}\right]}(z)$
converges to $Z_{-\frac{1}{2}}^{1}(z-a_{1})$ uniformly in compact
subsets of $\left[\mathbb{C},a_{1}\right]$.
\begin{prop}
\label{prop:ident_limit}A subsequential limit $f_{\left[\Omega,a_{1},\ldots,a_{n}\right]}$
of $f_{\left[\Omega_{\delta},a_{1},\ldots,a_{n}\right]}$ satisfies
the conditions of Remark \ref{rem:hbvp}, and thus is unique.
\end{prop}

\begin{proof}
We first suppose $\Omega\neq\mathbb{C}$. Consider the renormalised
discrete square integral $h_{\left[\Omega_{\delta},a_{1},\ldots,a_{n}\right]}:=\left(\frac{2}{\pi}\delta\right)^{-1}H_{\left[\Omega_{\delta},a_{1},\ldots,a_{n}\right]}$.
By Propositions \ref{prop:squareint} and \ref{prop:unbounded-prop},
$h_{\left[\Omega_{\delta},a_{1},\ldots,a_{n}\right]}^{\bullet}=\delta^{-1}H_{\left[\Omega_{\delta},a_{1},\ldots,a_{n}\right]}^{\bullet}$
is discrete superharmonic on $\mathcal{V}\left[\Omega_{\delta}\right]\setminus\left\{ a_{1}+\delta\right\} $
and takes the boundary value $0$ on $\partial\mathcal{V}\left[\Omega_{\delta}\right]$.
Moreover, by Proposition \ref{prop:precompactness}, $\Delta^{\delta}h_{\left[\Omega_{\delta},a_{1},\ldots,a_{n}\right]}^{\bullet}(a_{1}+\delta)\leq cst=:c_{0}$.
Thus, by superharmonicity, $h_{\left[\Omega_{\delta},a_{1},\ldots,a_{n}\right]}^{\bullet}$
is lower bounded by $c_{0}\mathcal{G}_{\delta}$, where $\mathcal{G}_{\delta}$
is the domain Green's function $\mathcal{G}_{\delta}:=\mathcal{G}_{\mathcal{V}\left[\Omega_{\delta}\right]}(\cdot,a_{1}+\delta)$
with Dirichlet boundary condition on $\partial\mathcal{V}\left[\Omega_{\delta}\right]$.

Since $\mathcal{G}_{\delta}$ is $O(\delta)$ on any vertex $a_{\text{int}}\in\mathcal{V}\left[\Omega_{\delta}\right]$
adjacent to $a\in\mathcal{V}\left[\Omega_{\delta}\right]$ (Lemma
\ref{lem:harm-boundary}), we have that $O(\delta)=c_{0}\mathcal{G}_{\delta}(a_{\text{int}})\leq h_{\left[\Omega_{\delta},a_{1},\ldots,a_{n}\right]}^{\bullet}(a_{\text{int}})\leq0$
and $h_{\left[\Omega_{\delta},a_{1},\ldots,a_{n}\right]}^{\bullet}(a_{\text{int}})=-\left|f_{\left[\Omega_{\delta},a_{1},\ldots,a_{n}\right]}\left(\frac{a_{\text{int}}+a}{2}\right)\right|^{2}\delta=O(\delta)$
by (\ref{eq:squareint_outernormal}). We see that therefore $f_{\left[\Omega_{\delta},a_{1},\ldots,a_{n}\right]}$
is uniformly bounded on boundary edges and corners (see Figure \ref{fig:The-square-lattice.}).
Since $\left|f\right|^{2}$ is subharmonic by (\ref{eq:f2-subharm}),
$f_{\left[\Omega_{\delta},a_{1},\ldots,a_{n}\right]}$ is uniformly
bounded on $\Omega_{r}:=\Omega\setminus\bigcup_{j}B_{r}(a_{j})$ for
small fixed $r>0$. Therefore by equicontinuity $h_{\left[\Omega_{\delta},a_{1},\ldots,a_{n}\right]}\to h_{\left[\Omega,a_{1},\ldots,a_{n}\right]}$
uniformly on $\overline{\Omega_{r}}$, and $h_{\left[\Omega,a_{1},\ldots,a_{n}\right]}$
is continuous up to the boundary. Note that since $\partial_{\bar{z}}f_{\left[\Omega,a_{1},\ldots,a_{n}\right]}=m\overline{f_{\left[\Omega,a_{1},\ldots,a_{n}\right]}}$,
$h_{\left[\Omega,a_{1},\ldots,a_{n}\right]}=\text{Re}\int f_{\left[\Omega,a_{1},\ldots,a_{n}\right]}^{2}dz$
satisfies 
\begin{equation}
\Delta h_{\left[\Omega,a_{1},\ldots,a_{n}\right]}=4\partial_{\bar{z}}\partial_{z}h_{\left[\Omega,a_{1},\ldots,a_{n}\right]}=2\partial_{\bar{z}}f_{\left[\Omega,a_{1},\ldots,a_{n}\right]}^{2}=4m\left|f_{\left[\Omega,a_{1},\ldots,a_{n}\right]}\right|^{2}.\label{eq:continuoushlap}
\end{equation}

We now verify the two remaining conditions of the boundary value problem.
\begin{enumerate}
\item $h_{\left[\Omega,a_{1},\ldots,a_{n}\right]}$ in $\Omega_{r}$ is
a continuous solution of a Poisson equation with bounded data $4m\left|f_{\left[\Omega,a_{1},\ldots,a_{n}\right]}\right|^{2}$
and smooth boundary data; $h_{\left[\Omega,a_{1},\ldots,a_{n}\right]}$
is continuously differentiable up to the boundary thanks to standard
Green's function estimates (e.g. \cite[Theorem 4.3]{GiTr}). From
\cite[Remark 6.3]{chsm2012}, which only uses the superharmonicity
(in their convention, subharmonicity) of $h_{\left[\Omega_{\delta},a_{1},\ldots,a_{n}\right]}^{\bullet}$
and does not depend on a particular property of $h_{\left[\Omega_{\delta},a_{1},\ldots,a_{n}\right]}^{\circ}$,
we see that there is no neighbourhood of $x_{0}\in\partial\Omega$
where $h>0$.
\item $h_{\left[\Omega_{\delta},a_{1},\ldots,a_{n}\right]}$ is bounded
below by $c\mathcal{G}_{\delta}$, which is bounded from negative
infinity away from $a_{1}$ in the scaling limit. So near $a_{2},\ldots,a_{n}$,
$h_{\left[\Omega,a_{1},\ldots,a_{n}\right]}$ is bounded below. For
the asymptotic near $a_{1}$, note that $f_{\left[\Omega_{\delta},a_{1},\ldots,a_{n}\right]}^{\dagger}:=f_{\left[\Omega_{\delta},a_{1},\ldots,a_{n}\right]}-f_{\left[\mathbb{C}_{\delta},a_{1}\right]}$
is s-holomorphic near $a_{1}$ with $f_{\left[\Omega_{\delta},a_{1},\ldots,a_{n}\right]}^{\dagger}(a_{1}+\frac{\delta}{2})=0$,
so by Proposition \ref{prop:Appendix} it is everywhere massive harmonic
(unlike in the proof of Proposition \ref{prop:onepointspinor}, the
zero prevents a singularity near monodromy; see also \cite[Remark 2.6]{ghp}).
But both $f_{\left[\Omega_{\delta},a_{1},\ldots,a_{n}\right]},f_{\left[\mathbb{C}_{\delta},a_{1}\right]}$
are uniformly bounded, say, in a discrete circle $S_{r}:=B_{r+5\delta}(a_{1})\setminus B_{r}(a_{1})$
for small $r>0$. As above, we conclude that $f_{\left[\Omega_{\delta},a_{1},\ldots,a_{n}\right]}^{\dagger}$
is uniformly bounded in $B_{r}(a_{1})$. Given Lemma \ref{lem:one-point},
this suffices to show boundedness and well-definedness for the continuous
limits $f^{\dagger},h^{\dagger}$.
\end{enumerate}
If $\Omega=\mathbb{C}$, we first show that $f_{\left[\Omega_{\delta},a_{1},\ldots,a_{n}\right]}\to0$
at infinity uniformly in $\delta$ and $h_{\left[\Omega_{\delta},a_{1},\ldots,a_{n}\right]}$
is constant and finite at infinity using Proposition \ref{prop:hmestimate}
and the uniform $L^{2}$ bound (\ref{eq:precomp_lap}). Since $f_{\left[\Omega_{\delta},a_{1},\ldots,a_{n}\right]}\to0$
at infinity uniformly in $\delta$, we may compare with the Green's
function of a suitably large ball to obtain lower boundedness around
$a_{2},\ldots,a_{n}$. The above proof applies near $a_{1}$.
\end{proof}

\subsection{Analysis near the Singularity\label{subsec:Analysis-near-the}}

To show convergence of $F_{\left[\Omega_{\delta},a_{1},\ldots,a_{n}\right]}\left(a_{1}+\frac{3\delta}{2}\right)$,
we model the expansion of $f_{\left[\Omega,a_{1},\ldots,a_{n}\right]}$
around $a_{1}$ from Definition \ref{def:A-powerseries} in the discrete
setting, and carefully analyse the magnitude of the difference. The
candidate for discrete $Z_{-\frac{1}{2}}^{1}$ is clear: $F_{\left[\mathbb{C}_{\delta},a_{1}\right]}$
is able to cancel out the singularity of $F_{\left[\Omega_{\delta},a_{1},\ldots,a_{n}\right]}$,
and converges to $Z_{-\frac{1}{2}}^{1}$ by Lemma \ref{lem:one-point}.

Then, we need to build a discrete analogue of $Z_{\frac{1}{2}}^{1}$
(or rather, some massive s-holomorphic function that has square root
behaviour around $0$). Following \cite[(3.12)]{chelkak-hongler},
we define a discrete function $G_{\left[\mathbb{C}_{\delta},a_{1}\right]}$
by discrete integrating the $F_{\left[\mathbb{C}_{\delta},a_{1}\right]}$.
\begin{prop}
\label{prop:squareroot}Construct the spinor $G_{\left[\mathbb{C}_{\delta},a_{1}\right]}:\mathcal{C}^{1}\left[\Omega_{\delta},a_{1}\right]\to\mathbb{R}$
by
\begin{equation}
G_{\left[\mathbb{C}_{\delta},a_{1}\right]}\left(z\right):=\delta\sum_{j=0}^{\infty}\Gamma^{j}F_{\left[\mathbb{C}_{\delta},a_{1}\right]}(z-2j\delta),\label{eq:squareroot}
\end{equation}
where $\Gamma(\delta|\Theta):=\tan^{2}\left(\frac{\pi}{4}+2\Theta\right)$,
and $z-2j\delta$ is taken on the same sheet as $z$ if $\pi(z)\notin a_{1}+\mathbb{R}_{>0}$
or if $\pi(z),\pi\left(z-2j\delta\right)\in a_{1}+\mathbb{R}_{>0}$,
while $F_{\left[\mathbb{C}_{\delta},a_{1}\right]}(z-2j\delta)=0$
naturally as soon as $\pi\left(z-2j\delta\right)\in a_{1}+\mathbb{R}_{<0}$
(cf. Proposition \ref{prop:onepointspinor}).

It is massive harmonic with coefficient $M_{H}^{2}$ on $\mathbb{X}^{+}\cap\mathcal{C}^{1}\left[\mathbb{C}_{\delta},0\right]$,
and $\left(\frac{2}{\pi}\delta\right)^{-1/2}G_{\left[\mathbb{C}_{\delta},a_{1}\right]}$
converges uniformly in compact subsets of $\left[\mathbb{C},a_{1}\right]$
to $g_{\left[\mathbb{C},a_{1}\right]}$ which has the asymptotic behaviour
$g_{\left[\mathbb{C},a_{1}\right]}(z)\sim\text{Re}\sqrt{z-a_{1}}$
near $a_{1}$.
\end{prop}

\begin{proof}
We first show that the discrete integrand $\Gamma^{j}F_{\left[\mathbb{C}_{\delta},a_{1}\right]}(z-2j\delta)$
is summable. Without loss of generality, work on the sheet where $F_{\left[\mathbb{C}_{\delta},a_{1}\right]}(z-2j\delta)>0$.
By Proposition \ref{prop:onepointspinor}, $F_{\left[\mathbb{C}_{\delta},a_{1}\right]}(z-2j\delta)$
is the probability of a massive random walk started at the real corner
$z-2j\delta$ and extinguished at each step with probability $\left(1+\frac{2\sin^{2}2\Theta}{\cos(4\Theta)}\right)^{-1}\frac{2\sin^{2}2\Theta}{\cos(4\Theta)}$
surviving to hit $a_{1}+\frac{3\delta}{2}$ before $a_{1}+\mathbb{R}_{<0}$;
denote by $h_{j}$ the probability of the same event with $\Theta=0$,
i.e. the massless harmonic measure. Project the two-dimensional walk
into two independent one-dimensional walks respectively in the $x,y$-directions
with step lengths $\delta$. If $p_{2j}$ is the probability of the
one-dimensional massive random walk started at $x-2j\delta$ surviving
to hit $x$, we have the upper bound
\[
\Gamma^{j}F_{\left[\mathbb{C}_{\delta},a_{1}\right]}(z-2j\delta)\leq\Gamma^{j}p_{2j}h_{j}.
\]
However, $p_{2j}$ is the massive harmonic measure of $0$ as seen
from $-2j$ in $\mathbb{Z}_{\leq0}$. From the boundary conditions
$p_{0}=1$, $p_{-\infty}=0$ and massive harmonicity in 1D $p_{j-2}+p_{j}-2p_{j-1}=\frac{4\sin^{2}2\Theta}{\cos4\Theta}p_{j-1}$,
it is straightforward to verify $p_{j}=\Gamma^{-j/2}$, and we have
$\Gamma^{j}F_{\left[\mathbb{C}_{\delta},a_{1}\right]}(z-2j\delta)\leq h_{j}$.
By \cite[Lemma 3.4]{chelkak-hongler}, $h_{j}=O(j^{-3/2})$ uniformly
for $z$ in a compact subset, and thus the sum is finite.

We know that $\left(\frac{2}{\pi}\delta\right)^{-1/2}F_{\left[\mathbb{C}_{\delta},a_{1}\right]}(z')$
converges to $\text{Re}\frac{e^{2mr'}}{\sqrt{z'}}$ on $\mathcal{C}^{1}\left[\Omega_{\delta},a_{1}\right]$
uniformly in compact subsets away from $a_{1}$. Also note that in
the scaling limit $x-2j\delta=x'$, $\Gamma^{j}$ converges to $e^{-2m(x-x')}$.
By above, if $z=x+iy$ is in a compact subset, the discrete integrand
$\Gamma^{\frac{x-x'}{2\delta}}F_{\left[\mathbb{C}_{\delta},a_{1}\right]}(z')$
decays to zero as $x'\to-\infty$ uniformly in $z$, so we conclude
that $g_{\left[\mathbb{C}_{\delta},a_{1}\right]}:=\left(\frac{2}{\pi}\delta\right)^{-1/2}G_{\left[\mathbb{C}_{\delta},a_{1}\right]}$
converges
\begin{equation}
g_{\left[\mathbb{C}_{\delta},a_{1}\right]}\left(z=x+iy\right)\xrightarrow{\delta\to0}\frac{1}{2}\int_{-\infty}^{x}e^{2m(x'-x)}\text{Re}\frac{e^{2mr'}}{\sqrt{x'+iy'}}dx',\label{eq:sqroot_int}
\end{equation}
uniformly for $z$ in a compact subset. As $z\to0$, we can uniformly
bound $e^{2m(x'-x)}e^{2mr'}$ close to $1$, which gives the asymptotic
of $g_{\left[\mathbb{C},a_{1}\right]}$ near $z$.

Massive harmonicity is clear unless $\pi(z)\in\mathbb{R}_{>0}$. If
$\pi(z)\in a_{1}+\mathbb{R}_{>0}$, $\left(\Delta^{\delta}-M_{H}^{2}\right)F_{\left[\mathbb{C}_{\delta},a_{1}\right]}(z-2j\delta)=0$
if $\pi(z-2j\delta)\in a_{1}+\frac{3\delta}{2}+\mathbb{R}_{\geq0}$
while $F_{\left[\mathbb{C}_{\delta},a_{1}\right]}(z-2j\delta)=0$
if $\pi(z-2j\delta)\in a_{1}+\mathbb{R}_{<0}$, thus
\begin{alignat*}{1}
\left(\Delta^{\delta}-M_{H}^{2}\right)G_{\left[\mathbb{C}_{\delta},a_{1}\right]}\left(z\right) & =\delta\sum_{j=\left\lfloor (z-a_{1})/2\delta\right\rfloor }^{\infty}\Gamma^{j}\left(\Delta^{\delta}-M_{H}^{2}\right)F_{\left[\mathbb{C}_{\delta},a_{1}\right]}(z-2j\delta)\\
 & =\delta\Gamma^{\left\lfloor (z-a_{1})/2\delta\right\rfloor }\sum_{j=0}^{\infty}\Gamma^{j}\left(\Delta^{\delta}-M_{H}^{2}\right)F_{\left[\mathbb{C}_{\delta},a_{1}\right]}(a_{1}+\frac{3\delta}{2}-2j\delta),
\end{alignat*}
where the laplacian is taken on the planar slit domain $\mathbb{X}^{+}\cap\mathcal{C}^{1}\left[\mathbb{C}_{\delta},0\right]$
with zero boundary values on the slit.

We need to show that the last sum vanishes.
\begin{alignat*}{1}
 & \sum_{j=0}^{N}\Gamma^{j}\left(\Delta^{\delta}-M_{H}^{2}\right)F_{\left[\mathbb{C}_{\delta},a_{1}\right]}(a_{1}+\frac{3\delta}{2}-2j\delta)\\
= & \sum_{j=0}^{N}\Gamma^{j}\sum_{s=\pm1}\left[F_{\left[\mathbb{C}_{\delta},a_{1}\right]}(a_{1}+\frac{3\delta}{2}-(2j+s)\delta+i\delta)-\tan\left(\frac{\pi}{4}-s2\Theta\right)F_{\left[\mathbb{C}_{\delta},a_{1}\right]}(a_{1}+\frac{3\delta}{2}-2j\delta)\right]\\
+ & \sum_{j=0}^{N}\Gamma^{j}\sum_{s=\pm1}\left[F_{\left[\mathbb{C}_{\delta},a_{1}\right]}(a_{1}+\frac{3\delta}{2}-(2j+s)\delta-i\delta)-\tan\left(\frac{\pi}{4}+s2\Theta\right)F_{\left[\mathbb{C}_{\delta},a_{1}\right]}(a_{1}+\frac{3\delta}{2}-2j\delta)\right],
\end{alignat*}
where two sums are done respectively above and below the slit. By
massive discrete holomorphicity, we can convert the differences of
real corner values into differences of imaginary corner values: we
need to be careful, since the points in $\mathbb{X}^{+}$ directly
above and below the cut are on opposite sheets. We can in fact think
of $a_{1}+\frac{\delta}{2}$ as lying on the slit, since the singularity
(\ref{eq:singularity}) ascribes two values at $a_{1}+\frac{\delta}{2}$,
coming from above and below $a_{1}+\mathbb{R}$. Above the slit, (\ref{eq:mdhol})
implies
\begin{alignat*}{1}
 & \sum_{s=\pm1}\left[F_{\left[\mathbb{C}_{\delta},a_{1}\right]}(a_{1}+\frac{3\delta}{2}-(2j+s)\delta+i\delta)-\tan\left(\frac{\pi}{4}-s2\Theta\right)F_{\left[\mathbb{C}_{\delta},a_{1}\right]}(a_{1}+\frac{3\delta}{2}-2j\delta)\right]\\
= & -i\left[F_{\left[\mathbb{C}_{\delta},a_{1}\right]}(a_{1}+\frac{3\delta}{2}-2j\delta+i\delta)-\tan\left(\frac{\pi}{4}+2\Theta\right)F_{\left[\mathbb{C}_{\delta},a_{1}\right]}(a_{1}+\frac{3\delta}{2}-(2j-1)\delta)\right]\\
 & +i\left[F_{\left[\mathbb{C}_{\delta},a_{1}\right]}(a_{1}+\frac{3\delta}{2}-2j\delta+i\delta)-\tan\left(\frac{\pi}{4}-2\Theta\right)F_{\left[\mathbb{C}_{\delta},a_{1}\right]}(a_{1}+\frac{3\delta}{2}-(2j+1)\delta)\right]\\
= & i\tan\left(\frac{\pi}{4}+2\Theta\right)F_{\left[\mathbb{C}_{\delta},a_{1}\right]}(a_{1}+\frac{3\delta}{2}-(2j-1)\delta)-i\tan\left(\frac{\pi}{4}-2\Theta\right)F_{\left[\mathbb{C}_{\delta},a_{1}\right]}(a_{1}+\frac{3\delta}{2}-(2j+1)\delta).
\end{alignat*}
Thanks to the factor of $\Gamma^{j}$, the sum telescopes (see Figure
\ref{fig:Using-holomorphicity-to})
\begin{alignat*}{1}
 & \sum_{j=0}^{N}\Gamma^{j}\sum_{s=\pm1}\left[F_{\left[\mathbb{C}_{\delta},a_{1}\right]}(a_{1}+\frac{3\delta}{2}-(2j+s)\delta+i\delta)-\tan\left(\frac{\pi}{4}-s2\Theta\right)F_{\left[\mathbb{C}_{\delta},a_{1}\right]}(a_{1}+\frac{3\delta}{2}-2j\delta)\right]\\
= & i\tan\left(\frac{\pi}{4}+2\Theta\right)F_{\left[\mathbb{C}_{\delta},a_{1}\right]}(a_{1}+\frac{5\delta}{2})-i\Gamma^{N}\tan\left(\frac{\pi}{4}-2\Theta\right)F_{\left[\mathbb{C}_{\delta},a_{1}\right]}(a_{1}+\frac{3\delta}{2}-(2N+1)\delta).
\end{alignat*}

The sum below the slit analogously gives
\begin{alignat*}{1}
 & \sum_{j=0}^{N}\Gamma^{j}\sum_{s=\pm1}\left[F_{\left[\mathbb{C}_{\delta},a_{1}\right]}(a_{1}+\frac{3\delta}{2}-(2j+s)\delta-i\delta)-\tan\left(\frac{\pi}{4}+s2\Theta\right)F_{\left[\mathbb{C}_{\delta},a_{1}\right]}(a_{1}+\frac{3\delta}{2}-2j\delta)\right]\\
=- & i\tan\left(\frac{\pi}{4}+2\Theta\right)F_{\left[\mathbb{C}_{\delta},a_{1}\right]}(a_{1}+\frac{5\delta}{2})+i\Gamma^{N}\tan\left(\frac{\pi}{4}-2\Theta\right)F_{\left[\mathbb{C}_{\delta},a_{1}\right]}(a_{1}+\frac{3\delta}{2}-(2N+1)\delta),
\end{alignat*}
where $a_{1}+\frac{3\delta}{2}-(2N+1)\delta$ is approached from below
the sheet. Summing the two and taking $a_{1}+\frac{3\delta}{2}-(2N+1)\delta$
from above the slit,
\[
\sum_{j=0}^{N}\Gamma^{j}\left(\Delta^{\delta}-M_{H}^{2}\right)F_{\left[\mathbb{C}_{\delta},a_{1}\right]}(a_{1}+\frac{3\delta}{2}-2j\delta)=-2i\Gamma^{N}\tan\left(\frac{\pi}{4}-2\Theta\right)F_{\left[\mathbb{C}_{\delta},a_{1}\right]}(a_{1}+\frac{3\delta}{2}-(2N+1)\delta).
\]

Now, by Proposition \ref{prop:onepointspinor}, $-i\cdot F_{\left[\mathbb{C}_{\delta},a_{1}\right]}(a_{1}+\frac{3\delta}{2}-(2N+1)\delta)$
is the massive harmonic measure of the point $a_{1}+\frac{\delta}{2}$
in the discrete plane $\mathbb{Y}^{+}\cap\mathcal{C}^{i}\left[\mathbb{C}_{\delta},0\right]$
slit by $a_{1}+\mathbb{R}_{>0}$. $-i\cdot\Gamma^{N}F_{\left[\mathbb{C}_{\delta},a_{1}\right]}(a_{1}+\frac{3\delta}{2}-(2N+1)\delta)=O(N^{-1/2})$
since we may bound it by the massless harmonic measure as above, and
the sum decays to zero, as desired.
\end{proof}
\begin{thm}
\label{thm:convergence}As $\delta\to0$, for $a_{1},\ldots,a_{n}$
at a definite distance from the boundary and each other, we have uniformly
\begin{alignat}{1}
F_{\left[\Omega_{\delta},a_{1},\ldots,a_{n}\right]}\left(a_{1}+\frac{3\delta}{2}\right) & =1+2\delta\mathcal{A}_{\Omega_{\delta}}^{1}\left(a_{1},\ldots,a_{n}\right)+o(\delta),\label{eq:thm1}\\
\left|F_{\left[\Omega_{\delta},a_{1},a_{2}\right]}\left(a_{2}+\frac{\delta}{2}\right)\right| & =\left|\mathcal{B}_{\Omega}(a_{1},a_{2})\right|+o(1).\label{eq:thm2}
\end{alignat}
\end{thm}

\begin{proof}
We adapt the strategy of \cite[Subsection 3.5]{chelkak-hongler};
the massive harmonic nature of our functions are hardly visible since
at short length scale massive hitting probabilities approach simple
random walk hitting probabilities. While we try to use the same notation
for corresponding notions where appropriate, the functions we work
with are all massive harmonic.

Note that, due to explicit constructions in Propositions \ref{prop:onepointspinor}
and \ref{prop:squareroot}, (\ref{eq:thm1}) can be written as 
\[
F_{\left[\Omega_{\delta},a_{1},\ldots,a_{n}\right]}\left(a_{1}+\frac{3\delta}{2}\right)=F_{\left[\mathbb{C}_{\delta},a_{1}\right]}\left(a_{1}+\frac{3\delta}{2}\right)+2\mathcal{A}_{\Omega_{\delta}}^{1}\left(a_{1},\ldots,a_{n}\right)G_{\left[\mathbb{C}_{\delta},a_{1}\right]}\left(a_{1}+\frac{3\delta}{2}\right)+o(\delta).
\]

Define the reflection $\mathcal{R}\left(\Omega\right)$ of $\Omega$
across the line $a_{1}+\mathbb{R}$. Then there is a ball $B(a_{1})$
around $a_{1}$ which belongs to $\Omega\cap\mathcal{R}\left(\Omega\right)$.
Denote $\Lambda$ to be the lift of the slit neighbourhood $B(a_{1})\setminus\left[a_{1}+\mathbb{R}_{<0}\right]$
such that both $F_{\left[\Omega_{\delta},a_{1},\ldots,a_{n}\right]},F_{\left[\mathcal{R}\left(\Omega_{\delta}\right),a_{1},\ldots,a_{n}\right]}$
have their fixed lift of the path origin $a_{1}+\frac{\delta}{2}$
on $\Lambda$. By symmetry arguments about the path set $\Gamma(a_{1}+\frac{\delta}{2},z)$
similar to the proof of Proposition \ref{prop:onepointspinor}, we
have $F_{\left[\Omega_{\delta},a_{1},\ldots,a_{n}\right]}\left(a_{1}+\frac{3\delta}{2}\right)=F_{\left[\mathcal{R}\left(\Omega_{\delta}\right),a_{1},\ldots,a_{n}\right]}\left(a_{1}+\frac{3\delta}{2}\right)$,
whereas the boundary values $F_{\left[\Omega_{\delta},a_{1},\ldots,a_{n}\right]}(z)$,$F_{\left[\mathcal{R}\left(\Omega_{\delta}\right),a_{1},\ldots,a_{n}\right]}(z)$
on the slit $\pi(z)\in a_{1}+\mathbb{R}_{<0}$ cancel out to give
$F_{\left[\Omega_{\delta},a_{1},\ldots,a_{n}\right]}(z)+F_{\left[\mathcal{R}\left(\Omega_{\delta}\right),a_{1},\ldots,a_{n}\right]}(z)=0$
for $z\in\mathcal{C}^{1}\left[\Omega_{\delta},a_{1}\right]\cap\partial\Lambda$
(see also \cite[Subsection 3.5]{chelkak-hongler}). Then, restricted
to $\mathcal{C}^{1}\left[\Omega_{\delta},a_{1}\right]\cap\Lambda$,
(recall $g_{\left[\mathbb{C}_{\delta},a_{1}\right]}:=\left(\frac{2}{\pi}\delta\right)^{-1/2}G_{\left[\mathbb{C}_{\delta},a_{1}\right]}$)
\[
K_{\delta}:=\frac{1}{2}\left[f_{\left[\Omega_{\delta},a_{1},\ldots,a_{n}\right]}+f_{\left[\mathcal{R}\left(\Omega_{\delta}\right),a_{1},\ldots,a_{n}\right]}\right]-f_{\left[\mathbb{C}_{\delta},a_{1}\right]}-2\mathcal{A}_{\Omega_{\delta}}^{1}\left(a_{1},\ldots,a_{n}\right)g_{\left[\mathbb{C}_{\delta},a_{1}\right]},
\]
is everywhere massive harmonic with zero boundary values on the slit
$\mathcal{C}^{1}\left[\Omega_{\delta},a_{1}\right]\cap\partial\Lambda$.
We have to show
\[
K_{\delta}(a_{1}+\frac{3\delta}{2})=f_{\left[\Omega_{\delta},a_{1},\ldots,a_{n}\right]}(a_{1}+\frac{3\delta}{2})-1-2\mathcal{A}_{\Omega_{\delta}}^{1}\left(a_{1},\ldots,a_{n}\right)\delta=o(\delta^{1/2}).
\]

Noting the expansion of Definition \ref{def:A-powerseries}, we see
that, on $\Lambda$, away from $a_{1}$,
\[
K_{\delta}\to\frac{1}{2}\left[f_{\left[\Omega,a_{1},\ldots,a_{n}\right]}+f_{\left[\mathcal{R}\left(\Omega\right),a_{1},\ldots,a_{n}\right]}\right]-Z_{-\frac{1}{2}}^{1}-2\mathcal{A}_{\Omega_{\delta}}^{1}\left(a_{1},\ldots,a_{n}\right)\text{Re}\sqrt{z-a_{1}}=o\left(\left(z-a_{1}\right)^{1/2}\right),
\]
so on the discrete circle $S_{r}=B_{r+5\delta}(a_{1})\setminus B_{r}(a_{1})$
for $r>0$, $\max_{S_{r}}\left|K_{\delta}\right|=o(r^{1/2})$ as $r\to0$.
Sharp discrete Beurling estimate (see \cite[(3.4)]{chelkak-hongler}
for the form used here) for harmonic functions may be used to dominate
the value of the massive harmonic function $S_{\delta}$ at $a_{1}+\frac{3\delta}{2}$,
and we have
\[
\left|K_{\delta}(a_{1}+\frac{3\delta}{2})\right|\leq cst\cdot\delta^{1/2}r^{-1/2}\max_{S_{r}}\left|K_{\delta}\right|=cst\cdot\delta^{1/2}o(1),
\]
where $o(1)$ holds for $r\to0$. We conclude the right hand side
is $o(\delta^{1/2})$ as $\delta\to0$.

For (\ref{eq:thm2}), without loss of generality assume $\mathcal{B}_{\Omega}(a_{1},a_{2})>0$.
Define this time $\mathcal{R}\left(\Omega\right)$ as the reflection
of $\Omega$ across $a_{2}+\mathbb{R}$. As above, write $\Lambda$
for the lift of the slit disc $B(a_{2})\setminus\left[a_{2}+\mathbb{R}_{<0}\right]$
such that we simultaneously define $F_{\left[\Omega_{\delta},a_{1},a_{2}\right]}(a_{2}+\frac{\delta}{2})=F_{\left[\mathcal{R}\left(\Omega_{\delta}\right),a_{1},a_{2}\right]}(a_{2}+\frac{\delta}{2})=:i\mathcal{B}_{\delta}$
with $\mathcal{B}_{\delta}>0$. Again, by symmetry, $f_{\left[\Omega_{\delta},a_{1},a_{2}\right]}+f_{\left[\mathcal{R}\left(\Omega_{\delta}\right),a_{1},a_{2}\right]}$
is zero on the boundary $\partial\Lambda\cap\mathcal{C}^{i}\left[\Omega_{\delta},a_{1},a_{2}\right]$,
i.e. the lift of $a_{2}+\mathbb{R}_{<0}$. Define the massive harmonic
measure of the point $a_{1}+\frac{\delta}{2}$ in the lattice $\Lambda\cap\mathcal{C}^{i}\left[\Omega_{\delta},a_{1},a_{2}\right]$
as $W_{\delta}$. Then define
\[
T_{\delta}:=\frac{1}{2}\left[F_{\left[\Omega_{\delta},a_{1},\ldots,a_{n}\right]}+F_{\left[\mathcal{R}\left(\Omega_{\delta}\right),a_{1},\ldots,a_{n}\right]}\right]-i\mathcal{B}_{\delta}W_{\delta},
\]
which is massive harmonic on $\Lambda\cap\mathcal{C}^{i}\left[\Omega_{\delta},a_{1},a_{2}\right]\setminus\left\{ a_{2}+\frac{\delta}{2}\right\} $,
takes the value $0$ on $\partial\Lambda\cap\mathcal{C}^{i}\left[\Omega_{\delta},a_{1},a_{2}\right]$
and $\left\{ a_{2}+\frac{\delta}{2}\right\} $. Since $\frac{1}{2}\left[f_{\left[\Omega_{\delta},a_{1},\ldots,a_{n}\right]}+f_{\left[\mathcal{R}\left(\Omega_{\delta}\right),a_{1},\ldots,a_{n}\right]}\right]$
restricted to the imaginary corners converges to a continuous limit
with asymptotic $i\cdot\text{Re}\frac{\mathcal{B}_{\Omega_{\delta}}(a_{1},a_{2})}{\sqrt{z-a_{2}}}+o(\left(z-a_{2}\right)^{-1/2})$
on the discrete circle $S_{r}=B_{r+5\delta}(a_{2})\setminus B_{r}(a_{2})$,
it is easy to see that unless $\mathcal{B}_{\delta}\to\mathcal{B}_{\Omega}(a_{1},a_{2})$
and $T_{\delta}\left(a_{2}+\frac{\delta}{2}\right)=o(1)$, we can
find a point in the bulk which is greater than any value on the boundary
$S_{r}$, contradicting the maximum principle; see also \cite[(3.23)]{chelkak-hongler}.
\end{proof}

\section{Continuum Analysis: Isomonodromy and Painlevé III\label{sec:Continuum-Analysis:-Painlev}}

\subsection{Continuous Analysis of the Coefficients\label{subsec:Analysis-of-the}}

In this section, we carry out analysis of the continuum coefficients
such as $\mathcal{A}_{\Omega},\mathcal{B}_{\Omega}$ needed for the
proof of the main theorem and the derivation of the Painlevé III transcendent
in the next section. We will assume that a continuous mass parameter
$m<0$ is fixed throughout this section.

\subsubsection*{Preliminaries: Massive Cauchy Formula.}

In Section \ref{subsec:Massive-Complex-Analysis:}, we have seen that
massive holomorphic functions admit a generalisation of Laurent-type
expansions. Here we give explicit formulae related to the formal powers
$Z_{\nu}^{1,i}$ and note that a Cauchy-type integral formula holds.

Recall $\partial_{z}=\frac{1}{2}e^{-i\theta}\left(\partial_{r}-ir^{-1}\partial_{\theta}\right),\partial_{\bar{z}}=\frac{1}{2}e^{i\theta}\left(\partial_{r}+ir^{-1}\partial_{\theta}\right)$.
The following holds (\cite[Section 10]{dlmf}): 
\begin{alignat}{1}
I_{\nu}'(r) & =I_{\nu\pm1}(r)\pm\frac{\nu}{r}I_{\nu}(r),\nonumber \\
-2\frac{\sin\nu\pi}{\pi r} & =I_{\nu}(r)I_{-\nu-1}(r)-I_{\nu+1}(r)I_{-\nu}(r).\label{eq:modified}
\end{alignat}

Define $W_{\nu}(re^{i\theta}):=e^{i\nu\theta}I_{\nu}\left(2\left|m\right|r\right)$.
The formal powers $Z_{\nu}^{1,i}$ defined in can be written as
\begin{alignat}{1}
Z_{\nu}^{1} & =\frac{\Gamma(\nu+1)}{\left|m\right|^{\nu}}\left(W_{\nu}+\left(\text{sgn}m\right)\overline{W_{\nu+1}}\right),\label{eq:ZW}\\
Z_{\nu}^{i} & =\frac{\Gamma(\nu+1)}{\left|m\right|^{\nu}}\left(iW_{\nu}-i\left(\text{sgn}m\right)\overline{W_{\nu+1}}\right).\nonumber 
\end{alignat}

\begin{prop}[{\cite[Section 6]{bers}}]
Let $f$ is an $m$-massive holomorphic function defined on a ramified
disk $\left[B_{R}(a),a\right]$. The coefficients $A_{\nu}^{1,i}$
of $Z_{\nu}^{1,i}$ in the expansion (\ref{eq:formalpowers}) may
be extracted by the line integrals:
\begin{alignat}{1}
A_{\nu}^{1} & =\frac{\pi}{4\left|m\right|\nu^{2}}\text{Re}\oint_{C}f(z)Z_{-1-\nu}^{i}(z-a)dz,\label{eq:massive-cauchy}\\
A_{\nu}^{i} & =-\frac{\pi}{4\left|m\right|(1+\nu)^{2}}\text{Re}\oint_{C}f(z)Z_{-1-\nu}^{1}(z-a)dz,\nonumber 
\end{alignat}
where the line integral is taken on any smooth curve $C$ going once
around $a$.
\end{prop}

\begin{proof}
As in the discrete case, if $f,g$ are $m$-massive holomorphic functions,
the real part $\text{Re}\int f\cdot gdz$ is well defined: indeed,
the increment around a closed curve $\partial D$ inclosing a region
$D$ is $\int_{\partial D}f\cdot gdz=2i\int\partial_{\bar{z}}(fg)d^{2}z=2i\int m\text{Re}\left(fg\right)d^{2}z\in i\mathbb{R}$
by the Green-Riemann formula. So the integral (\ref{eq:massive-cauchy})
has a well-defined value regardless of choice of $C$.

We will take $C=\partial B_{r}(a)$ for some small $r>0$. In view
of the definition (\ref{eq:ZW}), it is straightforward to verify
first that any line integral of the form $\text{Re}\int_{C}Z_{\nu}^{1}Z_{\nu'}^{1}dz,\text{Re}\int_{C}Z_{\nu}^{i}Z_{\nu'}^{i}dz$
vanish. Similarly one can verify the mixed integral
\begin{alignat*}{1}
\text{Re}\int_{C}Z_{\nu}^{1}Z_{\nu'}^{i}dz & =\delta_{\nu+\nu',-1}\frac{-4\left|m\right|\nu^{2}}{\pi},
\end{alignat*}
with the help of (\ref{eq:modified}) and the standard Gamma function
equality \cite[Section 5.5]{dlmf}
\[
\Gamma(z)\Gamma(-1-z)=\frac{\Gamma(z)\Gamma(1-z)}{z(1+z)}=\frac{1}{z(1+z)}\frac{\pi}{\sin z\pi}.
\]
\end{proof}
Finally, we note the derivatives of the formal powers.

$\partial_{r}W_{\nu}(re^{i\theta})=e^{i\nu\theta}\cdot2\left|m\right|\cdot\left(I_{\nu\pm1}(2\left|m\right|r)\pm\frac{\nu}{2\left|m\right|r}I_{\nu}(2\left|m\right|r)\right)$
and $\partial_{\theta}W_{\nu}(re^{i\theta})=i\nu e^{i\nu\theta}I_{\nu}\left(2\left|m\right|r\right)$,
and we see that 
\[
\partial_{z}W_{\nu}=\left|m\right|W_{\nu-1},\partial_{\bar{z}}W_{\nu}=\left|m\right|W_{\nu+1},
\]
and the corresponding identities for $Z_{\nu}^{1,i}$ follow. In fact,
we will record, noting $\partial_{x}=\partial_{z}+\partial_{\bar{z}},\partial_{y}=i\left(\partial_{z}-\partial_{\bar{z}}\right)$,
\begin{alignat*}{1}
\partial_{x}Z_{\nu}^{1} & =\nu Z_{\nu-1}^{1}+(\nu+1)^{-1}m^{2}Z_{\nu+1}^{1},\partial_{x}Z_{\nu}^{i}=\nu Z_{\nu-1}^{i}+(\nu+1)^{-1}m^{2}Z_{\nu+1}^{i},\\
\partial_{y}Z_{\nu}^{1} & =\nu Z_{\nu-1}^{i}-(\nu+1)^{-1}m^{2}Z_{\nu+1}^{i},\partial_{y}Z_{\nu}^{i}=-\nu Z_{\nu-1}^{1}+(\nu+1)^{-1}m^{2}Z_{\nu+1}^{1}.
\end{alignat*}

\subsubsection*{Preliminaries: Cauchy Transform and Harmonic Conjugate.}

To work with generalised analytic functions, the \emph{Cauchy transform}
$\mathfrak{C}=\mathfrak{C}_{\Omega}:t(\cdot)\mapsto-\frac{1}{\pi}\int_{\Omega}\frac{t(w)}{w-z}d^{2}w$
will be used as the inverse of the derivative $\partial_{\bar{z}}$.
Define the Hölder seminorm $\left[t\right]_{C^{\alpha}(\Omega)}(w):=\sup_{z\in\Omega}\frac{\left|t(w)-t(z)\right|}{\left|w-z\right|^{\alpha}},\left[t\right]_{C^{\alpha}(\Omega)}:=\sup_{w\in\Omega}\left[t\right]_{C^{\alpha}(\Omega)}$
and the norm $\left|t\right|_{C^{\alpha}(\Omega)}:=\left|t\right|_{C(\Omega)}+\left[t\right]_{C^{\alpha}(\Omega)}$.
The following estimates are standard and may be shown by direct analysis
of the kernel $\frac{1}{w-z}$.
\begin{prop}
\label{prop:cauchy} Let $B=B_{r}(z)$ be a ball of radius $r>0$.
\begin{itemize}
\item If $g\in L^{p}(B)$ for some $p>2$, $\mathfrak{C}g\in C^{\alpha}(\mathbb{C})$
for $\alpha=\frac{p-2}{p}$, holomorphic outside of $B$, and vanishes
at infinity, with
\begin{equation}
\left|\mathfrak{C}g\right|\leq cst\cdot r^{\alpha}\left|g\right|_{L^{p}(B)};\left[\mathfrak{C}g\right]_{C^{\alpha}(\mathbb{C})}\leq cst\cdot\left|g\right|_{L^{p}(B)},\label{eq:cauchy-lp}
\end{equation}
where the constants only depend on $p$;
\item If $g\in C^{\alpha}(B)$, $\mathfrak{C}g$ is differentiable at $z$,
with
\begin{equation}
\left|\nabla\mathfrak{C}g(z)\right|\leq cst\left(\left|g(z)\right|+r^{\alpha}\left[g\right]_{C^{\alpha}(B)}(z)\right);\label{eq:cauchy-ca}
\end{equation}
where the constant only depends on $\alpha\in(0,1]$.
\end{itemize}
\end{prop}

\begin{proof}
See \cite[Sections 1.4-8]{vek}. Specifically, for (\ref{eq:cauchy-lp})
refer to \cite[Theorem 1.19]{vek}; (\ref{eq:cauchy-ca}) follows
from \cite[(8.2), (8.7)]{vek}.
\end{proof}
Another standard analytic fact that we use is the Hölder regularity
of harmonic conjugates.
\begin{prop}[Privalov's Theorem]
\label{prop:privalov}Let $D$ be a smooth bounded simply connected
domain, and fix a conformal map $\varphi:D\to\mathbb{D}$ such that
$M^{-1}<\left|\varphi'\right|<M$ for some $M>0$.

Let $t\in C^{\alpha}(\partial D)$ be a real-valued function on the
boundary. Then there exists a holomorphic function $g\in C^{\alpha}(\bar{D})$
such that $\text{Im}g=t$ which is unique up to a real constant. Moreover,
if $\text{Re}g=0$ at any point in $\bar{D}$,
\[
\left|g\right|_{C^{\alpha}(\bar{D})}\leq cst\cdot(1+M^{2\alpha})\left|t\right|_{C^{\alpha}(\partial D)},
\]
where the constant only depends on $\alpha$.
\end{prop}

\begin{proof}
See, e.g., \cite[Theorem 3.2]{harmonic-measure}, for the proof in
the unit disc. Then it is straightforward to transfer the result to
$D$ using the conformal map $\varphi'$ given that the Hölder seminorm
transforms with a factor of $\left|\varphi'\right|^{\alpha}$.
\end{proof}

\subsubsection*{Analysis of the Continuous Observables.}

We now study the continuous observables $f_{\left[\Omega,a_{1},\ldots,a_{n}\right]}(\cdot|m)$
with $m<0$ and its coefficients. Fix a conformal map $\varphi:\Omega\to\mathbb{D}$
and let $M>0$ be such that $M^{-1}<\varphi'<M$. Note that $\text{diam}\Omega\leq2M$.

We first give the following lemma based on the Lemma \ref{lem:similarity}
(similarity principle). The following proof in fact contains the idea
of the proof of the principle, but strictly speaking we do rely on
it, e.g. for the existence of the observable itself. Write $f_{m}(z):=f_{\left[\Omega,a_{1},\ldots,a_{n}\right]}(z|m)$.
\begin{lem}
\label{lem:obs-decomposition}There exists a unique function $s$
in $\bar{\Omega}$ such that
\begin{itemize}
\item $s$ is continuous in $\bar{\Omega}$;
\item $s$ is real on $\partial\Omega$ and $\text{Re}e^{-s(a)}=1$;
\item $e^{-s}f_{m}$ is holomorphic in $\left[\Omega,a_{1},\ldots,a_{n}\right]$.
\end{itemize}
Then, defining $c_{j}:=\text{Re}\lim_{z\to a_{j}}e^{-s(z)}\sqrt{z-a_{j}}f_{m}$
(note $c_{1}=1$)
\begin{equation}
e^{-s(z)}f_{m}(z)=\sum_{j=1}^{n}c_{b}f_{\left[\Omega,a_{1},\ldots,a_{n}\right]}(z|0),\label{eq:obs-decomposition}
\end{equation}

Moreover, for any $\alpha\in(0,1)$,
\[
\left|s\right|_{C^{\alpha}}\leq cst\cdot\left|m\right|(1+M^{1+2\alpha}),
\]
where the constant only depends on $\alpha$.
\end{lem}

\begin{proof}
First assume existence of such $s$. Suppose there are two such functions
$s_{1},s_{2}$. Then $e^{-s_{1}(z)+s_{2}(z)}$ is bounded, real on
$\partial\Omega$, and holomorphic in $\Omega\setminus\left\{ a,b\right\} $;
thus the value at $a$ fixes it to be constantly $1$. So $s_{1}-s_{2}\equiv0$
given that $\text{Im}\left(s_{1}-s_{2}\right)=0$ on $\partial\Omega$
and $s$ is unique.

Then $\hat{f}(z):=e^{-s(z)}f_{m}(z)$ is holomorphic in $\left[\Omega,a_{1},\ldots,a_{n}\right]$
and $\sqrt{\nu_{out}}\hat{f}\in\mathbb{R}$ on the boundary $\partial\Omega$.
Near $a_{j}$, $\text{Re}\sqrt{z-a_{j}}\hat{f}(z)\sim c_{j}$. Thus
we obtain (\ref{eq:obs-decomposition}) since the difference of both
sides is zero by the uniqueness of the boundary value problem (\cite[Lemma 2.9]{chelkak-hongler},
i.e. the massless version of Proposition \ref{prop:contbvp}).

Now we show the existence of $s$. Define $s_{0}:=\mathfrak{C}\left[\partial_{\bar{z}}f_{m}/f_{m}\right]$.
Note that
\begin{alignat}{1}
\frac{\partial_{\bar{z}}f_{m}}{f_{m}} & =\frac{\partial_{\bar{z}}f_{\left[\Omega,a,b\right]}(z|m)}{f_{\left[\Omega,a,b\right]}(z|m)}=m\left(\frac{\overline{f_{\left[\Omega,a,b\right]}}(z|m)}{f_{\left[\Omega,a,b\right]}(z|m)}\right)\label{eq:logderiv}
\end{alignat}
is in $L^{p}(\Omega)$ for any $p$ since $f_{m}$ only vanishes at
isolated points, so $s_{0}$ is bounded by Proposition \ref{prop:cauchy}.

We claim $s_{0}$ is differentiable almost everywhere in $\Omega$.
$s_{0}$ is differentiable at any point $z_{0}$ near which $\partial_{\bar{z}}\bar{f}_{m}/f_{m}$
satisfies a Hölder condition, since $\mathfrak{C}_{\Omega}\left[\partial_{\bar{z}}f_{m}/f_{m}\right]$
is the sum of $\mathfrak{C}_{\Omega\setminus B_{r}(z_{0})}\left[\partial_{\bar{z}}f_{m}/f_{m}\right]$,
which is holomorphic at $z_{0}$, and $\mathfrak{C}_{B_{r}(z_{0})}\left[\partial_{\bar{z}}f_{m}/f_{m}\right]$,
which is differentiable by Proposition \ref{prop:cauchy}. Since $f_{m}$
is smooth on $\left[\Omega,a_{1},\ldots,a_{n}\right]$, $s_{0}$ is
differentiable away from $a_{j}$ and at isolated points where $f_{m}$
vanishes.

So $e^{-s_{0}}f_{m}$ is holomorphic almost everywhere, and by removable
singularity it is holomorphic on $\left[\Omega,a_{1},\ldots,a_{n}\right]$.
Now let $s_{1}$ be a holomorphic function on $\Omega$ with boundary
data $\text{Im}s_{1}=\text{Im}s_{0}$ on $\partial\Omega$. We fix
$\text{Re}e^{s_{1}-s_{0}}=1$. Then $s:=s_{0}-s_{1}$ satisfies the
desired properties.

For the Hölder estimate, recall that $\text{diam}\Omega\leq2M$ and
thus $\left|\partial_{\bar{z}}f_{m}/f_{m}\right|_{L^{p}(\Omega)}\leq cst\cdot\left|m\right|M^{\frac{2}{p}}$.
Then from Proposition \ref{prop:cauchy} we have $\left|s_{0}\right|_{C^{\alpha}}\leq cst\cdot(1+M^{\alpha})\cdot\left|m\right|M^{1-\alpha}$.
Then the corresponding norm for $s_{1}$ is given by Proposition \ref{prop:privalov},
and the sum gives the desired estimate.
\end{proof}
The above lemma allows us to give the following results in the case
$n=2$.
\begin{lem}
\label{lem:B}$\left|\mathcal{B}_{\Omega}\left(a_{1},a_{2}|m\right)\right|\to1$
as $\frac{\left|a_{1}-a_{2}\right|}{\text{dist}\left(\left\{ a_{1},a_{2}\right\} ,\partial\Omega\right)}\to0$.
\end{lem}

\begin{proof}
By \cite[Remark 2.24]{chelkak-hongler}, the result holds for $m=0$.

Now consider the decomposition (\ref{eq:obs-decomposition}). Comparing
the imaginary parts of the coefficient of $\frac{1}{\sqrt{z-a_{2}}}$
in both sides, we see that $\mathcal{B}_{\Omega}\left(a_{1},a_{2}|m\right)\text{Re}\left[e^{-s(a_{2})}\right]=\mathcal{B}_{\Omega}(a_{1},a_{2}|0)$.
Since $\text{Re}e^{-s(a_{1})}=1$ and $e^{-s}$ is Hölder continuous
by the previous lemma, we have $\left|\frac{\mathcal{B}_{\Omega}\left(a_{1},a_{2}|m\right)}{\mathcal{B}_{\Omega}(a_{1},a_{2}|0)}\right|\to1$
as $\left|a_{1}-a_{2}\right|\to0$.
\end{proof}
Before we go on to give a more delicate estimate on the two-point
observable, we note a few facts we use.
\begin{rem}
As mentioned above, the possible zeroes of $f_{m}$ are problematic
for the regularity of $s$. However, in the case where $n=2$, $f_{m}$
cannot vanish in $\Omega$.

Indeed, $\hat{f}=e^{-s}f_{m}$ is a holomorphic function on $\left[\Omega,a_{1},a_{2}\right]$
with the boundary condition $\sqrt{\nu_{out}}\hat{f}\in\mathbb{R}$.
By the argument principle for solutions of a Hilbert boundary problem
(\cite[Theorem 2.2]{wen}), one sees that $1=\frac{1}{2}N_{\partial\Omega}+N_{\Omega}$,
where $1$ is the index of $\nu_{out}\cdot(z-a_{1})(z-a_{2})$ on
$\partial\Omega$, $N_{\partial\Omega},N_{\Omega}$ are number of
zeroes of $\hat{f}^{2}(z-a_{1})(z-a_{2})$ respectively on $\partial\Omega,\Omega$
counted with multiplicity. But since it is a square, any zero of $\hat{f}^{2}(z-a_{1})(z-a_{2})$
is second order; the only possible scenario then is that $N_{\partial\Omega}=2,N_{\Omega}=0$.
So $\hat{f}$ does not vanish in $\Omega$, and since $s$ is bounded,
$f_{m}$ does not.
\end{rem}

\begin{rem}
\label{rem:misc}Since $\hat{f}$ is a linear combination of the two
observables with $m=0$, we may estimate $f_{m}$ using properties
of them. The massless observables are conformally covariant: if $\phi:\Omega\to\Omega'$
is a conformal map \cite[Lemmas 2.9, 2.21]{chelkak-hongler},
\begin{alignat}{1}
f_{\left[\Omega,a_{1},\ldots,a_{n}\right]}(z|0) & =f_{\left[\Omega',\phi(a_{1}),\ldots,\phi(a_{n})\right]}(\phi(z)|0)\phi'(z)^{1/2};\nonumber \\
f_{\left[\mathbb{H},a,b\right]}(z|0) & =\frac{\left(2i\text{Im}a\right)^{1/2}}{\left|b-\bar{a}\right|+\left|b-a\right|}\frac{\left[(\bar{b}-\bar{a})(\bar{b}-a)\right]^{1/2}(z-b)+\left[(b-a)(b-\bar{a})\right]^{1/2}(z-\bar{b})}{\sqrt{(z-a)(z-\bar{a})(z-b)(z-\bar{b})}}.\label{eq:halfplaneexplicit}
\end{alignat}
\end{rem}

Now we list some properties of the half-plane observable which need
simple verifications.
\begin{lem}
\label{lem:halfplane}Suppose $a,b\in\mathbb{H}$ with $\text{Im}a=\text{Im}b=\epsilon$,
$\text{Re}a<\text{Re}b$ (i.e. $\left|b-a\right|=\text{Re}b-\text{Re}a$)
\begin{enumerate}
\item Let $r\leq\frac{1}{2}\min(\epsilon,d)$. Then
\begin{alignat*}{1}
\left|\sqrt{z-a}f_{\left[\mathbb{H},a,b\right]}(z)\right| & \leq cst\text{ for }z\in B_{r}(a),\\
\left|\sqrt{z-b}f_{\left[\mathbb{H},a,b\right]}(z)\right| & \leq cst\text{ for }z\in B_{r}(b),
\end{alignat*}
with constants independent of the positions of $a,b,r$.
\item Let $n_{a,b}\in\mathbb{R}$ be the zero of $f_{\left[\mathbb{H},a,b\right]}$.
\begin{alignat*}{1}
0 & \leq n_{a,b}-n_{b,a}\leq\left|b-a\right|+2\epsilon;\\
 & \left|n_{a,b}-a\right|\geq\left|b-a\right|+\epsilon.
\end{alignat*}
\item $\frac{f_{\left[\Omega,b,a\right]}(z|0)}{f_{\left[\Omega,a,b\right]}(z|0)}$
can be extended to a holomorphic function in $\Omega$, and 
\[
\left|\frac{f_{\left[\Omega,b,a\right]}(a|0)}{f_{\left[\Omega,a,b\right]}(a|0)}\right|\leq3;\text{ \ensuremath{\left|\left(\frac{f_{\left[\mathbb{H},b,a\right]}(z|0)}{f_{\left[\mathbb{H},a,b\right]}(z|0)}\right)'_{z=a}\right|}\ensuremath{\ensuremath{\leq\frac{cst}{\left|b-a\right|+\epsilon}}.}}
\]
\end{enumerate}
\end{lem}

\begin{proof}
These results all follow from the explicit formulae above.
\begin{enumerate}
\item Note that $\left|b-a\right|\leq\left|b-\bar{a}\right|\leq2\epsilon+\left|b-a\right|$.
We have
\begin{alignat*}{1}
\left|\frac{\left[(\bar{b}-\bar{a})(\bar{b}-a)\right]^{1/2}}{\left|b-\bar{a}\right|+\left|b-a\right|}\right| & \leq\left|\frac{\left[\left|b-a\right|(\left|b-a\right|+2\epsilon)\right]^{1/2}}{2\left|b-a\right|+2\epsilon}\right|\leq\frac{\sqrt{\left|b-a\right|}}{\sqrt{2\left|b-a\right|+2\epsilon}}<1.
\end{alignat*}
Then we estimate the two terms in (\ref{eq:halfplaneexplicit}) separately,
noting that $\text{Im}a=\epsilon\leq|z-\bar{a}|$, and also $\left|\sqrt{z-b}\right|\leq\left|\sqrt{z-\bar{b}}\right|$,
\begin{equation}
\left|\frac{\left(2i\text{Im}a\right)^{1/2}(z-b)}{\sqrt{(z-a)(z-\bar{a})(z-b)(z-\bar{b})}}\right|\leq\left|\frac{\left(2i\text{Im}a\right)^{1/2}\sqrt{(z-b)}}{\sqrt{(z-a)(z-\bar{a})(z-\bar{b})}}\right|\leq\frac{cst}{\left|\sqrt{z-a}\right|}.\label{eq:halfplaneest}
\end{equation}
Since $\left|z-\bar{b}\right|\leq\epsilon+\frac{|b-a|}{2}$,
\[
\left|\frac{\sqrt{\left|b-a\right|}}{\sqrt{2\left|b-a\right|+2\epsilon}}\frac{\left(2i\text{Im}a\right)^{1/2}(z-\bar{b})}{\sqrt{(z-a)(z-\bar{a})(z-b)(z-\bar{b})}}\right|\leq\left|\frac{cst\sqrt{\left|b-a\right|}}{\sqrt{(z-a)(z-b)}}\right|.
\]
If $z\in B_{r}(a)$, we have the result from $\left|z-b\right|>\frac{1}{2}\left|a-b\right|$.
If $z\in B_{r}(b)$, similarly note $\left|z-a\right|\geq\frac{1}{2}\left|a-b\right|$,
and apply $\left|\sqrt{z-a}\right|>\left|\sqrt{z-b}\right|$ to (\ref{eq:halfplaneest}).
\item From the formula on $f_{\left[\mathbb{H},a,b\right]}$, we may write
\begin{alignat*}{1}
n_{a,b} & =\frac{\text{Re}\sqrt{(\bar{b}-\bar{a})(\bar{b}-a)}b}{\text{Re}\sqrt{(\bar{b}-\bar{a})(\bar{b}-a)}}=\text{Re}b-\frac{\text{Im}\sqrt{\bar{b}-a}\text{Im}b}{\text{Re}\sqrt{\bar{b}-a}},\\
n_{b,a} & =\text{Re}a-\frac{\text{Im}\sqrt{b-\bar{a}}\text{Im}a}{\text{Re}\sqrt{b-\bar{a}}}=\text{Re}a+\frac{\text{Im}\sqrt{\bar{b}-a}\text{Im}a}{\text{Re}\sqrt{\bar{b}-a}}.
\end{alignat*}
So
\begin{alignat*}{1}
0\leq n_{a,b}-n_{b,a} & \leq\left|b-a\right|+2\epsilon,
\end{alignat*}
and 
\[
\left|n_{a,b}-a\right|\geq(n_{a,b}-\text{Re}a)+\text{Im}a\ge\left|b-a\right|+\epsilon,
\]
since $\sqrt{\bar{b}-a}$ belongs to the fourth quadrant and thus
\[
0<-\frac{\text{Im}\sqrt{\bar{b}-a}}{\text{Re}\sqrt{\bar{b}-a}}\leq1.
\]
\item Again examining the formula, we have, away from $b,a$ (and then everywhere
in $\mathbb{H}$ by removable singularity),
\[
\frac{f_{\left[\mathbb{H},b,a\right]}(z|0)}{f_{\left[\mathbb{H},a,b\right]}(z|0)}=\frac{\text{Re}\left[(\bar{a}-\bar{b})(\bar{a}-b)\right]^{1/2}(z-n_{b,a})}{\text{Re}\left[(\bar{b}-\bar{a})(\bar{b}-a)\right]^{1/2}(z-n_{a,b})}=:\rho\frac{z-n_{b,a}}{z-n_{a,b}},
\]
so
\begin{alignat*}{1}
\left|\frac{f_{\left[\mathbb{H},b,a\right]}(a|0)}{f_{\left[\mathbb{H},a,b\right]}(a|0)}\right| & =\left|1+\frac{n_{a,b}-n_{b,a}}{a-n_{a,b}}\right|\leq1+\frac{\text{\ensuremath{\left|b-a\right|}}+2\epsilon}{\left|b-a\right|+\epsilon}\leq3;\\
\left|\left(\frac{f_{\left[\mathbb{H},b,a\right]}(z|0)}{f_{\left[\mathbb{H},a,b\right]}(z|0)}\right)'_{z=a}\right| & =\text{\ensuremath{\left|\frac{n_{a,b}-n_{b,a}}{\left(a-n_{a,b}\right)^{2}}\right|\leq\frac{\left|b-a\right|+2\epsilon}{\left(\left|b-a\right|+\epsilon\right)^{2}}\leq\frac{cst}{\left|b-a\right|+\epsilon}.}}
\end{alignat*}
\end{enumerate}
\end{proof}
Recall that we fix a conformal map $\varphi:\Omega\to\mathbb{D}$
such that $M^{-1}\leq\varphi'\leq M$. Now fix the standard conformal
map $\mathbb{D}\to\mathbb{H}$ such that $0\in\mathbb{D}$ is mapped
to $i\in\mathbb{H}$, and fix $\varphi_{\mathbb{H}}:\Omega\to\mathbb{H}$.
Similar estimate $cst(R)\cdot M^{-1}\leq\varphi_{H}'\leq Cst(R)\cdot M$
holds in $\varphi^{-1}\left(B_{R}\cap\mathbb{H}\right)$ for $R>0$.
Write $\mathcal{A}_{\Omega}^{1}(a_{1},a_{2}|m)+i\mathcal{A}_{\Omega}^{i}(a_{1},a_{2}|m)=:\mathcal{A}_{\Omega}(a_{1},a_{2}|m)$.
\begin{lem}
\label{lem:two-point-delicate}Let $a_{1},a_{2}\in\varphi_{\mathbb{H}}^{-1}(B_{R}\cap\mathbb{H})\subset\Omega$
be such that $\text{Im}\varphi_{\mathbb{H}}(a_{1})=\text{Im}\varphi_{\mathbb{H}}(a_{2})=\epsilon$.
Then for any fixed $0<\gamma<1$, we have
\[
\left|\mathcal{A}_{\Omega}(a_{1},a_{2}|m)-\mathcal{A}_{\Omega}(a_{1},a_{2}|0)\right|\leq cst\cdot\left(\epsilon^{-\gamma}+\left|a_{1}-a_{2}\right|^{-\gamma}\right),
\]
where the constant only depends on $M,R,m,\gamma$.
\end{lem}

\begin{proof}
We will refer to quantities only depending on $M,R,m,\gamma$ as constants
in this proof. To extract the desired difference $\mathcal{A}^{\Delta}:=\mathcal{A}_{\Omega}(a_{1},a_{2}|m)-\mathcal{A}_{\Omega}(a_{1},a_{2}|0)$,
we will study $e^{-2m\left|z-a_{1}\right|}f_{m}(z)/f_{\left[\Omega,a_{1},a_{2}\right]}(z|0)$.
Indeed, around $a_{1}$, we have
\begin{alignat*}{1}
\frac{e^{-2m\left|z-a_{1}\right|}f_{m}(z)}{f_{\left[\Omega,a_{1},a_{2}\right]}(z|0)} & =\frac{1}{e^{2m\left|z-a_{1}\right|}}\frac{\left(\frac{e^{2m\left|z-a_{1}\right|}}{\sqrt{z-a_{1}}}+2\mathcal{A}_{\Omega}(a_{1},a_{2}|m)\sqrt{z-a_{1}}+o(\left|z-a_{1}\right|^{1/2})\right)}{\left(\frac{1}{\sqrt{z-a_{1}}}+2\mathcal{A}_{\Omega}(a_{1},a_{2}|0)\sqrt{z-a_{1}}+o(\left|z-a_{1}\right|^{1/2})\right)}\\
 & =1+2\mathcal{A}^{\Delta}\left(z-a_{1}\right)+o(\left|z-a\right|),
\end{alignat*}
therefore $\mathcal{A}^{\Delta}$ is half the derivative of $\frac{e^{-2m\left|z-a_{1}\right|}f_{m}(z)}{f_{\left[\Omega,a_{1},a_{2}\right]}(z|m)}$
at $a_{1}$.

By (\ref{eq:obs-decomposition}),
\[
\frac{e^{-2m\left|z-a_{1}\right|}f_{m}(z|m)}{f_{\left[\Omega,a_{1},a_{2}\right]}(z|0)}=e^{s(z)-2m\left|z-a_{1}\right|}\left(1+c_{2}\frac{f_{\left[\Omega,a_{2},a_{1}\right]}(z|0)}{f_{\left[\Omega,a_{1},a_{2}\right]}(z|0)}\right),
\]
and we will estimate the derivatives of the two factors separately.
Fix $\alpha=1-\gamma$.

For the second factor, note that $\left|c_{2}\right|=\left|\text{Im}e^{-s(a_{2})}\mathcal{B}_{\Omega}(a_{1},a_{2}|m)\right|\leq\left|\text{Im}e^{-s(a_{2})}\right|$,
but since $s$ is uniform $\alpha$-Hölder continuous and purely real
on $\partial\Omega$, $\left|\text{Im}e^{-s(a_{2})}\right|\leq cst\cdot\text{dist}(a_{2},\partial\Omega)^{\alpha}\leq cst\cdot\left(M\epsilon\right)^{\alpha}$.
Now by conformal covariance, $\frac{f_{\left[\Omega,a_{2},a_{1}\right]}(z|0)}{f_{\left[\Omega,a_{1},a_{2}\right]}(z|0)}=\frac{f_{\left[\Omega,\varphi_{\mathbb{H}}(a_{2}),\varphi_{\mathbb{H}}(a_{1})\right]}(\varphi_{\mathbb{H}}(z)|0)}{f_{\left[\Omega,\varphi_{\mathbb{H}}(a_{1}),\varphi_{\mathbb{H}}(a_{2})\right]}(\varphi_{\mathbb{H}}(z)|0)}$,
so we may apply the third estimate in Lemma \ref{lem:halfplane}:
\begin{alignat}{1}
1+c_{2}\frac{f_{\left[\Omega,a_{2},a_{1}\right]}(a_{1}|0)}{f_{\left[\Omega,a_{1},a_{2}\right]}(a_{1}|0)} & \leq1+cst\cdot(M\epsilon)^{\alpha}\leq cst,\nonumber \\
\left(1+c_{2}\frac{f_{\left[\Omega,a_{2},a_{1}\right]}(z|0)}{f_{\left[\Omega,a_{1},a_{2}\right]}(z|0)}\right)_{z=a_{1}}^{'} & \leq\frac{cst\cdot M^{\alpha+1}\epsilon^{\alpha}}{\left|a_{2}-a_{1}\right|+\epsilon}\leq cst\cdot M^{\alpha+1}\epsilon^{\alpha-1}.\label{eq:secondpart}
\end{alignat}

Now for the first factor, we claim that $s(z)-2m|z-a_{1}|$ is differentiable
at $a_{1}$. Take $r=\frac{1}{2M}\min(\left|a_{1}-a_{2}\right|,\epsilon)$.
$f_{m}(z)=e^{-s(z)}\hat{f}(z)$, where $\hat{f}(z)=f_{\left[\Omega,a_{1},a_{2}\right]}+c_{2}f_{\left[\Omega,a_{2},a_{1}\right]}$.
By the first estimate of Lemma \ref{lem:halfplane}, $\sqrt{z-a_{1}}\hat{f}(z)$
is bounded on $B_{r}(a_{1})$ (the factor of $M$ in $r$ is there
so that $\varphi(B_{r}(a_{1}))\subset\mathbb{H}$ satisfies the condition
of the first estimate Lemma \ref{lem:halfplane}). Expand for $z\in B_{\frac{r}{4}}(a_{1})$,
\begin{equation}
\left|\sqrt{z-a_{1}}\hat{f}(z)-e^{s(a_{1})}\right|\leq\frac{cst}{r}\left|z-a_{1}\right|,\label{eq:lipschitzest}
\end{equation}
where we use the fact that $\left.\sqrt{z-a_{1}}\hat{f}(z)\right|_{z=a_{1}}=e^{s(a_{1})}$
and the derivative of $\sqrt{z-a_{1}}\hat{f}(z)$ is bounded by $\frac{cst}{r}$
uniformly in $B_{\frac{r}{4}}(a_{1})$ by the Cauchy formula. Then
again using the uniform Hölder regularity of $s$,
\begin{alignat*}{1}
\left|\sqrt{z-a_{1}}f_{m}(z)-1\right| & =\left|e^{-s(z)}\sqrt{z-a_{1}}\hat{f}(z)-e^{-s(z)+s(a_{1})}+e^{-s(z)+s(a_{1})}-1\right|\\
 & =\left|e^{-s(z)}\right|\left|\sqrt{z-a_{1}}\hat{f}(z)-e^{s(a_{1})}\right|+\left|e^{-s(z)+s(a_{1})}-1\right|\\
 & \leq\frac{cst}{r}\left|z-a_{1}\right|+cst\cdot\left|z-a_{1}\right|^{\alpha}\leq\frac{cst}{r^{\alpha}}\left|z-a_{1}\right|^{\alpha}.
\end{alignat*}

Note that by (\ref{eq:lipschitzest}) there is a constant $\mathfrak{c}\in(0,\frac{1}{4}]$
such that $\left|\sqrt{z-a_{1}}f_{m}(z)\right|>cst>0$ on $B_{\mathfrak{c}r}(a_{1})$.
Thus we have in $B_{\mathfrak{c}r}(a_{1})$

\begin{alignat*}{1}
\left|\partial_{\bar{z}}(s(z)-2m|z-a_{1}|)\right| & =\left|\frac{m\bar{f_{m}}}{f_{m}}-\frac{m\sqrt{z-a_{1}}}{\overline{\sqrt{z-a_{1}}}}\right|\\
 & =\left|m\frac{\bar{f}_{m}\overline{\sqrt{z-a_{1}}}-f_{m}\sqrt{z-a_{1}}}{f_{m}\overline{\sqrt{z-a_{1}}}}\right|\\
 & \leq\frac{cst\cdot\left|m\right|}{r^{\alpha}}\left|z-a_{1}\right|^{\alpha}.
\end{alignat*}

Then by (\ref{eq:cauchy-ca}), $\mathfrak{C}_{B_{\mathfrak{c}r}(a_{1})}\partial_{\bar{z}}(s(z)-2m|z-a_{1}|)$
is differentiable at $a_{1}$, with the derivative bounded by $\left|s(a_{1})\right|+cst\cdot\left|m\right|\leq cst$.

The remainder $s(z)-2m|z-a_{1}|-\mathfrak{C}_{B_{\mathfrak{c}r}(a_{1})}\partial_{\bar{z}}(s(z)-2m|z-a_{1}|)$
is holomorphic in $B_{\mathfrak{c}r}(a_{1})$, with uniformly $\alpha$-Hölder
boundary values on $\partial B_{\mathfrak{c}r}\left(a_{1}\right)$.
By Cauchy formula again, its derivative at $a_{1}$ is bounded by
$\frac{cst}{r^{1-\alpha}}$. Therefore, $s(z)-2m|z-a_{1}|$ is differentiable
and its derivative is bounded by $\frac{cst}{r^{1-\alpha}}$ for small
$r$. Combining this estimate with (\ref{eq:secondpart}) and writing
$\gamma=1-\alpha$ gives the desired estimate.
\end{proof}
Finally, we show that the full plane observables are differentiable
in the positions $a_{j}$ of the spins, grounding the isomonodromic
analysis in the next section.
\begin{prop}
\label{prop:The-full-plane-obs-differentiability}The value of the
full-plane observable $f_{\left[\mathbb{C},a_{1},\ldots,a_{n}\right]}(z)$
and thus its coefficients of the formal power series expansion around
$a_{j}$ are differentiable in the positions $a_{1},\ldots,a_{n}$.
\end{prop}

\begin{proof}
Differentiability of the coefficients follow directly from that of
the observable value, since we can recover the coefficients using
the Cauchy formula (\ref{eq:massive-cauchy}).

Without loss of generality, we show the $x$-derivative in $a_{1}$
exists. Now, set $h_{0}>4h>0,a_{1}^{h}=a_{1}+h$ and embed all $\left[\mathbb{C},a_{1}^{h},\ldots,a_{n}\right]$
in the double cover $\mathbb{C}^{2h}$ of $\mathbb{C}\setminus\left[B_{2h}(a_{1})\cup\left\{ a_{2},\ldots,a_{n}\right\} \right]$.
Consider the difference $f^{h}:=f_{\left[\mathbb{C},a_{1}^{h},\ldots,a_{n}\right]}-f_{\left[\mathbb{C},a_{1},\ldots,a_{n}\right]}$
defined on $\mathbb{C}^{2h}$.

Clearly $f^{h}$ is massive holomorphic, $\lim_{z\to a_{j}}\sqrt{z-a_{j}}f^{h}(z)=:i\mathcal{B}_{j}^{h}\in i\mathbb{R}$
for $j\geq2$, and $f^{h}$ decays exponentially fast at infinity.
So by applying the Green-Riemann theorem as in (\ref{eq:green-riemann})
to $\left(f^{h}\right)^{2}$ on $\mathbb{C}^{h_{0}}$, we see that
\begin{equation}
i\oint_{\partial B_{h_{0}}(a_{1})}\left(f^{h}\right)^{2}dz=-2\iint_{\mathbb{C}\setminus B_{h_{0}}(a_{1})}2m|f^{h}|^{2}dz+\sum_{j}2\pi\left(\mathcal{B}_{j}^{h}\right)^{2}.\label{eq:green-riemann 2}
\end{equation}

We claim that the left hand side, which is nonnegative and real since
it is equal to the right hand side, is $O(h^{2})$ on $\partial B_{h_{0}}(a_{1})$
as $h\to0$, which will be proven below. Then $\left\{ \frac{1}{h}f^{h}\right\} _{h<h_{0}}$
is uniformly $L^{2}$-bounded in $\mathbb{C}^{h_{0}}$. By a continuous
version of Proposition \ref{prop:hmestimate}, there exists a subsequence
$h_{k}$ such that $\frac{1}{h_{k}}f^{h_{k}}$ converges uniformly
in compact subsets; the limit satisfies the same boundary value problem
as in Proposition \ref{prop:contbvp} away from $a_{1}$. By diagonalising
as $h_{0}\to0$, we may assume that there exists a limit $f$ on $\left[\mathbb{C},a_{1},\ldots,a_{n}\right]$.
It suffices to show that $f$ is unique. But by the expansion (\ref{eq:A}),
the form of a subsequential limit near $a_{1}$ is determined; indeed,
the singular behaviour is 
\[
-\frac{1}{2}Z_{-\frac{3}{2}}^{1}(z-a_{1})+2\mathcal{A}_{\Omega}^{1}Z_{-\frac{1}{2}}^{1}(z-a_{1})+2\mathcal{A}_{\Omega}^{i}Z_{-\frac{1}{2}}^{i}(z-a_{1})+(\text{regular part}),
\]
so the difference of any two limits is zero by the uniqueness of the
boundary value problem.

Once smoothness in the position of $a_{1}$ is proved, one may repeat
the same arguments for other points, say, $a_{2}$. The difference
is that the singularity at $a_{2}$ is $\frac{i\mathcal{B}_{2}}{\sqrt{z-a_{2}}}$,
and we need to first show differentiability of $\mathcal{B}_{2}$
to make the argument work. But $f_{\left[\mathbb{C},a_{2},a_{1},\ldots,a_{n}\right]}$
and its coefficients is differentiable in $a_{2}$ by above, and $i\frac{\left(\mathcal{B}_{2}^{h}-\mathcal{B}_{2}\right)}{h}$
is equal to the coefficient of $\frac{1}{\sqrt{z-a_{1}}}$ in $\frac{1}{h}\left[f_{\left[\mathbb{C},a_{2}^{h},a_{1},\ldots,a_{n}\right]}-f_{\left[\mathbb{C},a_{2},a_{1},\ldots,a_{n}\right]}\right]$
by Green-Riemann's formula applied to $f_{\left[\mathbb{C},a_{1},a_{2}^{h},\ldots,a_{n}\right]}f_{\left[\mathbb{C},a_{2}^{h},a_{1},\ldots,a_{n}\right]}$,
so its limit as $h\to0$ exists.
\begin{proof}[Proof of Claim]

Note that (\ref{eq:green-riemann 2}) shows that the integral is real
and positive. Suppose $\frac{i}{h^{2}}\oint_{\partial B_{h_{0}}(a_{1})}\left(f^{h}\right)^{2}dz\to\infty$
as $h\to0$, possibly along a subsequence (for notational convenience,
we do not explicitly write the subsequence). Define $s_{h_{0}}^{h}:=i\oint_{\partial B_{h_{0}}(a_{1})}\left(f^{h}\right)^{2}dz$,
then $\frac{h^{2}}{s_{h_{0}}^{h}}\to0$. Consider $Z^{h}:=Z_{-\frac{1}{2}}^{1}(z-a_{1}^{h})-Z_{-\frac{1}{2}}^{1}(z-a_{1})$.
Clearly, $\frac{1}{h}Z^{h}$ is bounded on $\partial B_{h_{0}}(a_{1})$,
so $\frac{i}{s_{h_{0}}^{h}}\oint_{\partial B_{h_{0}}(a_{1})}\left(f^{h}-Z^{h}\right)^{2}dz\to1$.
But $f^{h}-Z^{h}$ is massive holomorphic in $\mathbb{C}^{2h}$, so
the Green-Riemann's formula gives 
\begin{alignat*}{1}
i\oint_{\partial B_{2h}(a_{1})}\left(f^{h}-Z^{h}\right)^{2}dz & =i\oint_{\partial B_{h_{0}}(a_{1})}\left(f^{h}-Z^{h}\right)^{2}dz-2\int2m\left|f^{h}\right|^{2}d^{2}z,\\
\frac{i}{s_{h_{0}}^{h}}\oint_{\partial B_{2h}(a_{1})}\left(f^{h}-Z^{h}\right)^{2}dz & \geq\frac{i}{s_{h_{0}}^{h}}\oint_{\partial B_{h_{0}}(a_{1})}\left(f^{h}-Z^{h}\right)^{2}dz\xrightarrow{h\to0}1,
\end{alignat*}
 so it suffices to show that the left hand side tends to zero as $h\to0$
to derive contradiction.

Fix $h_{0}$. Define $f_{h}^{\dagger}(z):=f_{\left[\mathbb{C},a_{1}^{h},\ldots,a_{n}\right]}(z)-Z_{-\frac{1}{2}}^{1}(z-a_{1}^{h})$
for $h<\frac{h_{0}}{4}$, which is bounded near $a_{1}^{h}$. Note
that, since the $L^{2}$ norm of $f_{\left[\mathbb{C},a_{1}^{h},\ldots,a_{n}\right]}(z)$
is bounded on $B_{\frac{3}{2}h_{0}}(a_{1})\setminus B_{\frac{1}{2}h_{0}}(a_{1})$
(uniformly in $h<\frac{h_{0}}{4}$), $\left|f_{\left[\mathbb{C},a_{1}^{h},\ldots,a_{n}\right]}\right|$
is bounded (uniformly for $h<\frac{h_{0}}{4}$) on the circle $\partial B_{h_{0}}(a_{1})$
by the continuous version of Proposition \ref{prop:hmestimate}. Similarly,
$Z_{-\frac{1}{2}}^{1}(z-a_{1}^{h})$ is bounded on $\partial B_{h_{0}}(a_{1})$
(uniformly for $h<\frac{h_{0}}{4}$). By Proposition \ref{prop:cauchy}
and the proof of Lemma \ref{lem:obs-decomposition}, $\hat{f_{h}}:=\exp\left[-\mathfrak{C}_{B_{h_{0}}(a_{1})}\left[m\frac{\bar{f_{h}^{\dagger}}}{f_{h}^{\dagger}}\right]\right]f_{h}^{\dagger}$
is a holomorphic function in $\left[B_{h_{0}}(a_{1}),a_{1}^{h}\right]$
which remains bounded near $a_{1}^{h}$ and bounded (uniformly for
$h<\frac{h_{0}}{4}$) on $\partial B_{h_{0}}(a_{1})$. Therefore,
$\hat{f_{h}}^{2}$ is a holomorphic function on $B_{h_{0}}(a_{1})$
which is zero at $a_{1}^{h}\in B_{2h}(a_{1})$ and has a bounded derivative
in $B_{2h}(a_{1})$ (uniformly for $h<\frac{h_{0}}{4}$), applying
Cauchy integral formula on $\partial B_{h_{0}}(a_{1})$. Therefore,
$\left|\hat{f}_{h}^{2}(z)\right|\leq cst\cdot h$ on $B_{2h}(a_{1})$
for any $h<\frac{h_{0}}{4}$. Taking into account uniform boundedness
of $\exp\left[\mathfrak{C}_{B_{h_{0}}(a_{1})}\left[m\frac{\bar{f_{h}^{\dagger}}}{f_{h}^{\dagger}}\right]\right]$,
$f_{h}^{\dagger}=\exp\left[\mathfrak{C}_{B_{h_{0}}(a_{1})}\left[m\frac{\bar{f_{h}^{\dagger}}}{f_{h}^{\dagger}}\right]\right]\hat{f_{h}}$
is similarly bounded: $\left|f_{h}^{\dagger}(z)\right|^{2}\leq cst\cdot h$
on $B_{2h}(a_{1})$ for $h<\frac{h_{0}}{4}$. Therefore the integral
\[
\frac{i}{s_{h_{0}}^{h}}\oint_{\partial B_{2h}(a_{1})}\left(f^{h}-Z^{h}\right)^{2}dz=\frac{i}{s_{h_{0}}^{h}}\oint_{\partial B_{2h}(a_{1})}\left(f_{h}^{\dagger}-f_{0}^{\dagger}\right)^{2}dz\leq\frac{cst\cdot h^{2}}{s_{h_{0}}^{h}},
\]
 tend to zero.
\end{proof}
\end{proof}

\subsection{Derivation of Painlevé III\label{subsec:Derivation-of-Painlev}}

In this section, we take the convergence results established in Section
\ref{sec:Discrete-Analysis:-Scaling} and derive established correlation
results in the full plane, first shown in \cite{wmtb} and reformulated
in terms of isomonodromic deformation in \cite{sato-miwa-jimbo}.
We will explicitly carry out the basic $2$-point case following the
presentation of \cite[Sections III, IV]{kako80}, using the continuous
limit of our discrete massive fermions which has been characterised
in terms of a boundary value problem in Definition \ref{def:A-powerseries}.
We cannot directly cite their formulae, since instead of considering
a complex space of functions which solve a two-dimensional Dirac equation,
we cast them in terms of a real space of massive holomorphic functions
because massive holomorphicity is an $\mathbb{R}$-linear notion.
The resulting analysis is equivalent.

\subsubsection*{Isomonodromy.}

We would first like to note how the functions behave under rotation
around the origin. We will compose rotation of the coordinate system
with multiplication by a phase factor and denote it by $R_{\phi}W_{\nu}(z):=W_{\nu}(e^{-i\phi}z)e^{-i\phi/2}$
and so on: first we see that $R_{\phi}W_{\nu}(z)=e^{-i(\nu+\frac{1}{2})\phi}W_{\nu}(z)$,
and similarly
\begin{alignat}{1}
R_{\phi}Z_{\nu}^{1} & =\frac{\Gamma(\nu+1)}{\left|m\right|^{\nu}}\left(e^{-i(\nu+\frac{1}{2})\phi}W_{\nu}+\left(\text{sgn}m\right)e^{i\left(\nu+\frac{1}{2}\right)\phi}\overline{W_{\nu+1}}\right)\label{eq:rot}\\
 & =\cos\left[\left(\nu+\frac{1}{2}\right)\phi\right]Z_{\nu}^{1}+\sin\left[\left(\nu+\frac{1}{2}\right)\phi\right]Z_{\nu}^{i},\nonumber \\
R_{\phi}Z_{\nu}^{i} & =\cos\left[\left(\nu+\frac{1}{2}\right)\phi\right]Z_{\nu}^{i}-\sin\left[\left(\nu+\frac{1}{2}\right)\phi\right]Z_{\nu}^{1}.\nonumber 
\end{alignat}

Recall we fix $m<0$. Suppose $a>0$ is a positive real number, and
consider the double cover $\left[\mathbb{C},-a,a\right]$. Consider
the real vector space of $m$-massive holomorphic functions on the
double cover which have singularity of order at most $3/2$ at each
monodromy and decay at infinity. Around each monodromy, we can expand
the singular part of a function in $Z_{-\frac{3}{2},-\frac{1}{2}}^{1,i}$,
and from Proposition \ref{prop:contbvp} we see in fact fixing the
coefficients of $Z_{-\frac{3}{2}}^{1,i},Z_{-\frac{1}{2}}^{1}$ at
each monodromy fixes the function. $6$ basis functions are given
by the two fermions $f_{1}:=f_{\left[\mathbb{C},-a,a\right]},f_{2}:=f_{\left[\mathbb{C},a,-a\right]}$
and their derivatives $\partial_{x}f_{1},\partial_{y}f_{1},\partial_{x}f_{2},\partial_{y}f_{2}$.
The idea is to express the variation of $f_{1}$ under movement of
the monodromies $\pm a$ as a linear combination of these six functions,
and to get a nontrivial equality by looking at the dependent coefficient
of $Z_{-\frac{1}{2}}^{i}$. First, we augment the expansions (\ref{eq:A}),
(\ref{eq:B}): $f_{\left[\mathbb{C},-a,a\right]}$ is equal to ($\mathcal{C}^{1,i}$
are real constants, unrelated to the discrete notation $\mathcal{C}^{1,i}\left[\Omega_{\delta}\right]$)
\begin{alignat*}{1}
 & Z_{-\frac{1}{2}}^{1}(z+a)+2\mathcal{A}^{1}Z_{\frac{1}{2}}^{1}(z+a)+2\mathcal{A}^{i}Z_{\frac{1}{2}}^{i}(z+a)+2\mathcal{D}^{1}Z_{\frac{3}{2}}^{1}(z+a)\\
 & +2\mathcal{D}^{i}Z_{\frac{3}{2}}^{i}(z+a)+O\left((z+a)^{5/2}\right)\text{ near \ensuremath{z=-a},}\\
 & \mathcal{B}_{\mathbb{C}}Z_{-\frac{1}{2}}^{i}(z-a)+2\mathcal{C}^{1}Z_{\frac{1}{2}}^{1}(z-a)+2\mathcal{C}^{i}Z_{\frac{1}{2}}^{i}(z-a)+2\mathcal{E}^{1}Z_{\frac{3}{2}}^{1}(z-a)\\
 & +2\mathcal{E}^{i}Z_{\frac{3}{2}}^{i}(z-a)+O\left((z-a)^{5/2}\right)\text{ near }z=a.
\end{alignat*}

We need to make some reductions. Let us first note that $f_{1}(z)$
and $\bar{f_{1}}(\bar{z})$ both solves the boundary value problem
of Proposition \ref{prop:contbvp} on $\left[\mathbb{C},-a,a\right]$,
and thus are equal to each other; since only $Z_{\nu}^{i}$ switches
sign under the same transformation, we can conclude $\mathcal{A}^{i}=\mathcal{C}^{1}=\mathcal{D}^{i}=\mathcal{E}^{1}=0$.

Similarly, $f_{2}(z)$ and $if_{1}(e^{i\pi}z)$ are equal, assuming
we carefully track the interaction between global rotation $z\mapsto e^{i\pi}z$
and the expansion bases $Z_{\nu}^{\tau}(\cdot\pm a)$; given a sign
choice of $Z_{\nu}^{\tau}(z+a)$ (which is fixed by the coefficient
$+1$ of $Z_{-\frac{1}{2}}^{1}(z+a)$) we will define $Z_{-\frac{1}{2}}^{\tau}(z-a):=iZ_{-\frac{1}{2}}^{\tau}(e^{i\pi}z+a)$,
which then fixes signs for general $\nu$ by $Z_{\nu}^{\tau}(e^{i\pi}z\pm a)=\pm e^{i\nu\pi}Z_{\nu}^{\tau}(z\mp a)$.
As a result we have $\mathcal{A}_{\mathbb{C}}^{1,i}(a,-a)=-\mathcal{A}_{\mathbb{C}}^{1,i}(-a,a)=-\mathcal{A}^{1,i}$
, $\mathcal{B}_{\mathbb{C}}(a,-a)=-\mathcal{B}_{\mathbb{C}}(-a,a)=-\mathcal{B}$,
$\mathcal{C}_{\mathbb{C}}^{1,i}(a,-a)=\mathcal{C}_{\mathbb{C}}^{1,i}(-a,a)=\mathcal{C}^{1,i}$,
$\mathcal{D}_{\mathbb{C}}^{1,i}(a,-a)=\mathcal{D}_{\mathbb{C}}^{1,i}(-a,a)=\mathcal{D}^{1,i}$,
$\mathcal{E}_{\mathbb{C}}^{1,i}(a,-a)=-\mathcal{E}_{\mathbb{C}}^{1,i}(-a,a)=-\mathcal{E}^{1,i}$.

In fact, given $Z_{\nu}^{\tau}(z+a)$, if we define $Z_{-\frac{1}{2}}^{\tau}(z\pm e^{i\phi}a):=e^{-\frac{i\phi}{2}}Z_{-\frac{1}{2}}^{\tau}(e^{-i\phi}z\pm a)$
for small $\left|\phi\right|$, $f_{\left[\mathbb{C},-ae^{i\phi},ae^{i\phi}\right]}(z)=R_{\phi}f_{\left[\mathbb{C},-a,a\right]}(z)$
again from (\ref{eq:rot}):

\begin{alignat*}{1}
 & f_{\left[\mathbb{C},-ae^{i\phi},ae^{i\phi}\right]}(z)=R_{\phi}f_{1}(z)=Z_{-\frac{1}{2}}^{1}(z+ae^{i\phi})\\
 & +2\cos\phi\mathcal{A}^{1}Z_{\frac{1}{2}}^{1}(z+ae^{i\phi})+2\sin\phi\mathcal{A}^{1}Z_{\frac{1}{2}}^{i}(z+ae^{i\phi})+2\cos2\phi\mathcal{D}^{1}Z_{\frac{3}{2}}^{1}(z+ae^{i\phi})\\
 & +2\sin2\phi\mathcal{D}^{1}Z_{\frac{3}{2}}^{i}(z+ae^{i\phi})+O\left((z+ae^{i\phi})^{5/2}\right)\text{ near \ensuremath{z=-a},}\text{ and}\\
 & f_{\left[\mathbb{C},-ae^{i\phi},ae^{i\phi}\right]}(z)=\mathcal{B}_{\Omega}Z_{-\frac{1}{2}}^{i}(z-ae^{i\phi})+2\cos\phi\mathcal{C}^{i}Z_{\frac{1}{2}}^{i}(z-ae^{i\phi})-2\sin\phi\mathcal{C}^{i}Z_{\frac{1}{2}}^{1}(z-ae^{i\phi})\\
 & +2\cos2\phi\mathcal{E}^{i}Z_{\frac{3}{2}}^{i}(z-ae^{i\phi})-2\sin2\phi\mathcal{E}^{i}Z_{\frac{3}{2}}^{1}(z-ae^{i\phi})+O\left(\left(z-ae^{i\phi}\right)^{5/2}\right)\text{ near }z=a.
\end{alignat*}

\subsubsection*{Expansion.}

We are now ready to analyse the variation of $R_{\phi}f_{1}$ under
both $\partial_{a}$ and $\partial_{\phi}$ at $\phi=0$.

Around $z=-a$,
\begin{alignat*}{1}
f_{1}(z) & =Z_{-\frac{1}{2}}^{1}(z+a)+2\mathcal{A}^{1}Z_{\frac{1}{2}}^{1}(z+a)+2\mathcal{D}^{1}Z_{\frac{3}{2}}^{1}(z+a)+O\left(\left(z+a\right)^{5/2}\right),\\
\partial_{x}f_{1}(z) & =-\frac{1}{2}Z_{-\frac{3}{2}}^{1}(z+a)+\mathcal{A}^{1}Z_{-\frac{1}{2}}^{1}(z+a)+\left(3\mathcal{D}^{1}+2m^{2}\right)Z_{\frac{1}{2}}^{1}(z+a)+O\left(\left(z+a\right)^{3/2}\right),\\
\partial_{y}f_{1}(z) & =-\frac{1}{2}Z_{-\frac{3}{2}}^{i}(z+a)+\mathcal{A}^{1}Z_{-\frac{1}{2}}^{i}(z+a)+\left(3\mathcal{D}^{1}-2m^{2}\right)Z_{\frac{1}{2}}^{i}(z+a)+O\left(\left(z+a\right)^{3/2}\right),
\end{alignat*}
while around $a$,
\begin{alignat*}{1}
f_{1}(z) & =\mathcal{B}Z_{-\frac{1}{2}}^{i}(z-a)+2\mathcal{C}^{i}Z_{\frac{1}{2}}^{i}(z-a)+2\mathcal{E}^{i}Z_{\frac{3}{2}}^{i}(z-a)+O\left((z-a)^{5/2}\right),\\
\partial_{x}f_{1}(z) & =-\frac{\mathcal{B}}{2}Z_{-\frac{3}{2}}^{i}(z-a)+\mathcal{C}^{i}Z_{-\frac{1}{2}}^{i}(z-a)+\left(3\mathcal{E}^{i}+2m^{2}\mathcal{B}\right)Z_{\frac{1}{2}}^{i}(z-a)+O\left(\left(z-a\right)^{3/2}\right),\\
\partial_{y}f_{1}(z) & =\frac{\mathcal{B}}{2}Z_{-\frac{3}{2}}^{1}(z-a)-\mathcal{C}^{i}Z_{-\frac{1}{2}}^{1}(z-a)+\left(2m^{2}\mathcal{B}-3\mathcal{E}^{i}\right)Z_{\frac{1}{2}}^{1}(z-a)+O\left(\left(z-a\right)^{3/2}\right),
\end{alignat*}
and similar formulae hold for $f_{2}$ with $-a$ and $a$ interchanged
and the signs in front of $\mathcal{A},\mathcal{B},\mathcal{E}$ reversed.
As for the varied functions, we have
\begin{alignat*}{1}
\partial_{a}f_{1}(z) & =-\frac{1}{2}Z_{-\frac{3}{2}}^{1}(z+a)+\mathcal{A}^{1}Z_{-\frac{1}{2}}^{1}(z+a)+\left(2\left(\partial_{a}\mathcal{A}^{1}\right)+3\mathcal{D}^{1}+2m^{2}\right)Z_{\frac{1}{2}}^{1}(z+a)\\
 & +O\left((z+a)^{3/2}\right),\\
\partial_{\phi}R_{\phi}f_{1}(z) & =\frac{a}{2}Z_{-\frac{3}{2}}^{i}(z+a)-a\mathcal{A}^{1}Z_{-\frac{1}{2}}^{i}(z+a)+\left(2\mathcal{A}^{1}-3a\mathcal{D}^{1}+2am^{2}\right)Z_{\frac{1}{2}}^{i}(z+a)\\
 & +O\left(\left(z+a\right)^{3/2}\right),
\end{alignat*}
and
\begin{alignat*}{1}
\partial_{a}f_{1}(z) & =\frac{\mathcal{B}}{2}Z_{-\frac{3}{2}}^{i}(z-a)+\left(\partial_{a}\mathcal{B}-\mathcal{C}^{i}\right)Z_{-\frac{1}{2}}^{i}(z-a)+\left(2\left(\partial_{a}\mathcal{C}^{i}\right)-3\mathcal{E}^{i}-2m^{2}\mathcal{B}\right)Z_{\frac{1}{2}}^{i}(z-a)\\
 & +O\left((z-a)^{3/2}\right),\\
\partial_{\phi}R_{\phi}f_{1}(z) & =\frac{a\mathcal{B}}{2}Z_{-\frac{3}{2}}^{1}(z-a)-a\mathcal{C}^{i}Z_{-\frac{1}{2}}^{1}(z-a)+\left(2am^{2}\mathcal{B}-3a\mathcal{E}^{i}-2\mathcal{C}^{i}\right)Z_{\frac{1}{2}}^{1}(z-a)\\
 & +O\left(\left(z-a\right)^{3/2}\right).
\end{alignat*}

The resulting expansions are
\begin{alignat}{1}
\partial_{a}f_{1} & =-\frac{2(\mathcal{A}^{1}\mathcal{B}+\mathcal{C}^{i})\mathcal{B}}{1-\mathcal{B}^{2}}f_{1}+\frac{1+\mathcal{B}^{2}}{1-\mathcal{B}^{2}}\partial_{x}f_{1}-\frac{2\mathcal{B}}{1-\mathcal{B}^{2}}\partial_{y}f_{2},\label{eq:f1}\\
\partial_{\phi}R_{\phi}f_{1} & =-\frac{2a(\mathcal{A}^{1}\mathcal{B}+\mathcal{C}^{i})}{1-\mathcal{B}^{2}}f_{2}-\frac{2a\mathcal{B}}{1-\mathcal{B}^{2}}\partial_{x}f_{2}-a\frac{1+\mathcal{B}^{2}}{1-\mathcal{B}^{2}}\partial_{y}f_{1}.\label{eq:f2}
\end{alignat}

\subsubsection*{Derivation.}

Comparing the coefficients of $Z_{-\frac{1}{2}}^{i},Z_{\frac{1}{2}}^{1},Z_{\frac{1}{2}}^{i}$,
we get from (\ref{eq:f1})
\begin{alignat*}{1}
\partial_{a}\mathcal{B}-\mathcal{C}^{i} & =-\frac{2(\mathcal{A}^{1}\mathcal{B}+\mathcal{C}^{i})\mathcal{B}}{1-\mathcal{B}^{2}}\mathcal{B}+\frac{1+\mathcal{B}^{2}}{1-\mathcal{B}^{2}}\mathcal{C}^{i}+\frac{2\mathcal{B}}{1-\mathcal{B}^{2}}\mathcal{A}^{1},\\
2\left(\partial_{a}\mathcal{A}^{1}\right)+3\mathcal{D}^{1}+\frac{2m^{2}}{3} & =-\frac{2(\mathcal{A}^{1}\mathcal{B}+\mathcal{C}^{i})\mathcal{B}}{1-\mathcal{B}^{2}}2\mathcal{A}^{1}+\frac{1+\mathcal{B}^{2}}{1-\mathcal{B}^{2}}(3\mathcal{D}^{1}+2m^{2})-\frac{2\mathcal{B}}{1-\mathcal{B}^{2}}\left(-2m^{2}\mathcal{B}+3\mathcal{E}^{i}\right),\\
2\left(\partial_{a}\mathcal{C}^{i}\right)-3\mathcal{E}^{i}-\frac{2m^{2}\mathcal{B}}{3} & =-\frac{2(\mathcal{A}^{1}\mathcal{B}+\mathcal{C}^{i})\mathcal{B}}{1-\mathcal{B}^{2}}2\mathcal{C}^{i}+\frac{1+\mathcal{B}^{2}}{1-\mathcal{B}^{2}}\left(3\mathcal{E}^{i}+2m^{2}\mathcal{B}\right)-\frac{2\mathcal{B}}{1-\mathcal{B}^{2}}(3\mathcal{D}^{1}-2m^{2}),
\end{alignat*}

while for (\ref{eq:f2}) we get 
\begin{alignat*}{1}
2\mathcal{A}^{1}-3a\mathcal{D}^{1}+\frac{2am^{2}}{3} & =-\frac{2a(\mathcal{A}^{1}\mathcal{B}+\mathcal{C}^{i})}{1-\mathcal{B}^{2}}2\mathcal{C}^{i}-\frac{2a\mathcal{B}}{1-\mathcal{B}^{2}}\left(-3\mathcal{E}^{i}-2m^{2}\mathcal{B}\right)-a\frac{1+\mathcal{B}^{2}}{1-\mathcal{B}^{2}}(3\mathcal{D}^{1}-2m^{2}),\\
\frac{2am^{2}\mathcal{B}}{3}-3a\mathcal{E}^{i}-2\mathcal{C}^{i} & =\frac{2a(\mathcal{A}^{1}\mathcal{B}+\mathcal{C}^{i})}{1-\mathcal{B}^{2}}2\mathcal{A}^{1}-\frac{2a\mathcal{B}}{1-\mathcal{B}^{2}}(3\mathcal{D}^{1}+2m^{2})-a\frac{1+\mathcal{B}^{2}}{1-\mathcal{B}^{2}}\left(2m^{2}\mathcal{B}-3\mathcal{E}^{i}\right).
\end{alignat*}

We now make the dependence in $m$ explicit. Similarly to above, for
any $k>0$, $f_{\left[\mathbb{C},-ak^{-1},ak^{-1}\right]}(z|mk)=f_{\left[\mathbb{C},-a,a\right]}(kz|m)k^{1/2}$.
Analysing the effect of this dilation, which leaves $r:=am$ fixed,
on the individual coefficients, we can write $\mathcal{A}^{1}(a,m)=:\text{ }m\mathcal{A}_{0}(r)$,
$\mathcal{B}(a,m)=:\text{ }\mathcal{B}_{0}(r)$, $\mathcal{C}^{i}(a,m)=:\text{ }m\mathcal{C}_{0}(r)$.
Then we have $\partial_{a}\mathcal{A}^{1}=m^{2}\mathcal{A}_{0}'$,
$\partial_{a}\mathcal{B}=m\mathcal{B}_{0}'$, $\partial_{a}\mathcal{C}^{1}=m^{2}\mathcal{C}_{0}'$.
In terms of these functions, we have
\begin{alignat}{1}
\mathcal{B}_{0}' & =-\frac{2(\mathcal{A}_{0}\mathcal{B}_{0}+\mathcal{C}_{0})}{1-\mathcal{B}_{0}^{2}}\mathcal{B}_{0}^{2}+\frac{2}{1-\mathcal{B}_{0}^{2}}\mathcal{C}_{0}+\frac{2\mathcal{B}_{0}}{1-\mathcal{B}_{0}^{2}}\mathcal{A}_{0}=2\mathcal{A}_{0}\mathcal{B}_{0}+2\mathcal{C}_{0},\label{eq:1}\\
\mathcal{A}_{0}' & =-\frac{2\mathcal{B}_{0}(\mathcal{A}_{0}\mathcal{B}_{0}+\mathcal{C}_{0})}{1-\mathcal{B}_{0}^{2}}\mathcal{A}_{0}+\frac{4\mathcal{B}_{0}^{2}}{\left(1-\mathcal{B}_{0}^{2}\right)}+\frac{3m^{-2}\mathcal{B}}{1-\mathcal{B}^{2}}(\mathcal{B}\mathcal{D}^{1}-\mathcal{E}^{i}),\label{eq:2}\\
\mathcal{C}_{0}' & =-\frac{2\mathcal{B}_{0}(\mathcal{A}_{0}\mathcal{B}_{0}+\mathcal{C}_{0})}{1-\mathcal{B}_{0}^{2}}\mathcal{C}_{0}+\frac{4\mathcal{B}_{0}}{\left(1-\mathcal{B}_{0}^{2}\right)}-\frac{3m^{-2}}{1-\mathcal{B}^{2}}(\mathcal{B}\mathcal{D}^{1}-\mathcal{E}^{i}),\label{eq:3}\\
\mathcal{A}_{0} & =-\frac{2r(\mathcal{A}_{0}\mathcal{B}_{0}+\mathcal{C}_{0})}{1-\mathcal{B}_{0}^{2}}\mathcal{C}_{0}+\frac{4r\mathcal{B}_{0}^{2}}{\left(1-\mathcal{B}_{0}^{2}\right)}-\frac{3am^{-1}\mathcal{B}}{1-\mathcal{B}^{2}}(\mathcal{B}\mathcal{D}^{1}-\mathcal{E}^{i}),\label{eq:4}\\
\mathcal{C}_{0} & =-\frac{2r(\mathcal{A}_{0}\mathcal{B}_{0}+\mathcal{C}_{0})}{1-\mathcal{B}_{0}^{2}}\mathcal{A}_{0}+\frac{4r\mathcal{B}_{0}}{\left(1-\mathcal{B}_{0}^{2}\right)}+\frac{3am^{-1}}{1-\mathcal{B}^{2}}(\mathcal{B}\mathcal{D}^{1}-\mathcal{E}^{i}).\label{eq:5}
\end{alignat}

Define $\mathcal{B}_{0}=:\text{ }\tanh h_{0}$. Then $\frac{\mathcal{B}_{0}'}{1-\mathcal{B}_{0}^{2}}=h_{0}'$
and $\frac{4\mathcal{B}_{0}^{2}}{\left(1-\mathcal{B}_{0}^{2}\right)^{2}}=\sinh^{2}2h_{0}$.
From (\ref{eq:4}),(\ref{eq:5}),
\begin{alignat}{1}
\mathcal{A}_{0}+\mathcal{B}_{0}\frac{\mathcal{A}_{0}\mathcal{B}_{0}+\mathcal{C}_{0}}{1-\mathcal{B}_{0}^{2}}=\frac{\mathcal{A}_{0}+\mathcal{B}_{0}\mathcal{C}_{0}}{1-\mathcal{B}_{0}^{2}} & =-r\left[\frac{2(\mathcal{A}_{0}\mathcal{B}_{0}+\mathcal{C}_{0})^{2}}{\left(1-\mathcal{B}_{0}^{2}\right)^{2}}-\frac{8\mathcal{B}_{0}^{2}}{\left(1-\mathcal{B}_{0}^{2}\right)^{2}}\right],\label{eq:r1}
\end{alignat}
and noting (\ref{eq:1}),
\begin{equation}
\mathcal{A}_{0}=-\frac{1}{2}\left(\ln\cosh h_{0}\right)'-r\left[\frac{1}{2}\left(h_{0}'\right)^{2}-2\sinh^{2}2h_{0}\right].\label{eq:painleve_A}
\end{equation}

To characterise $h$, first combine (\ref{eq:2}),(\ref{eq:3}) to
get
\begin{alignat*}{1}
\frac{\mathcal{A}_{0}'+\mathcal{C}_{0}'\mathcal{B}_{0}}{1-\mathcal{B}_{0}^{2}} & =-\frac{2(\mathcal{A}_{0}\mathcal{B}_{0}+\mathcal{C}_{0})(\mathcal{A}_{0}\mathcal{B}_{0}+\mathcal{B}_{0}^{2}\mathcal{C}_{0})}{\left(1-\mathcal{B}_{0}^{2}\right)^{2}}+\frac{8\mathcal{B}_{0}^{2}}{\left(1-\mathcal{B}_{0}^{2}\right)^{2}},\\
 & =-\left[\frac{2(\mathcal{A}_{0}\mathcal{B}_{0}+\mathcal{C}_{0})^{2}}{\left(1-\mathcal{B}_{0}^{2}\right)^{2}}-\frac{8\mathcal{B}_{0}^{2}}{\left(1-\mathcal{B}_{0}^{2}\right)^{2}}\right]+\frac{\mathcal{B}_{0}'\mathcal{C}_{0}}{1-\mathcal{B}_{0}^{2}},
\end{alignat*}
then differentiate (\ref{eq:r1}) to get
\begin{alignat*}{1}
 & \frac{\mathcal{A}_{0}'+\mathcal{C}_{0}'\mathcal{B}_{0}+\mathcal{B}_{0}'\mathcal{C}_{0}}{1-\mathcal{B}_{0}^{2}}+\frac{2\mathcal{B}_{0}\mathcal{B}_{0}'\left(\mathcal{A}_{0}+\mathcal{B}_{0}\mathcal{C}_{0}\right)}{\left(1-\mathcal{B}_{0}^{2}\right)^{2}}\\
= & -\left[\frac{2(\mathcal{A}_{0}\mathcal{B}_{0}+\mathcal{C}_{0})^{2}}{\left(1-\mathcal{B}_{0}^{2}\right)^{2}}-\frac{8\mathcal{B}_{0}^{2}}{\left(1-\mathcal{B}_{0}^{2}\right)^{2}}\right]-r\left[\frac{2(\mathcal{A}_{0}\mathcal{B}_{0}+\mathcal{C}_{0})^{2}}{\left(1-\mathcal{B}_{0}^{2}\right)^{2}}-\frac{8\mathcal{B}_{0}^{2}}{\left(1-\mathcal{B}_{0}^{2}\right)^{2}}\right]'.
\end{alignat*}

Then combining the two we finally have
\[
\frac{2\mathcal{B}_{0}'\left(\mathcal{A}_{0}\mathcal{B}_{0}+\mathcal{C}_{0}\right)}{\left(1-\mathcal{B}_{0}^{2}\right)^{2}}=-r\left[\frac{2(\mathcal{A}_{0}\mathcal{B}_{0}+\mathcal{C}_{0})^{2}}{\left(1-\mathcal{B}_{0}^{2}\right)^{2}}-\frac{8\mathcal{B}_{0}^{2}}{\left(1-\mathcal{B}_{0}^{2}\right)^{2}}\right]',
\]
or
\[
\left(h_{0}'\right)^{2}=-r\left[\frac{1}{2}\left(h_{0}'\right)^{2}-2\sinh^{2}2h_{0}\right]'.
\]

Simplifying, we have $h_{0}''+\frac{h_{0}'}{r}=4\sinh4h_{0}(r)$.
This is equivalent to the Painlevé III equation $r\eta_{0}\eta_{0}''=r\left(\eta_{0}'\right)^{2}-\eta_{0}\eta_{0}'-4r+4r\eta_{0}^{4}$
by a change of variables $\eta_{0}=e^{-2h_{0}}$ \cite[(4.12)]{kako80}. 

\appendix

\section{Harmonicity Estimates\label{sec:Appendix:-Discrete-Green's}}

\begin{figure}[b]
\includegraphics[width=0.9\columnwidth]{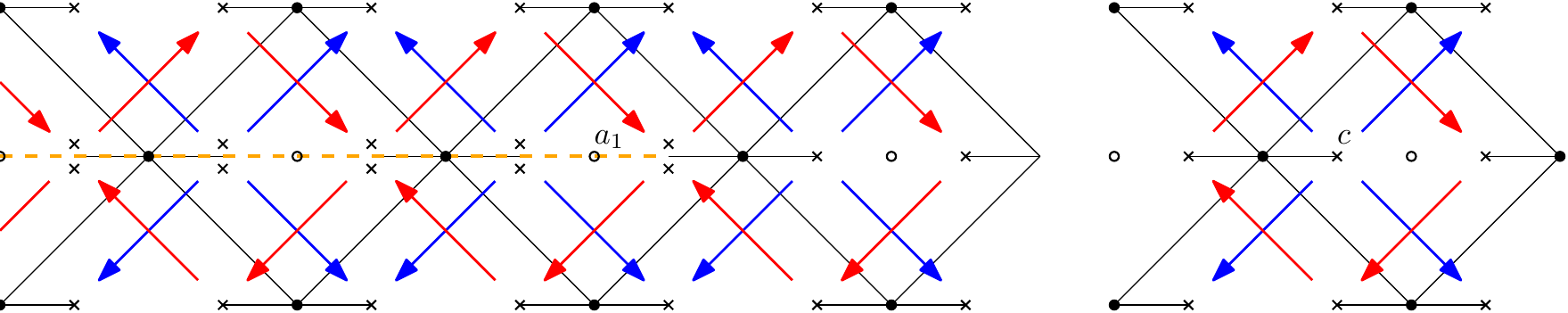}\caption{\label{fig:Using-holomorphicity-to}Using holomorphicity to get harmonicity.
Blue differences (laplacian) are turned into red, which telescope.
Left: in the presence of a branch cut (orange). Right: simple planar
case.}

\end{figure}

In this Appendix, we collect together the discrete analytic calculations
and estimates used in the paper. Fix a discrete simply connected planar
graph $G_{\delta}$, which can be thought of as a subgraph of $\Omega_{\delta}$
or $\left[\Omega_{\delta},a_{1},\ldots,a_{n}\right]$ .
\begin{prop}
\label{prop:Appendix}A massive s-holomorphic function $F:\mathcal{EC}\left[G_{\delta}\right]\to\mathbb{C}$
is \emph{massive discrete holomorphic}, that is to say

\begin{alignat}{1}
\cos\left(\frac{\pi}{4}+2\Theta\right)F(r_{+})-\cos\left(\frac{\pi}{4}-2\Theta\right)F(r_{-}) & =-i\left(\cos\left(\frac{\pi}{4}+2\Theta\right)F(i_{+})-\cos\left(\frac{\pi}{4}-2\Theta\right)F(i_{-})\right),\nonumber \\
\cos\left(\frac{\pi}{4}-2\Theta\right)F(i_{+}')-\cos\left(\frac{\pi}{4}+2\Theta\right)F(i_{-}') & =-i\left(\cos\left(\frac{\pi}{4}-2\Theta\right)F(r_{+}')-\cos\left(\frac{\pi}{4}+2\Theta\right)F(r_{-}')\right),\label{eq:mdhol}
\end{alignat}
if there is a $\lambda$-corner $c$ such that $r_{\pm}=c\pm\frac{\delta+\delta i}{2}$
(real corners) and $i_{\pm}=c\pm\frac{-\delta+\delta i}{2}$ (imaginary
corners), or a $\bar{\lambda}$-corner $c'$ such that $i_{\pm}'=c\pm\frac{\delta+\delta i}{2}$
and $r_{\pm}'=c\pm\frac{-\delta+\delta i}{2}$ (resp. imaginary and
real corners).

It is \emph{massive harmonic}, i.e.
\begin{alignat}{1}
\Delta^{\delta}F(c) & =2\left(\frac{\cos\left(\frac{\pi}{4}-2\Theta\right)}{\cos\left(\frac{\pi}{4}+2\Theta\right)}+\frac{\cos\left(\frac{\pi}{4}+2\Theta\right)}{\cos\left(\frac{\pi}{4}-2\Theta\right)}-2\right)F(c)\label{eq:dmharm}\\
 & =\left(\frac{8\sin^{2}2\Theta}{\cos4\Theta}\right)F(c)=:M_{H}^{2}F(c)\text{ for }c\in\mathcal{C}^{1,i}\left[G_{\delta}\right].\nonumber 
\end{alignat}

In addition, its square satisfies
\begin{equation}
\partial_{\bar{z}}^{\delta}F^{2}\left(x\right)=\begin{cases}
A_{\Theta}\sum_{n=0}^{3}\left|F\left(x+i^{n}\frac{\delta}{2}\right)\right|^{2}+B_{\Theta}\left|\partial_{\bar{z}}\bar{F}\right|^{2}(x) & x\in\mathcal{F}\left[\Omega_{\delta}\right]\setminus\left\{ a_{2},\ldots,a_{n}\right\} \\
-A_{-\Theta}\sum_{n=0}^{3}\left|F\left(x+i^{n}\frac{\delta}{2}\right)\right|^{2}-B_{-\Theta}\left|\partial_{\bar{z}}\bar{F}\right|^{2}(x) & x\in\mathcal{V}\left[\Omega_{\delta}\right]\setminus\left\{ a_{1}+\delta\right\} 
\end{cases}.\label{eq:squareint_dbar}
\end{equation}
where $A_{\Theta}=\frac{2\left(\sqrt{2}\cos\left(\frac{\pi}{4}-2\Theta\right)-1\right)}{\sqrt{2}\cos^{2}\Theta\cos^{2}\left(\frac{\pi}{4}+2\Theta\right)}$,
$B_{\Theta}=\frac{1}{2\sqrt{2}\cos^{2}\Theta}$. 
\end{prop}

\begin{proof}
For the first line in (\ref{eq:mdhol}), note that by massive s-holomorphicity
we have the edge values $F\left(\frac{r_{+}+i_{-}}{2}\right)=e^{-i\Theta}\left[F(r_{+})+F(i_{-})\right]$
and $F\left(\frac{i_{+}+r_{-}}{2}\right)=e^{i\Theta}\left[F(i_{+})+F(r_{-})\right]$.
Since $F$ is s-holomorphic at the $\lambda$-corner $c$, which is
adjacent to both of them, writing $e^{-i\Theta}\text{Proj}_{e^{i\Theta}\lambda\mathbb{R}}F\left(\frac{r_{+}+i_{-}}{2}\right)=e^{i\Theta}\text{Proj}_{e^{-i\Theta}\lambda\mathbb{R}}F\left(\frac{i_{+}+r_{-}}{2}\right)$,
equivalent to $\frac{1}{2}\left[e^{-2i\Theta}\left(F(r_{+})+F(i_{-})\right)+ie^{2i\Theta}\left(F(r_{+})-F(i_{-})\right)\right]=\frac{1}{2}\left[e^{2i\Theta}\left(F(i_{+})+F(r_{-})\right)+ie^{-2i\Theta}\left(-F(i_{+})+F(r_{-})\right)\right]$,
and rearranging gives the result. For the second line, notice that
$iF$ is $(-\Theta)$-massive s-holomorphic if we move to the dual
graph $G_{\delta}^{*}$ (i.e. $\mathcal{V}(G_{\delta}^{*}):=\mathcal{F}(G_{\delta})$).
Since this duality transformation converts $\bar{\lambda}$-corners
into $\lambda$-corners, we can use the previous calculation.

For (\ref{eq:dmharm}), suppose $c$ is a real corner. Take four copies
of the previous result (\ref{eq:mdhol}) (see Figure \ref{fig:Using-holomorphicity-to})
around for each of the four middle corners $c\pm\frac{\delta\pm i\delta}{2}$.
Each of them involve $c$ and one of the four neighbouring real corners
$c\pm\left(\delta\pm i\delta\right)$; summing the four equations
with scalar factors so that the coefficients of $F\left(c\pm\left(\delta\pm i\delta\right)\right)$
in each equation is $1$, the result is straightforward. The case
where $c$ is imaginary is immediate from duality as above.

For (\ref{eq:squareint_dbar}), take $x\in\mathcal{F}\left[G_{\delta}\right]$
and note that the value at each of the neighbouring edges $x+i^{n}\frac{\lambda\delta}{\sqrt{2}}$
can be reconstructed from two of the four corners $x+\frac{i^{n}\delta}{2}$.
Explicitly, inverting s-holomorphicity projections give
\begin{alignat*}{1}
\cos\left(\frac{\pi}{4}+2\Theta\right)\lambda i^{n+1}F\left(x+i^{n}\frac{\lambda\delta}{\sqrt{2}}\right) & =e^{i\Theta}\bar{F}\left(x+\frac{i^{n}}{2}\delta\right)-e^{-i\Theta}\bar{F}\left(x+\frac{i^{n+1}}{2}\delta\right)\\
=-i^{n} & \left[e^{i\Theta}F\left(x+\frac{i^{n}}{2}\delta\right)-e^{-i\Theta}iF\left(x+\frac{i^{n+1}}{2}\delta\right)\right].
\end{alignat*}
noting that $\bar{F}\left(x+\frac{i^{n}}{2}\delta\right)=-i^{n}F\left(x+\frac{i^{n}}{2}\delta\right)$.

So multiplying the two lines
\begin{alignat*}{1}
 & \cos^{2}\left(\frac{\pi}{4}+2\Theta\right)i^{2n+2}\lambda^{2}F\left(x+i^{n}\frac{\lambda\delta}{\sqrt{2}}\right)^{2}=\\
 & -i^{n}\cdot\left[e^{2i\Theta}\left|F\left(x+\frac{i^{n}}{2}\delta\right)\right|^{2}+ie^{-2i\Theta}\left|F\left(x+\frac{i^{n+1}}{2}\delta\right)\right|^{2}\right.\\
 & \left.-F\left(x+\frac{i^{n}}{2}\delta\right)\bar{F}\left(x+\frac{i^{n+1}}{2}\delta\right)-i\bar{F}\left(x+\frac{i^{n}}{2}\delta\right)F\left(x+\frac{i^{n+1}}{2}\delta\right)\right]\\
 & =-i^{n}\cdot\left[e^{2i\Theta}\left|F\left(x+\frac{i^{n}}{2}\delta\right)\right|^{2}+ie^{-2i\Theta}\left|F\left(x+\frac{i^{n+1}}{2}\delta\right)\right|^{2}\right.\\
 & \left.+2i^{n+1}F\left(x+\frac{i^{n}}{2}\delta\right)F\left(x+\frac{i^{n+1}}{2}\delta\right)\right].
\end{alignat*}

So

\begin{alignat*}{1}
\cos^{2}\left(\frac{\pi}{4}+2\Theta\right)\partial_{\bar{z}}^{\delta}F\left(x\right)^{2} & =\cos^{2}\left(\frac{\pi}{4}+2\Theta\right)\sum_{n=0}^{3}i^{n}\lambda F\left(x+i^{n}\frac{\lambda\delta}{\sqrt{2}}\right)^{2}\\
= & \lambda^{-1}\left(e^{2i\Theta}+ie^{-2i\Theta}\right)\sum_{n=0}^{3}\left|F\left(x+\frac{i^{n}}{2}\delta\right)\right|^{2}\\
 & +2\lambda^{-1}\sum_{n=0}^{3}i^{n+1}F\left(x+\frac{i^{n}}{2}\delta\right)F\left(x+\frac{i^{n+1}}{2}\delta\right)\\
=2\cos\left(\frac{\pi}{4}-2\Theta\right)\sum_{n=0}^{3}\left|F\left(x+\frac{i^{n}}{2}\delta\right)\right|^{2} & +2\lambda^{-1}\sum_{n=0}^{3}i^{n+1}F\left(x+\frac{i^{n}}{2}\delta\right)F\left(x+\frac{i^{n+1}}{2}\delta\right).
\end{alignat*}

Now reuse the first relation
\begin{alignat*}{1}
\cos\left(\frac{\pi}{4}+2\Theta\right)\lambda^{-1}i^{-n}F\left(x+i^{n}\frac{\lambda\delta}{\sqrt{2}}\right) & =-i^{-2n}\left[e^{i\Theta}\bar{F}\left(x+\frac{i^{n}}{2}\delta\right)-e^{-i\Theta}\bar{F}\left(x+\frac{i^{n+1}}{2}\delta\right)\right]\\
\cos\left(\frac{\pi}{4}+2\Theta\right)\lambda i^{n}\bar{F}\left(x+i^{n}\frac{\lambda\delta}{\sqrt{2}}\right) & =(-1)^{n+1}\left[e^{-i\Theta}F\left(x+\frac{i^{n}}{2}\delta\right)-e^{i\Theta}F\left(x+\frac{i^{n+1}}{2}\delta\right)\right]\\
\cos\left(\frac{\pi}{4}+2\Theta\right)\partial_{\bar{z}}\bar{F}\left(x\right) & =\left(e^{i\Theta}+e^{-i\Theta}\right)\sum_{n=0}^{3}(-1)^{n+1}F\left(x+\frac{i^{n}\delta}{2}\right)\\
 & =2\cos\Theta\sum_{n=0}^{3}(-1)^{n+1}F\left(x+\frac{i^{n}\delta}{2}\right).
\end{alignat*}

Taking squares
\begin{alignat*}{1}
\cos^{2}\left(\frac{\pi}{4}+2\Theta\right)\left|\partial_{\bar{z}}\bar{F}\left(x\right)\right|^{2} & =4\cos^{2}\Theta\left[\sum_{n=0}^{3}\left|F\left(x+\frac{i^{n}\delta}{2}\right)\right|^{2}\right.\\
+\text{Re} & \left.\sum_{n\neq n'}(-1)^{n+1}F\left(x+\frac{i^{n}\delta}{2}\right)(-1)^{n'+1}\bar{F}\left(x+\frac{i^{n'}\delta}{2}\right)\right]\\
 & =4\cos^{2}\Theta\left[\sum_{n=0}^{3}\left|F\left(x+\frac{i^{n}\delta}{2}\right)\right|^{2}\right.\\
-\text{Re} & \left.\sum_{n\neq n'}(-1)^{n+n'}i^{n'}F\left(x+\frac{i^{n}\delta}{2}\right)F\left(x+\frac{i^{n'}\delta}{2}\right)\right],
\end{alignat*}
but $i^{n'}F\left(x+\frac{i^{n}\delta}{2}\right)F\left(x+\frac{i^{n'}\delta}{2}\right)\in i\mathbb{R}$
if $|n-n'|=2$. The remaining $8$ combinations of $n,n'$ all give
rise to purely real terms, and resumming gives
\begin{alignat*}{1}
\text{Re}\sum_{n\neq n'}(-1)^{n+n'}i^{n'}F\left(x+\frac{i^{n}\delta}{2}\right)F\left(x+\frac{i^{n'}\delta}{2}\right) & =-\sum_{n=0}^{3}\left(i^{n}+i^{n+1}\right)F\left(x+\frac{i^{n}\delta}{2}\right)F\left(x+\frac{i^{n+1}\delta}{2}\right)\\
\cos^{2}\left(\frac{\pi}{4}+2\Theta\right)\left|\partial_{\bar{z}}\bar{F}\left(x\right)\right|^{2} & =4\cos^{2}\Theta\left[\sum_{n=0}^{3}\left|F\left(x+\frac{i^{n}\delta}{2}\right)\right|^{2}\right.\\
 & \left.+\sum_{n=0}^{3}\sqrt{2}i^{n+1}\lambda^{-1}F\left(x+\frac{i^{n}\delta}{2}\right)F\left(x+\frac{i^{n+1}\delta}{2}\right)\right].
\end{alignat*}

Comparing the two expressions
\begin{alignat*}{1}
\cos^{2}\left(\frac{\pi}{4}+2\Theta\right)\left[2\sqrt{2}\cos^{2}\Theta\cdot\partial_{\bar{z}}^{\delta}F\left(x\right)^{2}-\left|\partial_{\bar{z}}\bar{F}\left(x\right)\right|^{2}\right]\\
=4\left(\sqrt{2}\cos\left(\frac{\pi}{4}-2\Theta\right)-1\right)\sum_{n=0}^{3}\left|F\left(x+\frac{i^{n}}{2}\delta\right)\right|^{2},
\end{alignat*}
we have the full result given duality.
\end{proof}
\begin{rem}
\label{rem:mdhol}(\ref{eq:mdhol}) is equivalent to massive s-holomorphicity
in the sense that if we have such values of $F$ on $\mathcal{C}^{1,i}\left[G_{\delta}\right]$
then it is easy to see from the proof that we have enough data to
extend the values s-holomorphically first to $\mathcal{E}\left[G_{\delta}\right]$
and then the $\lambda,\bar{\lambda}$-corners. In other words, bound
on $\mathcal{C}^{1,i}$ is equivalent to a global bound in an s-holomorphic
function. On $\mathcal{E}\left[G_{\delta}\right]$, (\ref{eq:mdhol})
becomes
\begin{alignat*}{1}
\partial_{\bar{z}}^{\delta}F(x) & :=\sum_{n=0}^{3}i^{n}e^{i\pi/4}F\left(x+i^{n}e^{i\pi/4}\frac{\delta}{\sqrt{2}}\right)\\
 & =\sin\Theta\sec\left(\frac{\pi}{4}+2\Theta\right)\sum_{n=0}^{3}\overline{F\left(x+i^{n}e^{i\pi/4}\frac{\delta}{\sqrt{2}}\right)},
\end{alignat*}
i.e. a discretised version of $\partial_{\bar{z}}f=m\bar{f}$ given
$\Theta\sim\frac{m\delta}{2}$.
\end{rem}

\begin{lem}
\label{lem:harm-boundary}Suppose $\Omega'\subset\Omega$ are smooth
simply connected domains. Any function $H_{0}$ on $\mathcal{V}\left[\left(\Omega\setminus\Omega'\right)_{\delta}\right]$
which is harmonic and takes the boundary value $0$ on $\partial\mathcal{V}\left[\Omega_{\delta}\right]$
and $1$ on $\partial\mathcal{V}\left[\Omega'_{\delta}\right]$ satisfies
$0\leq H_{0}(a_{\text{int}})\leq C(\Omega,\Omega')\delta$ on any
$a_{\text{int}}\in\mathcal{V}\left[\left(\Omega\setminus\Omega'\right)_{\delta}\right]$
adjacent to $\partial\mathcal{V}\left[\Omega_{\delta}\right]$ for
a constant $C(\Omega,\Omega')$ .
\end{lem}

\begin{proof}
We believe this lemma is standard. One possible proof would proceed
by mapping $\Omega\setminus\Omega'$ to the annulus $B_{1}\setminus B_{r_{0}}$
for some $r_{0}>0$ using a Riemann map which smoothly extends to
$\overline{B_{1}\setminus B_{r_{0}}}$. The radial function $\frac{1-r}{1-r_{0}}$
on $B_{1}\setminus B_{r_{0}}$ is superharmonic, so its composition
with the Riemann map is (continuous) superharmonic on $\Omega\setminus\Omega'$;
the restriction to $\mathcal{V}\left[\left(\Omega\setminus\Omega'\right)_{\delta}\right]$
is discrete superharmonic for small enough $\delta$ since the discrete
laplacian (suitably renormalised) and continuous laplacian are uniformly
close on smooth functions, and we can use it to upper bound $H_{0}$.
\end{proof}
We frequently have local $L^{2}$-bounds for our function $F$; it
turns out that thanks to massive harmonicity, this is sufficient for
equicontinuity.

We estimate massive harmonic functions by using the \emph{massive
random walk}, a simple random walk which is extinguished at each step
with probability $\left(1+\frac{2\sin^{2}2\Theta}{\cos(4\Theta)}\right)^{-1}\frac{2\sin^{2}2\Theta}{\cos(4\Theta)}$.
Recall that the massive harmonic measure $\text{hm}{}_{A}^{a}(z|\Theta)$
for a discrete domain $A$, $z\in A$, and $a\in\partial A\cup A$
is the probability of a massive random walk started at $z$ hitting
$a$ before $\partial A\setminus\left\{ a\right\} $. It is the unique
$\Theta$-massive harmonic function on $A$ which takes the boundary
value $1$ at $a$ and $0$ on $\partial A\setminus\left\{ a\right\} $.
In the scaling limit $\delta\downarrow0$ and $\Theta\sim\frac{m\delta}{2}$,
the massive random walk is extinguished after an exponential step
of mean $\frac{1}{2m^{2}\delta^{2}}$. Taking into account the square-root
scaling for the random walk, this corresponds to a distance of order
$\sqrt{2}\delta\cdot\frac{1}{\sqrt{2}\left|m\right|\delta}=\frac{1}{\left|m\right|}$.
For more precise asymptotics than below, we refer to \cite{bdtr}.
\begin{prop}
\label{prop:hmestimate}There are constants $C,C',c>0$ such that,
for a real massive harmonic function $F:\mathcal{C}^{1}\left[\left(B_{R}\right)_{\delta}\right]\to\mathbb{C}$
(where $\delta<\frac{R}{4}$) with $\Delta^{\delta}F=M_{H}^{2}(\Theta)F$,
\begin{alignat}{1}
\left|F\left(-\frac{\delta}{2}\right)\right| & \leq Ce^{cmR}\sqrt{\frac{L}{R}},\label{eq:harnack}\\
\delta^{-1}\left|F\left(\frac{\delta}{2}+i\delta\right)-F\left(-\frac{\delta}{2}\right)\right| & \leq C'e^{cmR}\sqrt{\frac{L}{R^{3}}},\nonumber 
\end{alignat}
where $L=\sum_{c\in\mathcal{C}^{1}\left[\left(B_{R}\right)_{\delta}\right]}\left|F(c)\right|^{2}\delta^{2}$.
\end{prop}

\begin{proof}
For the first bound, note that $F^{2}\geq0$ is subharmonic: 
\begin{equation}
\Delta^{\delta}F^{2}(c)=\left(2M_{H}^{2}+\frac{M_{H}^{4}}{4}\right)F^{2}(c)+\frac{1}{2}\sum_{c\sim x,y}\left[F(x)-F(y)\right]^{2}.\label{eq:f2-subharm}
\end{equation}
So we can use the mean value property for harmonic functions: for
$0<r<R$, write the discrete circle $S_{r}=\mathcal{C}^{1}\left[\left(B_{R}\right)_{\delta}\right]\cap\left(B_{r+4\delta}\setminus B_{r}\right)$
\[
\left|F\left(-\frac{\delta}{2}\right)\right|^{2}\leq\frac{cst}{r}\sum_{c\in S_{r}}\left|F\left(c\right)\right|^{2}\delta,
\]
multiplying by $\delta$ and summing over the $O(\delta)$ discrete
circles $S_{r}$ such that their union equals $B_{R}\setminus B_{R/2}$
\begin{equation}
\left|F\left(-\frac{\delta}{2}\right)\right|^{2}\leq\frac{cst}{R}\sum_{c\in B_{R}\setminus B_{R/2}}\left|F\left(-\frac{\delta}{2}\right)\right|^{2}\delta^{2}\leq\frac{cst\cdot L}{R}.\label{eq:l2-harm}
\end{equation}

For the desired bounds, note that by first applying (\ref{eq:l2-harm})
to smaller balls of radii $R/2$ we can opt for a bound of the form
\begin{alignat}{1}
\left|F\left(-\frac{\delta}{2}\right)\right| & \leq cst\cdot e^{cmR}\max_{B_{R/2}}\left|F\right|,\label{eq:maxharnack}\\
\delta^{-1}\left|F\left(\frac{\delta}{2}+i\delta\right)-F\left(-\frac{\delta}{2}\right)\right| & \leq cst\cdot e^{cmR}\frac{\max_{B_{R/2}}\left|F\right|}{R/2}.\nonumber 
\end{alignat}

Consider the first estimate. By the maximum and minimum principles,
we may bound
\[
-\max_{B_{R/2}}\left|F\right|\text{hm}_{B_{R/2}}^{S_{R/2}}\leq F\leq\max_{B_{R/2}}\left|F\right|\text{hm}_{B_{R/2}}^{S_{R/2}}.
\]
The massive harmonic measure $\text{hm}_{B_{R/2}}^{S_{R/2}}(c)$ is
the hitting probability of $S_{R/2}$ of the massive random walk started
at $c$. For the bound at $-\frac{\delta}{2}$, simply note that the
probability of a massive random walk reaching a box at distance $d$
decays exponentially fast in $\left|m\right|d$. (see e.g. the projection
argument in the proof of Proposition \ref{prop:squareroot}).

For the second, by decomposing $F(c)=\sum_{c'\in S_{R/2}}\text{hm}_{B_{R/2}}^{\left\{ c'\right\} }(c)F(c')$
for $c\in B_{R/2}$, with $\text{hm}_{B_{R/2}}^{\left\{ c'\right\} }=\text{hm}_{B_{R/2}}^{\left\{ c'\right\} }\left(\cdot|m\right)$
being the massive harmonic function on $B_{R/2}$ whose boundary value
is $0$ on $S_{R/2}\setminus\{c'\}$ and $1$ at $c'\in S_{R/2}$,
it suffices to show $\left|\text{hm}_{B_{R/2}}^{\left\{ c'\right\} }(\frac{\delta}{2}+i\delta)-\text{hm}_{B_{R/2}}^{\left\{ c'\right\} }(-\frac{\delta}{2})\right|\leq cst\cdot\frac{\delta^{2}e^{cmR}}{R^{2}}$.
We know that the hitting probability for the simple random walk (i.e.
the harmonic measure of the point $c'$, with $m=0$) satisfies the
desired estimate (e.g. \cite[Proposition 2.7]{chsm2011}): as $\delta\to0$,
the difference of the probabilities $P_{1},P_{2}$ of simple random
walk started at neighbouring points near $0$ reaching $c'$ before
other points in $S_{R/2}$ is bounded above by $cst\cdot P_{1}\cdot\frac{\delta}{R}\leq cst\cdot\frac{\delta^{2}}{R^{2}}$.
For the massive random walk, these instances (say, coupled with the
same exponential clock) need to survive to contribute to the difference;
therefore the difference decays by an additional exponential factor.
\end{proof}
\begin{rem}
\label{rem:smoothness}The second bound in (\ref{eq:harnack}) is
valid for differences in other directions as well, since massive harmonicity
and the bound are rotationally invariant. By considering smaller balls
within $B_{R}$, we in fact deduce uniform bounds for $F$ and its
discrete derivative in, say, $B_{R/2}$. Then, defining $D_{\lambda}^{\delta}F(c):=F(c+\delta+i\delta)-F(c),D_{\bar{\lambda}}^{\delta}F(c):=F(c+\delta-i\delta)-F(c)$,
which are massive harmonic functions uniformly bounded in $B_{R/2}$,
and using the bound (\ref{eq:maxharnack}) on them, we actually have
bound on discrete derivatives of second order in, say, $B_{R/4}$.
Recursively, we see that derivatives of any order can be locally bounded.
\end{rem}

\end{document}